\documentclass[a4paper, 12pt]{amsart}
\usepackage[a4paper, margin=3cm]{geometry}

\pdfoutput=1
\usepackage{adjustbox}

\usepackage[utf8]{inputenc}
\usepackage[T1]{fontenc}
\usepackage[tracking=true]{microtype}
\SetTracking%
  {encoding=T1, shape=sc}
  {10}
\SetProtrusion%
  {encoding=T1,size={7,8}}
  {1={ ,750},2={ ,500},3={ ,500},4={ ,500},5={ ,500},
    6={ ,500},7={ ,600},8={ ,500},9={ ,500},0={ ,500}}

\usepackage{mathtools}
\usepackage{amssymb}
\usepackage{scalerel}

\usepackage[ttscale=.85]{libertine}
\usepackage{libertinust1math}
\usepackage[cal=boondoxo, calscaled=0.96, frak=euler]{mathalfa}

\usepackage[english]{babel}
\usepackage{csquotes}

\usepackage{pgfplots}
\usepackage{tikz}
\usetikzlibrary{
  backgrounds,
  calc,
  cd,
  fit,
  scopes,
  intersections,
  pgfplots.fillbetween,
  arrows.meta,
  decorations.text,
  decorations.markings,
  decorations.pathmorphing
}

\tikzset{curve/.style={settings={#1},to path={(\tikztostart)
    .. controls ($(\tikztostart)!\pv{pos}!(\tikztotarget)!\pv{height}!270:(\tikztotarget)$)
    and ($(\tikztostart)!1-\pv{pos}!(\tikztotarget)!\pv{height}!270:(\tikztotarget)$)
    .. (\tikztotarget)\tikztonodes}},
    settings/.code={\tikzset{quiver/.cd,#1}
        \def\pv##1{\pgfkeysvalueof{/tikz/quiver/##1}}},
    quiver/.cd,pos/.initial=0.35,height/.initial=0}
\tikzset{tail reversed/.code={\pgfsetarrowsstart{tikzcd to}}}
\tikzset{2tail/.code={\pgfsetarrowsstart{Implies[reversed]}}}
\tikzset{2tail reversed/.code={\pgfsetarrowsstart{Implies}}}
\tikzset{no body/.style={/tikz/dash pattern=on 0 off 1mm}}

\usepackage{tikz}
\usetikzlibrary{backgrounds}
\usetikzlibrary{arrows}
\usetikzlibrary{shapes,shapes.geometric,shapes.misc}
\tikzstyle{tikzfig}=[baseline=-0.25em,scale=0.5]
\pgfkeys{/tikz/tikzit fill/.initial=0}
\pgfkeys{/tikz/tikzit draw/.initial=0}
\pgfkeys{/tikz/tikzit shape/.initial=0}
\pgfkeys{/tikz/tikzit category/.initial=0}
\pgfdeclarelayer{edgelayer}
\pgfdeclarelayer{nodelayer}
\pgfsetlayers{background,edgelayer,nodelayer,main}
\tikzstyle{none}=[inner sep=0mm]
\newcommand{\tikzfig}[1]{%
  {%
    \tikzstyle{every picture}=[tikzfig]
    \IfFileExists{#1.tikz}
    {\input{#1.tikz}}
    {%
      \IfFileExists{#1.tikz}
      {\input{#1.tikz}}
      {\tikz[baseline=-0.5em]{\node[draw=red,font=\color{red},fill=red!10!white] {\textit{#1}};}}%
    }%
  }%
}
\tikzstyle{every loop}=[]
\tikzstyle{morphism} = [rectangle, fill=white, draw=black, line width=1pt, font=\scriptsize]
\tikzstyle{intersection-pt} = [fill=white, inner sep = 2.5pt]
\tikzstyle{fullblackdot}=[fill=black, draw, shape=circle, scale=0.7]
\tikzstyle{blackdot}=[fill=white, draw, shape=circle, scale=0.3]
\tikzstyle{comodule-edge}=[-, draw=qyd-color,very thick]
\tikzstyle{ddd}=[-, draw=black, dash dot dot, very thick]
\tikzstyle{unit}=[-, draw=black, very thick, densely dotted]
\tikzstyle{morphism-edge} = [very thick, black]
\tikzstyle{directed} = [morphism-edge, postaction=decorate]
\tikzstyle{Front}=[-, draw=black, fill=white, opacity=0.5]
\tikzstyle{Hidden}=[-, draw=black, fill=white, opacity=0.5]

\definecolor{blue}{rgb}{0.38, 0.51, 0.71}
\definecolor{red}{RGB}{175, 49, 39}
\definecolor{green}{RGB}{146, 227, 95}
\definecolor{qyd-color}{rgb}{0.38, 0.51, 0.71}

\usepackage{xspace}
\usepackage{enumitem}

\usepackage[colorlinks, linktoc=page, linkcolor={red}, citecolor={green!65!black}, urlcolor={blue}]{hyperref}
\usepackage{amsthm}
\usepackage[capitalise, noabbrev]{cleveref}

\Crefname{diagram}{Diagram}{Diagrams}
\Crefformat{diagram}{Diagram~(#2#1#3)}

\numberwithin{equation}{section}

\theoremstyle{plain}
\newtheorem{theorem}{Theorem}
\newtheorem{lemma}[theorem]{Lemma}

\newtheorem{proposition}[theorem]{Proposition}
\newtheorem*{proposition*}{Proposition}

\theoremstyle{definition}
\newtheorem*{theorem*}{Theorem}
\newtheorem{definition}[theorem]{Definition}
\newtheorem*{definition*}{Definition}
\newtheorem{example}[theorem]{Example}

\newtheorem{remark}[theorem]{Remark}

\newtheorem*{question*}{Question}
\newtheorem*{remark*}{Remark}

\newcommand*{\eg}{e.g.}
\newcommand*{\ie}{i.e.}

\newcommand*{\EM}{Eilenberg--Moore\xspace}

\renewcommand{\colon}{\!\nobreak\mskip2mu\mathpunct{}\nonscript%
  \mkern-\thinmuskip{:}\mskip6muplus1mu\relax%
}
\newcommand{\from}{\colon}
\newcommand*{\nt}{\Longrightarrow}
\renewcommand*{\to}{\longrightarrow}
\renewcommand*{\mapsto}{\longmapsto}

\DeclareMathOperator*{\Ob}{Ob}
\newcommand*{\cat}[1]{\ensuremath{\mathcal{#1}}}

\makeatletter
\def\ST@firstletter#1#2@{#1}
\def\ST@restletters#1#2@{#2}
\newcommand*{\BiCat}[1]{%
  \ensuremath{%
    \mathbb{\ST@firstletter #1@}\mathsf{\ST@restletters #1@}%
  }\xspace%
}
\makeatother

\makeatletter
\newcommand*{\@smallerTriangle}[2]{\raisebox{\depth+0.025em}{\scalebox{0.75}{\ensuremath{#1#2}}}}
\newcommand*{\@smallTriangle}[2]{\raisebox{\depth+0.025em}{\scalebox{0.90}{\ensuremath{#1#2}}}}
\newcommand*{\ract}{\mathbin{\mathpalette\@smallerTriangle\lhd}}
\newcommand*{\bract}{\mathbin{\mathpalette\@smallTriangle\blacktriangleleft}}
\makeatother

\newcommand*{\id}{\mathrm{id}}
\newcommand*{\Id}{\mathrm{Id}}

\makeatletter
\newcommand{\ST@colon}{:\!}
\newcommand*{\cocolon}{%
  \nobreak%
  \mskip6mu plus1mu
  \mathpunct{}%
  \nonscript%
  \mkern-\thinmuskip%
  {\ST@colon}%
  \mskip2mu
  \relax
}
\newcommand*{\adj}[4]{\ensuremath{#1 \colon #3 \rightleftarrows #4 \cocolon #2}\xspace}
\makeatother
\newcommand*{\stdadj}{\adj{F}{U}{\cat{C}}{\cat{D}}}
\newcommand*{\adjspace}{\,}
\newcommand*{\adjoint}{\adjspace\dashv\adjspace}
\newcommand*{\lad}{\adjoint}

\newcommand*{\blank}{{-}}
\newcommand*{\bblank}{{=}}
\newcommand*{\bbblank}{{\equiv}}
\renewcommand*{\k}{\ensuremath{\mathtt{k}}\xspace}

\DeclareFontFamily{U}{DSSerif}{\skewchar \font =45}% openface
\DeclareFontShape{U}{DSSerif}{m}{n}{<-> s*[1]  DSSerif}{}
\DeclareFontSubstitution{U}{DSSerif}{m}{n}
\DeclareMathAlphabet{\mathbbbb}{U}{DSSerif}{m}{n}
\newcommand*{\1}{\text{\usefont{U}{DSSerif}{m}{n}1}}

\DeclareMathOperator{\Com}{Com}

\DeclareMathOperator{\Inc}{Inc}
\DeclareMathOperator{\Alg}{Alg}

\usepackage[ backend=bibtex, maxnames=4, maxalphanames=4, style=alphabetic, citestyle=alphabetic]{biblatex}

\DeclareRobustCommand{\SkipTocEntry}[5]{}

\author{Sebastian Halbig}
\address{S.H., Philipps-Universität Marburg, Arbeitsgruppe Algebraische Lie-Theorie, Hans-Meer\-wein-Straße 6, 35043 Marburg, Germany}
\email{sebastian.halbig@uni-marburg.de}

\author{Tony Zorman}
\address{T.Z., TU Dresden, Institut für Geometrie, Zellescher Weg 12--14, 01062 Dresden, Germany}
\email{tony.zorman@tu-dresden.de}

\date{\today}
\title{Diagrammatics for Comodule Monads}

\bibliography{main}

\begin{document}

\maketitle

\begin{abstract}\vspace{-0.4cm}
  We extend Willerton's~\cite{Willerton2008} graphical calculus for bimonads to comodule monads,
  a monadic interpretation of module categories over a monoidal category.
  As an application, we prove a version of Tannaka--Krein duality for these structures.
\end{abstract}

\microtypesetup{protrusion=false}
\tableofcontents
\microtypesetup{protrusion=true}

\section{Introduction}

\noindent Given a monad \(B\) on a monoidal category \(\cat{C}\),
one might ask to which extent monoidal structures on the category \(\cat{C}^B\) of \(B\)-algebras
are controlled by additional structures on the monad itself.
This can be seen as an extension of the classical \emph{Tannaka--Krein duality},
and was proved by Moerdijk~\cite[Theorem~7.1]{Moerdijk2002} and McCrudden~\cite[Corollary~3.13]{McCrudden2002}.

\begin{theorem}\label{thm:intr:bimonad-reconstruction}
  Let \(B\) be a monad on a monoidal category \(\cat{C}\).
  Bimonad structures on \(B\) are
  in one-to-one correspondence with
  monoidal structures on \(\cat{C}^B\), such that the forgetful functor \(U^B\) is strict monoidal.
\end{theorem}

In~\cite[Section~2.3]{Willerton2008},
\cref{thm:intr:bimonad-reconstruction} is proved in a graphical fashion.
The aim of this article is to generalise both the statement
and its graphical proof to comodule monads,
as developed by Aguiar and Chase~\cite{Aguiar2012}.
More precisely, we show the following.

\begin{theorem}\label{thm:intr:comodule-monad-reconstruction}
  Let \(B\) be a bimonad on the monoidal category \(\cat{C}\),
  and \(K\) a monad on a right \(\cat{C}\)-module category \(\cat{M}\).
  Coactions of \(B\) on \(K\) are
  in bijection with
  right actions of\, \(\cat{C}^B\) on \(\cat{M}^K\), such that \(U^{K}\) is a strict comodule functor over \(U^{B}\).
\end{theorem}

The article is structured as follows.
\Cref{sec:adjunctions-and-monads} serves to review basic concepts from category theory,
as well as their associated string diagrammatic counterparts.
\Cref{sec:bimonads} introduces Willerton's graphical calculus for monoidal categories,
and \cref{sec:comodule-monads} extends this construction to module categories and comodule monads,
allowing us to prove \cref{thm:intr:comodule-monad-reconstruction}.

\addtocontents{toc}{\SkipTocEntry}
\subsection*{Acknowledgements}

The authors would like to thank Simon Willerton,
as well as the two anonymous referees,
for many insightful comments and suggestions.
T.Z.~is supported by \textsc{dfg} grant \textsc{kr} \oldstylenums{5036/2--1}.

% Reset the theorem counter on every \section.
\makeatletter
\@addtoreset{theorem}{section}
\makeatother

\section{2-categories and the graphical calculus}\label{sec:adjunctions-and-monads}
\numberwithin{theorem}{section}

\noindent We assume the reader's familiarity with basic concepts of the theory of monoidal categories as discussed for example in~\cite{MacLane1998,Etingof2015,Riehl2017}.

Our study of comodule structures on monads involves---besides the composition of functors and natural transformations---products of categories.
To that end we let \(\BiCat{X}\) be the monoidal 2-category of (small) categories, functors, and natural transformations,
equipped with the monoidal structure induced by the Cartesian product of categories.\footnote{%
  More generally, our applications of the graphical calculus can be formulated in any (weakly) monoidal 2-category, see~\cite[Chatper~12]{johnson-yau2021:TwoDimensionalCategories}, which \emph{admits the construction of algebras}, see~\cite[\S 1]{Street1972}.
}
We use juxtaposition for the horizontal composition of\, $\mathbb{X}$, or, in case we want to emphasise the direction of composition,
\begin{equation*}
  \blank \odot_{\mathcal{A},\mathcal{B},\mathcal{C}} \blank \from \mathbb{X}(\mathcal{B},\mathcal{C}) \times \mathbb{X}(\mathcal{A},\mathcal{B}) \to \mathbb{X}(\mathcal{A},\mathcal{C}), \qquad \qquad \mathcal{A},\mathcal{B},\mathcal{C} \in \Ob(\mathbb{X}).
\end{equation*}

String diagrams will serve as the main tool for doing computations.
We closely follow the presentation and conventions in~\cite{halbig22:pivot-hopf}.
In the case of a 2-category \(\mathbb{X}\),
a \emph{string diagram} comprises regions, labelled with categories, strings labelled with functors,
and vertices between the strings labelled with natural transformations.
If two string diagrams can be transformed into each other, the natural transformations they represent are equal.
A more detailed description is given in~\cite{joyal91,Selinger2011}.
Our convention is to read diagrams from top to bottom and left to right.
Horizontal and vertical composition are given by horizontal and vertical gluing of  diagrams, respectively.
Identity natural transformations are given by unlabelled vertices.
The identity functor of a category is represented by the empty edge.
If the involved categories are clear from the context, we omit writing them explicitly.

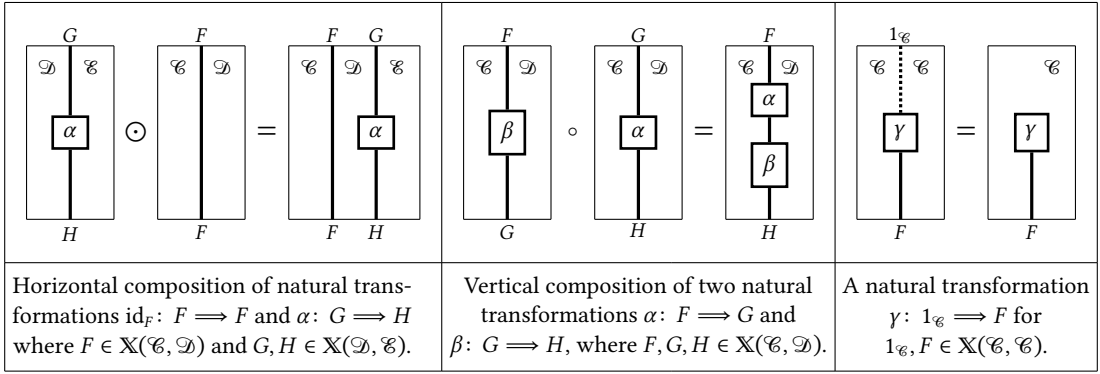
\begin{center}
  \begin{tikzpicture}[xscale=1.15, yscale=1.15]
	\begin{pgfonlayer}{nodelayer}
		\node [style=none] (0) at (-7.25, 3.25) {};
		\node [style=none] (1) at (-7.25, -0.75) {};
		\node [style=none] (5) at (-8.75, 1.25) {$\odot$};
		\node [style=none] (2) at (-10.25, 3.25) {};
		\node [style=morphism] (3) at (-10.25, 1.25) {$\alpha$};
		\node [style=none] (4) at (-10.25, -0.75) {};
		\node [style=none] (6) at (-5.75, 1.25) {$=$};
		\node [style=none] (7) at (-4.25, 3.25) {};
		\node [style=none] (8) at (-4.25, -0.75) {};
		\node [style=none] (9) at (-3.25, 3.25) {};
		\node [style=morphism] (10) at (-3.25, 1.25) {$\alpha$};
		\node [style=none] (11) at (-3.25, -0.75) {};
		\node [style=none] (12) at (-0.25, 3.25) {};
		\node [style=morphism] (13) at (-0.25, 1.25) {$\beta$};
		\node [style=none] (14) at (-0.25, -0.75) {};
		\node [style=none] (18) at (1.25, 1.25) {$\circ$};
		\node [style=none] (15) at (2.75, 3.25) {};
		\node [style=morphism] (16) at (2.75, 1.25) {$\alpha$};
		\node [style=none] (17) at (2.75, -0.75) {};
		\node [style=none] (19) at (4.25, 1.25) {$=$};
		\node [style=none] (20) at (5.75, 3.25) {};
		\node [style=morphism] (21) at (5.75, 2) {$\alpha$};
		\node [style=morphism] (22) at (5.75, 0.5) {$\beta$};
		\node [style=none] (23) at (5.75, -0.75) {};
		\node [style=none] (24) at (8.75, 3.25) {};
		\node [style=morphism] (25) at (8.75, 1.25) {$\gamma$};
		\node [style=none] (26) at (8.75, -0.75) {};
		\node [style=none] (27) at (10.25, 1.25) {$=$};
		\node [style=morphism] (28) at (11.75, 1.25) {$\gamma$};
		\node [style=none] (29) at (11.75, -0.75) {};
		\node [style=none] (30) at (-7.25, -1.05) {\tiny $F$};
		\node [style=none] (31) at (-7.25, 3.5) {\tiny $F$};
		\node [style=none] (32) at (-10.25, -1.1) {\tiny $H$};
		\node [style=none] (33) at (-10.25, 3.5) {\tiny $G$};
		\node [style=none] (34) at (-4.25, 3.5) {\tiny $F$};
		\node [style=none] (35) at (-3.25, 3.5) {\tiny $G$};
		\node [style=none] (36) at (-4.25, -1.05) {\tiny $F$};
		\node [style=none] (37) at (-3.25, -1.05) {\tiny $H$};
		\node [style=none] (38) at (-11.75, 4.25) {};
		\node [style=none] (39) at (-1.75, 4.25) {};
		\node [style=none] (40) at (-1.75, -1.75) {};
		\node [style=none] (41) at (-11.75, -1.75) {};
		\node [style=none] (42) at (-11.75, -4.25) {};
		\node [style=none] (43) at (-1.75, -4.25) {};
		\node [style=none] (44) at (-1.75, -1.75) {};
		\node [style=none, align=center] (45) at (-7, -3) {\scriptsize Horizontal composition of natural trans-\\[-0.2\baselineskip]\scriptsize formations $\id_F \from F \nt F$ and $\alpha \from G \nt H$ \\[-0.2\baselineskip]\scriptsize where $F \in \mathbb{X}(\cat{C},\cat{D})$ and $G, H \in \mathbb{X}(\cat{D},\cat{E})$.};
		\node [style=none] (47) at (3.5, -4.25) {};
		\node [style=none] (49) at (3.5, -1.75) {};
		\node [style=none] (50) at (7.25, -1.75) {};
		\node [style=none] (51) at (7.25, 4.25) {};
		\node [style=none] (52) at (7.25, -4.25) {};
		\node [style=none, align=center] (53) at (2.75, -3) {\scriptsize Vertical composition of two natural\\[-0.2\baselineskip]\scriptsize transformations $\alpha\from  F\nt G$ and\\[-0.2\baselineskip]\scriptsize $\beta\from G \nt H$, where $F,G,H \in \mathbb{X}(\cat{C}, \cat{D})$. };
		\node [style=none] (54) at (-0.25, 3.5) {\tiny $F$};
		\node [style=none] (55) at (-0.25, -1.05) {\tiny $G$};
		\node [style=none] (56) at (2.75, 3.5) {\tiny $G$};
		\node [style=none] (57) at (2.75, -1) {\tiny $H$};
		\node [style=none] (58) at (5.75, 3.5) {\tiny $F$};
		\node [style=none] (59) at (5.75, -1.05) {\tiny $H$};
		\node [style=none] (60) at (13.25, 4.25) {};
		\node [style=none] (61) at (13.25, -1.75) {};
		\node [style=none] (62) at (13.25, -4.25) {};
		\node [style=none, align=center] (63) at (10.25, -3) {\scriptsize A natural transformation \\[-0.2\baselineskip]\scriptsize$\gamma \colon 1_\cat{C} \nt F$ for \\[-0.2\baselineskip]\scriptsize $1_\cat{C}, F \in \mathbb{X}(\cat{C},\cat{C})$.};
		\node [style=none] (64) at (8.75, 3.5) {\tiny $1_{\cat C}$};
		\node [style=none] (65) at (8.75, -1.05) {\tiny $F$};
		\node [style=none] (66) at (11.75, -1.05) {\tiny $F$};
		\node [style=none] (67) at (-8.25, 3.25) {};
		\node [style=none] (68) at (-6.25, 3.25) {};
		\node [style=none] (69) at (-8.25, -0.75) {};
		\node [style=none] (70) at (-6.25, -0.75) {};
		\node [style=none] (71) at (-7.75, 2.75) {${}_\cat{C}$};
		\node [style=none] (72) at (-6.75, 2.75) {${}_\cat{D}$};
		\node [style=none] (73) at (-11.25, 3.25) {};
		\node [style=none] (74) at (-9.25, 3.25) {};
		\node [style=none] (75) at (-11.25, -0.75) {};
		\node [style=none] (76) at (-9.25, -0.75) {};
		\node [style=none] (77) at (-10.75, 2.75) {${}_\cat{D}$};
		\node [style=none] (78) at (-9.75, 2.75) {${}_\cat{E}$};
		\node [style=none] (79) at (-5.25, 3.25) {};
		\node [style=none] (80) at (-2.25, 3.25) {};
		\node [style=none] (81) at (-5.25, -0.75) {};
		\node [style=none] (82) at (-2.25, -0.75) {};
		\node [style=none] (83) at (-4.75, 2.75) {${}_\cat{C}$};
		\node [style=none] (84) at (-2.75, 2.75) {${}_\cat{E}$};
		\node [style=none] (85) at (-3.75, 2.75) {${}_\cat{D}$};
		\node [style=none] (86) at (-1.25, 3.25) {};
		\node [style=none] (87) at (0.75, 3.25) {};
		\node [style=none] (88) at (-1.25, -0.75) {};
		\node [style=none] (89) at (0.75, -0.75) {};
		\node [style=none] (90) at (-0.75, 2.75) {${}_\cat{C}$};
		\node [style=none] (91) at (0.25, 2.75) {${}_\cat{D}$};
		\node [style=none] (92) at (1.75, 3.25) {};
		\node [style=none] (93) at (3.75, 3.25) {};
		\node [style=none] (94) at (1.75, -0.75) {};
		\node [style=none] (95) at (3.75, -0.75) {};
		\node [style=none] (96) at (2.25, 2.75) {${}_\cat{C}$};
		\node [style=none] (97) at (3.25, 2.75) {${}_\cat{D}$};
		\node [style=none] (98) at (4.75, 3.25) {};
		\node [style=none] (99) at (6.75, 3.25) {};
		\node [style=none] (100) at (4.75, -0.75) {};
		\node [style=none] (101) at (6.75, -0.75) {};
		\node [style=none] (102) at (5.25, 2.75) {${}_\cat{C}$};
		\node [style=none] (103) at (6.25, 2.75) {${}_\cat{D}$};
		\node [style=none] (104) at (7.75, 3.25) {};
		\node [style=none] (105) at (9.75, 3.25) {};
		\node [style=none] (106) at (7.75, -0.75) {};
		\node [style=none] (107) at (9.75, -0.75) {};
		\node [style=none] (108) at (8.25, 2.75) {${}_\cat{C}$};
		\node [style=none] (109) at (9.25, 2.75) {${}_\cat{C}$};
		\node [style=none] (110) at (10.75, 3.25) {};
		\node [style=none] (111) at (12.75, 3.25) {};
		\node [style=none] (112) at (10.75, -0.75) {};
		\node [style=none] (113) at (12.75, -0.75) {};
		\node [style=none] (115) at (12.25, 2.75) {${}_\cat{C}$};
	\end{pgfonlayer}
	\begin{pgfonlayer}{edgelayer}
		\draw [style=morphism-edge] (0.center) to (1.center);
		\draw [style=morphism-edge] (2.center) to (3);
		\draw [style=morphism-edge] (3) to (4.center);
		\draw [style=morphism-edge] (7.center) to (8.center);
		\draw [style=morphism-edge] (9.center) to (10);
		\draw [style=morphism-edge] (10) to (11.center);
		\draw [style=morphism-edge] (12.center) to (13);
		\draw [style=morphism-edge] (13) to (14.center);
		\draw [style=morphism-edge] (15.center) to (16);
		\draw [style=morphism-edge] (16) to (17.center);
		\draw [style=morphism-edge] (20.center) to (21);
		\draw [style=morphism-edge] (21) to (22);
		\draw [style=morphism-edge] (22) to (23.center);
		\draw [unit] (24.center) to (25);
		\draw [style=morphism-edge] (25) to (26.center);
		\draw [style=morphism-edge] (28) to (29.center);
		\draw (41.center) to (38.center);
		\draw (38.center) to (39.center);
		\draw (39.center) to (40.center);
		\draw (40.center) to (41.center);
		\draw (43.center) to (44.center);
		\draw (41.center) to (42.center);
		\draw (42.center) to (43.center);
		\draw (49.center) to (44.center);
		\draw (43.center) to (47.center);
		\draw (52.center) to (50.center);
		\draw (50.center) to (49.center);
		\draw (47.center) to (52.center);
		\draw (39.center) to (51.center);
		\draw [in=90, out=-90, looseness=0.75] (51.center) to (50.center);
		\draw (62.center) to (52.center);
		\draw (50.center) to (61.center);
		\draw (61.center) to (62.center);
		\draw (61.center) to (60.center);
		\draw (60.center) to (51.center);
		\draw (67.center) to (69.center);
		\draw (69.center) to (70.center);
		\draw (70.center) to (68.center);
		\draw (68.center) to (67.center);
		\draw (73.center) to (75.center);
		\draw (75.center) to (76.center);
		\draw (76.center) to (74.center);
		\draw (74.center) to (73.center);
		\draw (79.center) to (81.center);
		\draw (81.center) to (82.center);
		\draw (82.center) to (80.center);
		\draw (80.center) to (79.center);
		\draw (86.center) to (88.center);
		\draw (88.center) to (89.center);
		\draw (89.center) to (87.center);
		\draw (87.center) to (86.center);
		\draw (92.center) to (94.center);
		\draw (94.center) to (95.center);
		\draw (95.center) to (93.center);
		\draw (93.center) to (92.center);
		\draw (98.center) to (100.center);
		\draw (100.center) to (101.center);
		\draw (101.center) to (99.center);
		\draw (99.center) to (98.center);
		\draw (104.center) to (106.center);
		\draw (106.center) to (107.center);
		\draw (107.center) to (105.center);
		\draw (105.center) to (104.center);
		\draw (110.center) to (112.center);
		\draw (112.center) to (113.center);
		\draw (113.center) to (111.center);
		\draw (111.center) to (110.center);
	\end{pgfonlayer}
\end{tikzpicture}
\end{center}

Recall from \eg~\cite{benabou67:introd} that a \emph{monad} on a category \(\cat{C}\in \mathbb{X}\) is
a functor \(T \from \cat{C} \to \cat{C}\)
together with two natural transformations \(\mu \from T^2 \nt T\) and \(u \from \Id_{\cat{C}} \nt T\),
called the \emph{multiplication} and \emph{unit} of \(T\),
satisfying appropriate \emph{associativity} and \emph{unitality} axioms:
\[
  \begin{tikzcd}
    {T^3} & { T^2} & {T^2} & T & T \\
    {T^2} & T && T
    \arrow["T\mu", Rightarrow, from=1-1, to=1-2]
    \arrow["\mu", Rightarrow, from=1-2, to=2-2]
    \arrow["{\mu T}"', Rightarrow, from=1-1, to=2-1]
    \arrow["\mu"', Rightarrow, from=2-1, to=2-2]
    \arrow[Rightarrow, no head, from=1-4, to=2-4]
    \arrow["{u T}"', Rightarrow, from=1-4, to=1-3]
    \arrow["\mu"', Rightarrow, from=1-3, to=2-4]
    \arrow["{T u}", Rightarrow, from=1-4, to=1-5]
    \arrow["\mu", Rightarrow, from=1-5, to=2-4]
  \end{tikzcd}
\]

We represent the multiplication and unit of a monad \(T\)
in terms of string diagrams:
\begin{equation}\label[diagram]{eq:mult-and-unit-of-monad}
  \begin{tikzpicture}
	\begin{pgfonlayer}{nodelayer}
		\node [style=none] (35) at (2.75, 1) {};
		\node [style=none] (36) at (0.75, 1) {};
		\node [style=none] (37) at (1.75, 0) {};
		\node [style=none] (39) at (1.75, -1) {};
		\node [style=blackdot] (48) at (9, 0) {};
		\node [style=none] (49) at (9, -1) {};
		\node [style=none] (52) at (0, -1) {};
		\node [style=none] (55) at (10, -1) {};
		\node [style=none] (61) at (1.75, -1.3) {\tiny $T$};
		\node [style=none] (62) at (0.75, 1.25) {\tiny $T$};
		\node [style=none] (63) at (2.75, 1.25) {\tiny $T$};
		\node [style=none] (64) at (9, -1.3) {\tiny $T$};
		\node [style=none] (65) at (0, 1) {};
		\node [style=none] (67) at (10, 1) {};
		\node [style=none] (68) at (3.5, 1) {};
		\node [style=none] (69) at (3.5, -1) {};
		\node [style=none] (70) at (8, 1) {};
		\node [style=none] (71) at (8, -1) {};
		\node [style=none] (72) at (5.75, 0) {and};
	\end{pgfonlayer}
	\begin{pgfonlayer}{edgelayer}
		\draw [style=morphism-edge] (36.center)
			 to [in=180, out=-90] (37.center)
			 to [in=-90, out=0] (35.center);
		\draw [style=morphism-edge] (37.center) to (39.center);
		\draw [style=morphism-edge] (48) to (49.center);
		\draw (67.center) to (55.center);
		\draw (65.center) to (52.center);
		\draw (65.center) to (68.center);
		\draw (68.center) to (69.center);
		\draw (69.center) to (52.center);
		\draw (67.center) to (70.center);
		\draw (70.center) to (71.center);
		\draw (71.center) to (55.center);
	\end{pgfonlayer}
\end{tikzpicture}
\end{equation}
Their associativity and unitality then equate to
\begin{equation}\label[diagram]{eq:assoc-and-unit-of-monad-string-diagram}
  \begin{tikzpicture}
	\begin{pgfonlayer}{nodelayer}
		\node [style=none] (4) at (7.75, -1.25) {};
		\node [style=none] (5) at (7.75, 1.25) {};
		\node [style=none] (10) at (-4, 0) {\scriptsize $=$};
		\node [style=none] (11) at (-6, -0.25) {};
		\node [style=none] (12) at (-6, -1.25) {};
		\node [style=none] (13) at (-4.75, 1.25) {};
		\node [style=none] (14) at (-5.75, 1.25) {};
		\node [style=none] (15) at (-5.25, 0.5) {};
		\node [style=none] (16) at (-6.75, 1.25) {};
		\node [style=none] (17) at (-6.75, 0.5) {};
		\node [style=none] (22) at (-2, -0.25) {};
		\node [style=none] (23) at (-2, -1.25) {};
		\node [style=none] (24) at (-3.25, 1.25) {};
		\node [style=none] (25) at (-2.25, 1.25) {};
		\node [style=none] (26) at (-2.75, 0.5) {};
		\node [style=none] (27) at (-1.25, 1.25) {};
		\node [style=none] (28) at (-1.25, 0.5) {};
		\node [style=none] (39) at (6.75, 0) {\scriptsize $=$};
		\node [style=none] (40) at (8.75, 0) {\scriptsize $=$};
		\node [style=none] (51) at (-7.25, 1.25) {};
		\node [style=none] (52) at (11.25, 1.25) {};
		\node [style=none] (58) at (-7.25, -1.25) {};
		\node [style=none] (59) at (11.25, -1.25) {};
		\node [style=none] (65) at (10.75, 1.25) {};
		\node [style=blackdot] (66) at (9.75, 0.5) {};
		\node [style=none] (67) at (10.25, 0) {};
		\node [style=none] (68) at (10.25, -1.25) {};
		\node [style=none] (69) at (4.75, 1.25) {};
		\node [style=blackdot] (70) at (5.75, 0.5) {};
		\node [style=none] (71) at (5.25, 0) {};
		\node [style=none] (72) at (5.25, -1.25) {};
		\node [style=none] (73) at (-0.75, 1.25) {};
		\node [style=none] (74) at (-0.75, -1.25) {};
		\node [style=none] (75) at (4.25, 1.25) {};
		\node [style=none] (76) at (4.25, -1.25) {};
		\node [style=none] (77) at (-6.75, 1.5) {\tiny $T$};
		\node [style=none] (78) at (-5.75, 1.5) {\tiny $T$};
		\node [style=none] (79) at (-4.75, 1.5) {\tiny $T$};
		\node [style=none] (80) at (-3.25, 1.5) {\tiny $T$};
		\node [style=none] (81) at (-2.25, 1.5) {\tiny $T$};
		\node [style=none] (82) at (-1.25, 1.5) {\tiny $T$};
		\node [style=none] (83) at (-6, -1.55) {\tiny $T$};
		\node [style=none] (84) at (-2, -1.55) {\tiny $T$};
		\node [style=none] (85) at (1.75, 0) {and};
		\node [style=none] (86) at (4.75, 1.5) {\tiny $T$};
		\node [style=none] (87) at (10.75, 1.5) {\tiny $T$};
		\node [style=none] (88) at (10.25, -1.55) {\tiny $T$};
		\node [style=none] (89) at (5.25, -1.55) {\tiny $T$};
		\node [style=none] (90) at (7.75, 1.5) {\tiny $T$};
		\node [style=none] (91) at (7.75, -1.55) {\tiny $T$};
	\end{pgfonlayer}
	\begin{pgfonlayer}{edgelayer}
		\draw [style=morphism-edge, in=90, out=-90] (5.center) to (4.center);
		\draw [style=morphism-edge] (11.center) to (12.center);
		\draw [style=morphism-edge, in=0, out=-90] (13.center) to (15.center);
		\draw [style=morphism-edge, in=180, out=-90] (14.center) to (15.center);
		\draw [style=morphism-edge] (16.center) to (17.center);
		\draw [style=morphism-edge, in=180, out=-90] (17.center) to (11.center);
		\draw [style=morphism-edge, in=0, out=-90] (15.center) to (11.center);
		\draw [style=morphism-edge] (22.center) to (23.center);
		\draw [style=morphism-edge, in=180, out=-90] (24.center) to (26.center);
		\draw [style=morphism-edge, in=0, out=-90] (25.center) to (26.center);
		\draw [style=morphism-edge] (27.center) to (28.center);
		\draw [style=morphism-edge, in=0, out=-90] (28.center) to (22.center);
		\draw [style=morphism-edge, in=180, out=-90] (26.center) to (22.center);
		\draw (58.center) to (51.center);
		\draw (52.center) to (59.center);
		\draw [style=morphism-edge, in=0, out=-90] (65.center) to (67.center);
		\draw [style=morphism-edge, in=180, out=-90] (66) to (67.center);
		\draw [style=morphism-edge] (67.center) to (68.center);
		\draw [style=morphism-edge, in=180, out=-90] (69.center) to (71.center);
		\draw [style=morphism-edge, in=0, out=-90] (70) to (71.center);
		\draw [style=morphism-edge] (71.center) to (72.center);
		\draw (73.center) to (51.center);
		\draw (58.center) to (74.center);
		\draw (74.center) to (73.center);
		\draw (75.center) to (52.center);
		\draw (59.center) to (76.center);
		\draw (76.center) to (75.center);
	\end{pgfonlayer}
\end{tikzpicture}
\end{equation}

Let us focus on the special case of\, \(\mathbb{X} = \BiCat{Cat}\),
though as discussed in~\cite[\S 1]{Street1972},
all of the following constructions generalise to general \(\mathbb{X}\)
admitting the construction of algebras.
There is a $2$-category $\BiCat{Mon}$ of monads on \(\BiCat{Cat}\).
Let
\begin{equation} \label{eq:inclusion-triv-monad}
  \Inc \from \BiCat{Cat} \to \BiCat{Mon},
  \qquad \cat{C} \mapsto (1_{\cat{C}}, \id_{1_{\cat{C}}}, \id_{1_{\cat{C}}})
\end{equation}
be the inclusion 2-functor, which maps any category to its identity monad.
The above functor has a right adjoint $\Alg \from \BiCat{Mon} \to \BiCat{Cat}$,
which maps any monad $(T, \mu, u)\from \cat{C} \to \cat{C}$
to its \emph{Eilenberg--Moore category} $\cat{C}^T$.
Using the previous 2-adjunction, one proves that to every monad $(T, \mu, u) \from\cat{C} \to\cat{C}$ one can associate adjoint functors
\begin{equation} \label{eq:Eilenberg--More-adjunction}
  F^T \from \cat{C} \to \cat{C}^T \qquad \text{ and } \qquad
  U^T \from \cat{C}^T \to \cat{C},
\end{equation}
whose unit and counit we denote by
$\eta \from 1_{\cat{C}} \nt U^T F^T$ and $\varepsilon \from F^T U^T \nt 1_{\cat{C}^T}$,
such that
$T = U^T F^T$,
$\mu = F^T \varepsilon U^T$,
and $u = \eta$.
We call the adjunction $\adj{F_T}{U_T}{\cat{C}}{\cat{C}^T}$ the \EM{} adjunction of\, $(T, \mu, u)$,
and depict \(\eta\) and \(\varepsilon\) by
\[
  \begin{tikzpicture}
	\begin{pgfonlayer}{nodelayer}
		\node [style=none] (63) at (1.25, 0) {\scriptsize $\eqdef$};
		\node [style=none] (64) at (4.5, 1.75) {};
		\node [style=none] (65) at (2, 1.75) {};
		\node [style=none] (66) at (4.5, -1.75) {};
		\node [style=none] (67) at (2, -1.75) {};
		\node [style=none] (68) at (3.75, -1.75) {};
		\node [style=none] (69) at (3.75, -2.1) {\tiny $U^T$};
		\node [style=none] (70) at (2.75, -1.75) {};
		\node [style=none] (71) at (2.75, -2.1) {\tiny $F^T$};
		\node [style=none] (72) at (3.25, 0) {};
		\node [style=none] (73) at (0.5, 1.72) {};
		\node [style=none] (74) at (-3, 1.72) {};
		\node [style=none] (75) at (0.5, -1.78) {};
		\node [style=none] (76) at (-3, -1.78) {};
		\node [style=none] (77) at (-2.25, -1.78) {};
		\node [style=none] (78) at (-2.25, -2.13) {\tiny $F^T$};
		\node [style=morphism] (79) at (-1.25, -0.03) {$\eta$};
		\node [style=none] (80) at (-0.25, -1.78) {};
		\node [style=none] (81) at (-0.25, -2.13) {\tiny $U^T$};
		\node [style=none] (82) at (-1.5, -0.28) {};
		\node [style=none] (83) at (-1, -0.28) {};
		\node [style=none] (84) at (-1.25, 1.72) {};
		\node [style=none] (85) at (-1.25, 2) {\tiny $\Id_{\cat{C}}$};
		\node [style=none] (86) at (13.25, -0.03) {\scriptsize $\eqdef$};
		\node [style=none] (87) at (14, -1.78) {};
		\node [style=none] (88) at (16.5, -1.78) {};
		\node [style=none] (89) at (14, 1.72) {};
		\node [style=none] (90) at (16.5, 1.72) {};
		\node [style=none] (91) at (14.75, 1.72) {};
		\node [style=none] (92) at (14.75, 1.97) {\tiny $U^T$};
		\node [style=none] (93) at (15.75, 1.72) {};
		\node [style=none] (94) at (15.75, 1.97) {\tiny $F^T$};
		\node [style=none] (95) at (15.25, -0.03) {};
		\node [style=none] (96) at (9, -1.75) {};
		\node [style=none] (97) at (12.5, -1.75) {};
		\node [style=none] (98) at (9, 1.75) {};
		\node [style=none] (99) at (12.5, 1.75) {};
		\node [style=none] (100) at (11.75, 1.75) {};
		\node [style=none] (101) at (11.75, 2) {\tiny $F^T$};
		\node [style=morphism] (102) at (10.75, 0) {$\varepsilon$};
		\node [style=none] (103) at (9.75, 1.75) {};
		\node [style=none] (104) at (9.75, 2) {\tiny $U^T$};
		\node [style=none] (105) at (11, 0.25) {};
		\node [style=none] (106) at (10.5, 0.25) {};
		\node [style=none] (107) at (10.75, -1.75) {};
		\node [style=none] (108) at (10.75, -2.13) {\tiny $\Id_{\cat{C}^T}$};
		\node [style=none] (109) at (-2.25, 1) {\scriptsize $\cat{C}$};
		\node [style=none] (110) at (-0.25, 1) {\scriptsize $\cat{C}$};
		\node [style=none] (111) at (-1.25, -1.25) {\scriptsize $\cat{C}^T$};
		\node [style=none] (112) at (10.75, 1.25) {\scriptsize $\cat{C}$};
		\node [style=none] (113) at (11.75, -1) {\scriptsize $\cat{C}^T$};
		\node [style=none] (114) at (9.75, -1) {\scriptsize $\cat{C}^T$};
		\node [style=none] (120) at (6.75, 0) {and};
	\end{pgfonlayer}
	\begin{pgfonlayer}{edgelayer}
		\draw (67.center) to (66.center);
		\draw (66.center) to (64.center);
		\draw (64.center) to (65.center);
		\draw (65.center) to (67.center);
		\draw [style=morphism-edge, in=90, out=180] (72.center) to (70.center);
		\draw [style=morphism-edge, in=90, out=0] (72.center) to (68.center);
		\draw (76.center) to (75.center);
		\draw (75.center) to (73.center);
		\draw (73.center) to (74.center);
		\draw (74.center) to (76.center);
		\draw [style=morphism-edge, in=90, out=-90] (82.center) to (77.center);
		\draw [style=morphism-edge, in=90, out=-90] (83.center) to (80.center);
		\draw [style=morphism-edge, dotted] (84.center) to (79);
		\draw (90.center) to (89.center);
		\draw (89.center) to (87.center);
		\draw (87.center) to (88.center);
		\draw (88.center) to (90.center);
		\draw [style=morphism-edge, in=-90, out=0] (95.center) to (93.center);
		\draw [style=morphism-edge, in=-90, out=180] (95.center) to (91.center);
		\draw (99.center) to (98.center);
		\draw (98.center) to (96.center);
		\draw (96.center) to (97.center);
		\draw (97.center) to (99.center);
		\draw [style=morphism-edge, in=-90, out=90] (105.center) to (100.center);
		\draw [style=morphism-edge, in=-90, out=90] (106.center) to (103.center);
		\draw [style=morphism-edge, dotted] (107.center) to (102);
	\end{pgfonlayer}
\end{tikzpicture}
\]
Graphically, the defining equations of adjunctions translate to the snake equations
\[
  \begin{tikzpicture}
	\begin{pgfonlayer}{nodelayer}
		\node [style=none] (0) at (-6.25, 0.75) {};
		\node [style=none] (1) at (-5.5, 0) {};
		\node [style=none] (2) at (-7, 0) {};
		\node [style=none] (3) at (-7.75, -0.75) {};
		\node [style=none] (4) at (-8.5, 0) {};
		\node [style=none] (5) at (-9.25, 1.75) {};
		\node [style=none] (6) at (-4.75, 1.75) {};
		\node [style=none] (7) at (-4.75, -1.75) {};
		\node [style=none] (8) at (-9.25, -1.75) {};
		\node [style=none] (9) at (-8.5, 1.75) {};
		\node [style=none] (10) at (-5.5, -1.75) {};
		\node [style=none] (11) at (-8.5, 2) {\tiny $U^T$};
		\node [style=none] (12) at (-5.5, -2.1) {\tiny $U^T$};
		\node [style=none] (13) at (-4.25, 0) {\scriptsize $=$};
		\node [style=none] (14) at (-3.75, 1.75) {};
		\node [style=none] (15) at (-2.25, 1.75) {};
		\node [style=none] (16) at (-2.25, -1.75) {};
		\node [style=none] (17) at (-3.75, -1.75) {};
		\node [style=none] (18) at (-3, 1.75) {};
		\node [style=none] (19) at (-3, -1.75) {};
		\node [style=none] (20) at (-3, 2) {\tiny $U^T$};
		\node [style=none] (21) at (-3, -2.1) {\tiny $U^T$};
		\node [style=none] (22) at (3.75, 0.75) {};
		\node [style=none] (23) at (3, 0) {};
		\node [style=none] (24) at (4.5, 0) {};
		\node [style=none] (25) at (5.25, -0.75) {};
		\node [style=none] (26) at (6, 0) {};
		\node [style=none] (27) at (2.25, 1.75) {};
		\node [style=none] (28) at (6.75, 1.75) {};
		\node [style=none] (29) at (6.75, -1.75) {};
		\node [style=none] (30) at (2.25, -1.75) {};
		\node [style=none] (31) at (6, 1.75) {};
		\node [style=none] (32) at (3, -1.75) {};
		\node [style=none] (33) at (6, 2) {\tiny $F^T$};
		\node [style=none] (34) at (3, -2.1) {\tiny $F^T$};
		\node [style=none] (35) at (7.25, 0) {\scriptsize $=$};
		\node [style=none] (36) at (7.75, 1.75) {};
		\node [style=none] (37) at (9.25, 1.75) {};
		\node [style=none] (38) at (9.25, -1.75) {};
		\node [style=none] (39) at (7.75, -1.75) {};
		\node [style=none] (40) at (8.5, 1.75) {};
		\node [style=none] (41) at (8.5, -1.75) {};
		\node [style=none] (42) at (8.5, 2) {\tiny $F^T$};
		\node [style=none] (43) at (8.5, -2.1) {\tiny $F^T$};
		\node [style=none] (44) at (0, 0) {and};
	\end{pgfonlayer}
	\begin{pgfonlayer}{edgelayer}
		\draw [style=morphism-edge] (9.center)
			 to (4.center)
			 to [in=180, out=-90] (3.center)
			 to [in=-90, out=0] (2.center)
			 to [in=180, out=90] (0.center)
			 to [in=90, out=0] (1.center)
			 to (10.center);
		\draw (8.center) to (5.center);
		\draw (5.center) to (6.center);
		\draw (6.center) to (7.center);
		\draw (7.center) to (8.center);
		\draw (17.center) to (14.center);
		\draw (14.center) to (15.center);
		\draw (15.center) to (16.center);
		\draw (16.center) to (17.center);
		\draw [style=morphism-edge] (19.center) to (18.center);
		\draw [style=morphism-edge] (32.center)
			 to (23.center)
			 to [in=-180, out=90] (22.center)
			 to [in=90, out=0] (24.center)
			 to [in=180, out=-90] (25.center)
			 to [in=-90, out=0] (26.center)
			 to (31.center);
		\draw (30.center) to (27.center);
		\draw (27.center) to (28.center);
		\draw (28.center) to (29.center);
		\draw (29.center) to (30.center);
		\draw (39.center) to (36.center);
		\draw (36.center) to (37.center);
		\draw (37.center) to (38.center);
		\draw (38.center) to (39.center);
		\draw [style=morphism-edge] (41.center) to (40.center);
	\end{pgfonlayer}
\end{tikzpicture}.
\]

\begin{remark}\label{rmk:diagrammatic-algebras}
  As noted in~\cite{Willerton2008}, the \EM{} category of  a monad \(T\) can be incorporated into the graphical calculus.
  Defining the natural transformation
  \[
    \vartheta \eqdef T U^T = U^T F^T U^T \xrightarrow{\;U^T\varepsilon\;} U^T
  \]
  we may write  \begin{equation}\label[diagram]{eq:action-over-monad-string-diagram}
    \begin{tikzpicture}
	\begin{pgfonlayer}{nodelayer}
		\node [style=none] (0) at (-1.25, 1.5) {};
		\node [style=morphism] (1) at (-3, 0) {$\vartheta$};
		\node [style=none] (2) at (-3.25, 1.5) {};
		\node [style=none] (3) at (-2.75, -1.5) {};
		\node [style=none] (4) at (-0.5, -1.5) {};
		\node [style=none] (5) at (-4.75, -1.5) {};
		\node [style=none] (6) at (-4.75, 1.5) {};
		\node [style=none] (7) at (-0.5, 1.5) {};
		\node [style=none] (8) at (0.25, 0) {\scriptsize $\eqdef$};
		\node [style=none] (9) at (3.5, 1.5) {};
		\node [style=none] (10) at (2.25, 0) {};
		\node [style=none] (11) at (2, 1.5) {};
		\node [style=none] (12) at (2.5, -1.5) {};
		\node [style=none] (13) at (4.25, -1.5) {};
		\node [style=none] (14) at (1, -1.5) {};
		\node [style=none] (15) at (1, 1.5) {};
		\node [style=none] (16) at (4.25, 1.5) {};
		\node [style=none] (17) at (-1.25, 1.75) {\tiny $T$};
		\node [style=none] (18) at (-3.25, 1.78) {\tiny $U^T$};
		\node [style=none] (20) at (-2, 1) {\scriptsize $\cat{C}$};
		\node [style=none] (21) at (-1.25, -0.75) {\scriptsize $\cat{C}$};
		\node [style=none] (22) at (-4, -0.75) {\scriptsize $\cat{C}^T$};
		\node [style=none] (23) at (-2.75, -1.83) {\tiny $U^T$};
		\node [style=none] (24) at (3.5, 1.75) {\tiny $T$};
		\node [style=none] (32) at (2, 1.78) {};
		\node [style=none] (34) at (2.5, -1.83) {};
	\end{pgfonlayer}
	\begin{pgfonlayer}{edgelayer}
		\draw [style=morphism-edge, in=0, out=-90] (0.center) to (1);
		\draw [style=ddd, in=-90, out=90] (3.center) to (2.center);
		\draw (6.center) to (7.center);
		\draw (7.center) to (4.center);
		\draw (4.center) to (5.center);
		\draw (5.center) to (6.center);
		\draw [style=morphism-edge, in=0, out=-90] (9.center) to (10.center);
		\draw [style=ddd, in=-90, out=90] (12.center) to (11.center);
		\draw (15.center) to (16.center);
		\draw (16.center) to (13.center);
		\draw (13.center) to (14.center);
		\draw (14.center) to (15.center);
	\end{pgfonlayer}
\end{tikzpicture}
  \end{equation}
  We think of\, $\vartheta \from T U^T \to U^T$ as an action of $T$ on $U^T$ due to the identities
  \[
    \begin{tikzpicture}
	\begin{pgfonlayer}{nodelayer}
		\node [style=none] (0) at (-5.5, 1.5) {};
		\node [style=none] (1) at (-6.5, 1.5) {};
		\node [style=none] (2) at (-7.5, 1.5) {};
		\node [style=none] (3) at (-6, 0.25) {};
		\node [style=none] (4) at (-6.75, -1.5) {};
		\node [style=none] (5) at (-7.02, -0.25) {};
		\node [style=none] (10) at (-8.25, -1.5) {};
		\node [style=none] (11) at (-8.25, 1.5) {};
		\node [style=none] (12) at (-4.75, 1.5) {};
		\node [style=none] (13) at (-4.75, -1.5) {};
		\node [style=none] (14) at (-4, 0) {\scriptsize $=$};
		\node [style=none] (15) at (-0.75, 1.5) {};
		\node [style=none] (16) at (-1.75, 1.5) {};
		\node [style=none] (17) at (-2.5, 1.5) {};
		\node [style=none] (19) at (-1.75, -1.5) {};
		\node [style=none] (20) at (-1.92, -0.5) {};
		\node [style=none] (21) at (-3.25, -1.5) {};
		\node [style=none] (22) at (-3.25, 1.5) {};
		\node [style=none] (23) at (0.25, 1.5) {};
		\node [style=none] (24) at (0.25, -1.5) {};
		\node [style=none] (25) at (-2.25, 0.25) {};
		\node [style=none] (28) at (6.75, 1.5) {};
		\node [style=none] (29) at (7.5, -1.5) {};
		\node [style=none] (31) at (6, -1.5) {};
		\node [style=none] (32) at (6, 1.5) {};
		\node [style=none] (33) at (9, 1.5) {};
		\node [style=none] (34) at (9, -1.5) {};
		\node [style=blackdot] (37) at (7.75, 0.5) {};
		\node [style=none] (38) at (7.25, -0.25) {};
		\node [style=none] (39) at (11.25, 1.5) {};
		\node [style=none] (40) at (12, -1.5) {};
		\node [style=none] (41) at (10.5, -1.5) {};
		\node [style=none] (42) at (10.5, 1.5) {};
		\node [style=none] (43) at (13.5, 1.5) {};
		\node [style=none] (44) at (13.5, -1.5) {};
		\node [style=none] (47) at (9.75, 0) {\scriptsize $=$};
		\node [style=none] (48) at (-0.75, 1.75) {\tiny $T$};
		\node [style=none] (49) at (-1.75, 1.75) {\tiny $T$};
		\node [style=none] (50) at (-5.5, 1.75) {\tiny $T$};
		\node [style=none] (51) at (-6.5, 1.75) {\tiny $T$};
		\node [style=none] (52) at (3, 0) {and};
	\end{pgfonlayer}
	\begin{pgfonlayer}{edgelayer}
		\draw [style=morphism-edge] (0.center)
			 to [in=0, out=-90] (3.center)
			 to [in=-90, out=180] (1.center);
		\draw [style=ddd, in=-90, out=90] (4.center) to (2.center);
		\draw [style=morphism-edge, in=0, out=-90] (3.center) to (5.center);
		\draw (13.center) to (10.center);
		\draw (10.center) to (11.center);
		\draw (11.center) to (12.center);
		\draw (12.center) to (13.center);
		\draw [style=ddd, in=-90, out=90] (19.center) to (17.center);
		\draw (24.center) to (21.center);
		\draw (21.center) to (22.center);
		\draw (22.center) to (23.center);
		\draw (23.center) to (24.center);
		\draw [style=morphism-edge, in=-90, out=0] (25.center) to (16.center);
		\draw [style=morphism-edge, in=-90, out=0] (20.center) to (15.center);
		\draw [style=ddd, in=-90, out=90] (29.center) to (28.center);
		\draw (34.center) to (31.center);
		\draw (31.center) to (32.center);
		\draw (32.center) to (33.center);
		\draw (33.center) to (34.center);
		\draw [style=morphism-edge, in=0, out=-90] (37) to (38.center);
		\draw [style=ddd, in=-90, out=90] (40.center) to (39.center);
		\draw (44.center) to (41.center);
		\draw (41.center) to (42.center);
		\draw (42.center) to (43.center);
		\draw (43.center) to (44.center);
	\end{pgfonlayer}
\end{tikzpicture}
  \]
\end{remark}

Any adjunction $\stdadj$ defines a monad \((UF, F \varepsilon U, \eta)\).
A question one might ask is
how much the functors \(F\) and \(U\) ``differ'' from the free and forgetful functors
\(F^T \from \cat{C} \to \cat{C}^T\) and \(U^T \from \cat{C}^T \to \cat{C}\) of \(T\), respectively.

\begin{lemma}[{\cite[Theorem 3]{Street1972}}]\label{def: comparison-functor}
  Let \(T\) be the monad of the adjunction \stdadj.
  There exists a unique functor \(\Sigma \from \cat D \to \cat{C}^T\) satisfying \(\Sigma F = F^T\) and \(U^T \Sigma = U\).
\end{lemma}

Given an adjunction $\stdadj$, we call the unique functor \(\Sigma \from \cat D \to \cat{C}^T\) discussed in the previous lemma the \emph{comparison functor} and refer to the adjunction $F \lad U$ as \emph{monadic} if\, $\Sigma$ is an equivalence.

\section{Bimonads}\label{sec:bimonads}

\noindent Bimonads are a vast generalisation of bialgebras.
They naturally arise in the study of (rigid) monoidal categories and topological quantum field theories, see amongst others~\cite{Kerler2001, Moerdijk2002, Bruguieres2007, Bruguieres2011, Turaev2017}.
We note that the term ``bimonad'' is also used by Mesablishvili and Wisbauer for a similar---though distinct---concept;
see~\cite{Mesablishvili2011}.
For this reason, what we call a bimonad is also sometimes called an opmonoidal (or comonoidal) monad in the literature.

Due to the lack of a braiding on the endofunctors \(\mathrm{End}(\cat{C})\) over \(\cat{C}\), the naive notion of bialgebra does not generalise to the monadic setting and needs to be adjusted.
One possible way of overcoming this problem was introduced and studied by Moerdijk under the name ``Hopf monad''\footnote{As remarked in~\cite{Moerdijk2002}, the concept of Hopf monads is strictly dual to that of monoidal comonads, which are studied for example in~\cite{Boardman1995}.} in~\cite{Moerdijk2002}; the idea being that the opmonoidal structure of a functor replaces the bialgebra's comultiplication and counit.
Following the conventions of~\cite{Bruguieres2007}, we refer to such structures as bimonads.

\begin{definition}\label{def: mon_functor}
  An \emph{opmonoidal functor} between (strict) monoidal categories \((\cat{C}, \otimes, 1)\) and \((\cat{D}, \otimes, 1)\)
  is a functor \(F\from \cat{C} \to \cat{D}\)
  together with two natural transformations\footnote{%
    \,We view the monoidal unit of\, $\cat{C}$ as a functor $1 \from \1 \to\cat{C}$, where $\1$ is the terminal category.}
  \[
    F_2 \from F(\blank \otimes \bblank) \nt F\blank \otimes F\bblank \qquad \text{ and } \qquad
    F_0 \from F1 \to 1
  \]
  called the \emph{comultiplication} and  \emph{counit},
  satisfying \emph{coassociativity} and \emph{counitality}:
  \begin{equation} \label{eq:opmon-fun->coassociativity-counitality}
    \adjustbox{scale=0.9}{
      \begin{tikzcd}
        {F(\blank \otimes \bblank \otimes \bbblank)} & {F\blank \otimes F(\bblank \otimes \bbblank)} & {F1 \otimes F\blank} & F\blank & {F\blank \otimes F1} \\
        {F(\blank \otimes \bblank) \otimes F\bbblank} & {F\blank \otimes F\bblank \otimes F\bbblank} && F\blank
        \arrow["{{F_{2;\blank, \bblank \otimes \bbblank}}}", from=1-1, to=1-2]
        \arrow["{{F\blank \otimes F_{2;\bblank,\bbblank}}}", from=1-2, to=2-2]
        \arrow["{{F_{2; \blank \otimes \bblank, \bbblank}}}"', from=1-1, to=2-1]
        \arrow["{{F_{2;\blank,\bblank} \otimes F\bbblank}}"', from=2-1, to=2-2]
        \arrow["{{F_{2; 1, \blank}}}"', from=1-4, to=1-3]
        \arrow["{{F_0 \otimes F\blank}}"', from=1-3, to=2-4]
        \arrow[Rightarrow, no head, from=1-4, to=2-4]
        \arrow["{{F_{2;\blank,1}}}", from=1-4, to=1-5]
        \arrow["{{F\blank \otimes F_0}}", from=1-5, to=2-4]
      \end{tikzcd}
    }
  \end{equation}

  If\, \(F_2\) and \(F_0\) are isomorphisms or identities,
  we call \(F\) \emph{strong monoidal} or \emph{strict monoidal}, respectively.
\end{definition}

\begin{definition}\label{def: mon_nat_trafo}
  An \emph{opmonoidal natural transformation}
  between opmonoidal functors \(F, G \from \cat{C} \to \cat{D}\) is
  a natural transformation \(\rho \from F \nt G\)
  such that the diagrams
  \[
    \begin{tikzcd}
      {F(\blank \otimes \bblank)} & {G(\blank \otimes \bblank)} & F1 & G1 \\
      {F\blank \otimes F\bblank} & {G\blank \otimes G\bblank} && 1
      \arrow["{{\rho_{\blank \otimes \bblank}}}", from=1-1, to=1-2]
      \arrow["{{G_{2,\blank,\bblank}}}", from=1-2, to=2-2]
      \arrow["{{F_{2,\blank,\bblank}}}"', from=1-1, to=2-1]
      \arrow["{{\rho_\blank \otimes \rho_\bblank}}"', from=2-1, to=2-2]
      \arrow["{{F_0}}"', from=1-3, to=2-4]
      \arrow["{{\rho_1}}", from=1-3, to=1-4]
      \arrow["{{G_0}}", from=1-4, to=2-4]
    \end{tikzcd}
  \]
  commute.

  If\, \(\rho\) is additionally a natural isomorphism, we call it an \emph{opmonoidal natural isomorphism}.
\end{definition}

Willerton introduced a graphical calculus for opmonoidal functors in~\cite{Willerton2008}.
The key idea is to incorporate the Cartesian product of categories, or, more generally, the tensor product of a monoidal 2-category into the diagrammatic calculus.

As before, we consider strings and vertices between them.
These are labelled with functors and natural transformations, respectively.
The strings and vertices are embedded into bounded rectangles, which we will call sheets.
Each (connected) region of a sheet is decorated with a category.
The same mechanics as for string diagrams apply: horizontal and vertical gluing represents composition of functors and natural transformations.
On top of these operations, we add stacking sheets behind each other to depict the monoidal product in \(\BiCat{Cat}\).
Our convention is to read diagrams from front to back, left to right, and top to bottom.

Two of the most vital building blocks in this new graphical language are the tensor product and unit of a monoidal category \((\cat{C}, \otimes, 1)\):
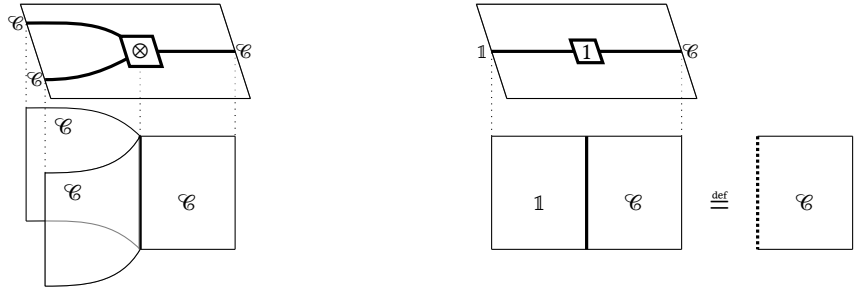
\[
  \begin{tikzpicture}
	\begin{pgfonlayer}{nodelayer}
		\node [style=none] (14) at (-3.25, 2.25) {};
		\node [style=none] (16) at (-8.75, 3) {};
		\node [style=none] (17) at (-8.25, 1.5) {};
		\node [style=none] (18) at (-5.75, 2.25) {\scriptsize $\otimes$};
		\node [style=none] (19) at (-8.9, 3.5) {};
		\node [style=none] (21) at (-3.65, 3.5) {};
		\node [style=none] (45) at (3.5, 0) {};
		\node [style=none] (46) at (3.5, -3) {};
		\node [style=none] (47) at (6, 0) {};
		\node [style=none] (48) at (8.5, 0) {};
		\node [style=none] (49) at (8.5, -3) {};
		\node [style=none] (50) at (6, -3) {};
		\node [style=none] (51) at (10.5, 0) {};
		\node [style=none] (52) at (9.5, -1.75) {$\eqdef$};
		\node [style=none] (53) at (13, -3) {};
		\node [style=none] (54) at (13, 0) {};
		\node [style=none] (56) at (11.75, -1.75) {\scriptsize $\cat C$};
		\node [style=none] (57) at (7.25, -1.75) {\scriptsize $\cat C$};
		\node [style=none] (58) at (4.75, -1.75) {\scriptsize $\1$};
		\node [style=none] (60) at (-6.025, 1.85) {};
		\node [style=none] (61) at (-6.275, 2.65) {};
		\node [style=none] (62) at (-5.425, 2.65) {};
		\node [style=none] (63) at (-5.175, 1.85) {};
		\node [style=none] (64) at (3.5, 2.25) {};
		\node [style=none] (65) at (6, 2.25) {\scriptsize $1$};
		\node [style=none] (66) at (8.5, 2.25) {};
		\node [style=none] (67) at (3.1, 3.5) {};
		\node [style=none] (68) at (3.9, 1) {};
		\node [style=none] (69) at (8.1, 3.5) {};
		\node [style=none] (70) at (8.9, 1) {};
		\node [style=none] (71) at (5.775, 1.95) {};
		\node [style=none] (72) at (5.575, 2.55) {};
		\node [style=none] (73) at (6.225, 2.55) {};
		\node [style=none] (74) at (6.425, 1.95) {};
		\node [style=none] (81) at (10.5, -3) {};
		\node [style=none] (0) at (-3.25, -3) {};
		\node [style=none] (2) at (-3.25, 0) {};
		\node [style=none] (4) at (-8.75, 0.75) {};
		\node [style=none] (7) at (-8.75, -2.25) {};
		\node [style=none] (20) at (-8.1, 1) {};
		\node [style=none] (22) at (-2.85, 1) {};
		\node [style=none] (38) at (-7.5, -1.5) {\scriptsize$\cat C$};
		\node [style=none] (40) at (-4.5, -1.75) {\scriptsize $\cat C$};
		\node [style=none] (1) at (-5.75, -3) {};
		\node [style=none] (3) at (-5.75, 0) {};
		\node [style=none] (5) at (-8.25, -0.975) {};
		\node [style=none] (6) at (-8.25, -3.975) {};
		\node [style=none] (39) at (-7.75, 0.25) {\scriptsize $\cat C$};
		\node [style=none] (82) at (-9, 3) {\tiny $\cat{C}$};
		\node [style=none] (83) at (-8.5, 1.5) {};
		\node [style=none] (84) at (-8.5, 1.5) {\tiny $\cat{C}$};
		\node [style=none] (85) at (-3, 2.25) {\tiny $\cat{C}$};
		\node [style=none] (86) at (8.75, 2.25) {\tiny $\cat{C}$};
		\node [style=none] (87) at (3.25, 2.25) {\tiny $\1$};
	\end{pgfonlayer}
	\begin{pgfonlayer}{edgelayer}
		\draw [dotted] (14.center) to (2.center);
		\draw [dotted] (5.center) to (17.center);
		\draw [dotted] (4.center) to (16.center);
		\draw [dotted] (3.center) to (18.center);
		\draw [dotted] (45.center) to (64.center);
		\draw [dotted] (48.center) to (66.center);
		\draw [in=135, out=0] (7.center) to (1.center);
		\draw (4.center) to (7.center);
		\draw [style=Front] (22.center)
			 to (21.center)
			 to (19.center)
			 to (20.center)
			 to cycle;
		\draw [in=0, out=-120] (3.center) to (5.center);
		\draw (1.center) to (0.center);
		\draw (0.center) to (2.center);
		\draw (2.center) to (3.center);
		\draw [in=-120, out=0] (6.center) to (1.center);
		\draw (6.center) to (5.center);
		\draw [style=very thick] (3.center) to (1.center);
		\draw [in=0, out=135] (3.center) to (4.center);
		\draw (19.center) to (20.center);
		\draw (21.center) to (22.center);
		\draw (20.center) to (22.center);
		\draw (19.center) to (21.center);
		\draw [style=morphism-edge] (18.center) to (14.center);
		\draw [style=morphism-edge, in=0, out=-150] (18.center) to (17.center);
		\draw [style=morphism-edge, in=0, out=135] (18.center) to (16.center);
		\draw (45.center) to (48.center);
		\draw (48.center) to (49.center);
		\draw (49.center) to (46.center);
		\draw (46.center) to (45.center);
		\draw [style=morphism-edge] (47.center) to (50.center);
		\draw (51.center) to (54.center);
		\draw [style=unit] (51.center) to (81.center);
		\draw (81.center) to (53.center);
		\draw (54.center) to (53.center);
		\draw [style=morphism-edge, fill=white] (60.center)
			 to (63.center)
			 to (62.center)
			 to (61.center)
			 to cycle;
		\draw [style=Front] (70.center)
			 to (69.center)
			 to (67.center)
			 to (68.center)
			 to cycle;
		\draw (67.center) to (68.center);
		\draw (69.center) to (70.center);
		\draw (68.center) to (70.center);
		\draw (67.center) to (69.center);
		\draw [style=morphism-edge] (64.center) to (65.center);
		\draw [style=morphism-edge] (65.center) to (66.center);
		\draw [style=morphism-edge, fill=white] (71.center)
			 to (74.center)
			 to (73.center)
			 to (72.center)
			 to cycle;
		\draw [style=Front] (5.center)
			 to (6.center)
			 to [in=-120, out=0] (1.center)
			 to (3.center)
			 to [in=0, out=-120] cycle;
	\end{pgfonlayer}
\end{tikzpicture}
\]
On the left, there are two sheets equating to two copies of\, \(\cat{C}\) joined by a line: the tensor product of\, \(\cat{C}\).
On the right, we have the unit of\, \(\cat{C}\) considered as a functor \(1 \from \1 \to \cat{C}\), where \(\1\) is the terminal category.
We represent \(\1\) by the empty sheet, and the unit of\, \(\cat{C}\) by a dashed line.

\begin{example}\label{ex:graphical-opmonoidal-functor}
  Consider an opmonoidal functor \((F, F_2, F_0) \from \cat{C} \to \cat{D}\).
  Graphically, one may depict \(F_2\) and \(F_0\) like the following:
  \[
    \begin{tikzpicture}
	\begin{pgfonlayer}{nodelayer}
		\node [style=none] (0) at (-2.25, -1.5) {};
		\node [style=none] (1) at (-4, -1.5) {};
		\node [style=none] (2) at (-2.25, 1.5) {};
		\node [style=none] (3) at (-4, 1.5) {};
		\node [style=none] (4) at (-7, 2.25) {};
		\node [style=none] (5) at (-6.5, 0.525) {};
		\node [style=none] (6) at (-6.5, -2.475) {};
		\node [style=none] (7) at (-7, -0.75) {};
		\node [style=blackdot] (8) at (-4, 0) {};
		\node [style=none] (9) at (-3, 1.5) {};
		\node [style=none] (11) at (-6.25, -0.75) {};
		\node [style=none] (12) at (-5.75, -2.45) {};
		\node [style=none] (63) at (-6, 0) {\scriptsize $\cat C$};
		\node [style=none] (64) at (-6.5, 1.75) {\scriptsize $\cat C$};
		\node [style=none] (69) at (-2.75, -1) {\scriptsize $\cat C$};
		\node [style=none] (81) at (4.25, -1.25) {};
		\node [style=none] (82) at (2.25, -1.25) {};
		\node [style=none] (83) at (4.25, 0.75) {};
		\node [style=none] (84) at (2.25, 0.75) {};
		\node [style=none] (85) at (3.25, 0.75) {};
		\node [style=none] (87) at (2.25, -0.25) {};
		\node [style=none] (89) at (3.75, -0.75) {\scriptsize $\cat C$};
		\node [style=none] (21) at (-3, 1.75) {\tiny $F$};
		\node [style=none] (22) at (-5.75, -2.75) {\tiny $F$};
		\node [style=none] (23) at (-6.25, -1.05) {\tiny $F$};
		\node [style=none] (86) at (3.25, 1) {\tiny $F$};
		\node [style=none] (90) at (-1.5, 0.25) {};
		\node [style=none] (91) at (-1.5, 0.5) {\tiny$F_2$};
		\node [style=none] (92) at (-3.8, -0.175) {};
	\end{pgfonlayer}
	\begin{pgfonlayer}{edgelayer}
		\draw [style=morphism-edge, in=90, out=-165, looseness=0.75] (8) to (11.center);
		\draw [in=135, out=0] (7.center) to (1.center);
		\draw (4.center) to (7.center);
		\draw [style=Front] (5.center)
			 to (6.center)
			 to [in=-120, out=0] (1.center)
			 to (3.center)
			 to [in=0, out=-120] cycle;
		\draw [in=0, out=-120] (3.center) to (5.center);
		\draw [style=morphism-edge, in=90, out=-150] (8) to (12.center);
		\draw (1.center) to (0.center);
		\draw (0.center) to (2.center);
		\draw (2.center) to (3.center);
		\draw [in=-120, out=0] (6.center) to (1.center);
		\draw (6.center) to (5.center);
		\draw [style=very thick] (3.center) to (1.center);
		\draw [in=0, out=135] (3.center) to (4.center);
		\draw [style=morphism-edge, in=15, out=-90] (9.center) to (8);
		\draw (83.center) to (81.center);
		\draw (83.center) to (84.center);
		\draw (82.center) to (81.center);
		\draw [style=unit] (84.center) to (82.center);
		\draw [style=morphism-edge, in=0, out=-90] (85.center) to (87.center);
		\draw [->, draw=black, dotted, in=-30, out=-90] (90.center) to (92.center);
	\end{pgfonlayer}
\end{tikzpicture}
  \]
  The coassociativity and counitality of \cref{eq:opmon-fun->coassociativity-counitality} become:
  \begin{equation}\label[diagram]{eq:coass-of-opmonoidal-functor}
    \begin{tikzpicture}[xscale=0.9, yscale=0.9]
	\begin{pgfonlayer}{nodelayer}
		\node [style=none] (272) at (4.55, 0.2) {};
		\node [style=none] (104) at (-0.5, 2) {};
		\node [style=none] (105) at (-2.75, 2) {};
		\node [style=none] (106) at (-0.5, -2) {};
		\node [style=none] (107) at (-4.75, 2.75) {};
		\node [style=none] (108) at (-7, 1.25) {};
		\node [style=none] (110) at (-7.75, 2.25) {};
		\node [style=none] (111) at (-2.75, -2) {};
		\node [style=none] (112) at (-8.5, -1) {};
		\node [style=none] (113) at (-5, -2.75) {};
		\node [style=none] (114) at (-7.75, -2.25) {};
		\node [style=none] (115) at (-7, -3.75) {};
		\node [style=none] (116) at (-3.75, -0.75) {};
		\node [style=none] (117) at (-8.5, 3.5) {};
		\node [style=none] (118) at (-1.25, 2) {};
		\node [style=none] (124) at (0, 0) {\scriptsize $=$};
		\node [style=none] (146) at (-1.25, 2.25) {\tiny $F$};
		\node [style=none] (207) at (-2.725, 1.25) {};
		\node [style=none] (208) at (-6.5, -3.75) {};
		\node [style=none] (209) at (-4.375, 0.65) {};
		\node [style=none] (210) at (-7.25, -2.25) {};
		\node [style=none] (211) at (-8, -1) {};
		\node [style=none] (212) at (-6.5, -4) {\tiny $F$};
		\node [style=none] (213) at (-7.25, -2.5) {\tiny $F$};
		\node [style=none] (214) at (-8, -1.25) {\tiny $F$};
		\node [style=none] (247) at (8.5, 2) {};
		\node [style=none] (248) at (6.25, 2) {};
		\node [style=none] (249) at (8.5, -2) {};
		\node [style=none] (250) at (4.25, 2.75) {};
		\node [style=none] (251) at (2, 1.25) {};
		\node [style=none] (252) at (1.25, 2.25) {};
		\node [style=none] (253) at (6.25, -2) {};
		\node [style=none] (254) at (0.5, -0.75) {};
		\node [style=none] (255) at (4, -2.75) {};
		\node [style=none] (256) at (1.25, -2.25) {};
		\node [style=none] (257) at (2, -3.75) {};
		\node [style=none] (258) at (5.25, 1) {};
		\node [style=none] (259) at (0.5, 3.5) {};
		\node [style=none] (260) at (7.75, 2) {};
		\node [style=none] (261) at (7.75, 2.25) {\tiny $F$};
		\node [style=none] (263) at (2.5, -3.75) {};
		\node [style=none] (265) at (1.75, -2.25) {};
		\node [style=none] (266) at (1, -0.75) {};
		\node [style=none] (267) at (2.5, -4) {\tiny $F$};
		\node [style=none] (268) at (1.75, -2.5) {\tiny $F$};
		\node [style=none] (269) at (1, -1) {\tiny $F$};
		\node [style=none] (270) at (4.4, -0.5) {};
		\node [style=none] (271) at (5.675, 0.5) {};
	\end{pgfonlayer}
	\begin{pgfonlayer}{edgelayer}
		\draw [in=135, out=-60] (250.center) to (258.center);
		\draw [in=135, out=-60] (107.center) to (116.center);
		\draw [style=morphism-edge, in=180, out=90] (266.center) to (272.center);
		\draw [style=morphism-edge, in=30, out=-165] (207.center) to (209.center);
		\draw [style=morphism-edge, in=90, out=-180] (209.center) to (211.center);
		\draw [in=-240, out=0, looseness=0.75] (112.center) to (111.center);
		\draw [in=0, out=-225] (113.center) to (114.center);
		\draw (105.center) to (104.center);
		\draw (104.center) to (106.center);
		\draw [style=Front] (113.center)
			 to [in=-120, out=0] (111.center)
			 to [in=315, out=90] (116.center)
			 to [in=-60, out=135] (107.center)
			 to [in=0, out=-135, looseness=0.75] (110.center)
			 to (114.center)
			 to [in=135, out=0] cycle;
		\draw [style=Front, in=-150, out=60] (113.center) to (116.center);
		\draw [in=-225, out=0] (107.center) to (105.center);
		\draw [style=morphism-edge, in=90, out=-165] (209.center) to (210.center);
		\draw [style=Front] (108.center)
			 to [in=-135, out=0, looseness=0.75] (105.center)
			 to [in=15, out=-90] (116.center)
			 to [in=90, out=-45] (111.center)
			 to [in=0, out=-120] (113.center)
			 to [in=0, out=-105] (115.center)
			 to cycle;
		\draw (106.center) to (111.center);
		\draw (117.center) to (112.center);
		\draw [in=60, out=-150] (116.center) to (113.center);
		\draw [in=360, out=135] (107.center) to (117.center);
		\draw [in=240, out=0] (113.center) to (111.center);
		\draw [in=0, out=-105] (113.center) to (115.center);
		\draw (115.center) to (108.center);
		\draw [in=225, out=0, looseness=0.75] (108.center) to (105.center);
		\draw [in=15, out=-90] (105.center) to (116.center);
		\draw [in=90, out=-45] (116.center) to (111.center);
		\draw [style=morphism-edge, in=270, out=15] (207.center) to (118.center);
		\draw [style=morphism-edge, in=90, out=-150] (207.center) to (208.center);
		\draw [in=-240, out=0, looseness=0.75] (254.center) to (253.center);
		\draw [in=0, out=-225] (255.center) to (256.center);
		\draw (248.center) to (247.center);
		\draw (247.center) to (249.center);
		\draw [style=Front] (255.center)
			 to [in=-120, out=0] (253.center)
			 to [in=315, out=90] (258.center)
			 to [in=-60, out=135] (250.center)
			 to [in=0, out=-135, looseness=0.75] (252.center)
			 to (256.center)
			 to [in=135, out=0] cycle;
		\draw [style=Front, in=-150, out=60] (255.center) to (258.center);
		\draw [in=-225, out=0] (250.center) to (248.center);
		\draw [style=morphism-edge, in=195, out=90] (265.center) to (270.center);
		\draw [style=morphism-edge, in=195, out=0] (272.center) to (271.center);
		\draw [style=Front] (251.center)
			 to [in=-135, out=0, looseness=0.75] (248.center)
			 to [in=15, out=-90] (258.center)
			 to [in=90, out=-45] (253.center)
			 to [in=0, out=-120] (255.center)
			 to [in=0, out=-105] (257.center)
			 to cycle;
		\draw (249.center) to (253.center);
		\draw (259.center) to (254.center);
		\draw [in=60, out=-150] (258.center) to (255.center);
		\draw [in=360, out=135] (250.center) to (259.center);
		\draw [in=240, out=0] (255.center) to (253.center);
		\draw [in=0, out=-105] (255.center) to (257.center);
		\draw (257.center) to (251.center);
		\draw [in=225, out=0, looseness=0.75] (251.center) to (248.center);
		\draw [in=15, out=-90] (248.center) to (258.center);
		\draw [in=90, out=-45] (258.center) to (253.center);
		\draw [style=morphism-edge, in=210, out=90] (263.center) to (270.center);
		\draw [style=morphism-edge, in=30, out=-90] (260.center) to (271.center);
		\draw [style=morphism-edge, in=210, out=30] (270.center) to (271.center);
		\draw [in=0, out=-135, looseness=0.75] (107.center) to (110.center);
		\draw [in=0, out=-135, looseness=0.75] (250.center) to (252.center);
		\draw (110.center) to (114.center);
		\draw (252.center) to (256.center);
	\end{pgfonlayer}
\end{tikzpicture}
  \end{equation}
  \begin{equation}\label[diagram]{eq:counit-of-opmonoidal-functor}
    \begin{tikzpicture}
	\begin{pgfonlayer}{nodelayer}
		\node [style=none] (0) at (-7, -1.5) {};
		\node [style=none] (1) at (-7, 1.5) {};
		\node [style=none] (2) at (-7.5, 1.5) {};
		\node [style=none] (3) at (-7.5, 1.75) {\tiny $F$};
		\node [style=none] (4) at (-8.5, -1.5) {};
		\node [style=none] (5) at (-11, 1.5) {};
		\node [style=none] (6) at (-11.25, -2.25) {};
		\node [style=none] (7) at (-10.5, -0.75) {};
		\node [style=none] (8) at (-8.5, 0.5) {};
		\node [style=none] (9) at (-8.5, -0.25) {};
		\node [style=none] (10) at (-10, 0.05) {};
		\node [style=none] (11) at (-10.75, -2.25) {};
		\node [style=none] (12) at (-10.75, -2.55) {\tiny $F$};
		\node [style=none] (13) at (-6.425, 0) {\scriptsize $=$};
		\node [style=none] (14) at (5.75, -1.5) {};
		\node [style=none] (15) at (4, -1.5) {};
		\node [style=none] (16) at (5.75, 1.5) {};
		\node [style=none] (17) at (4, 1.5) {};
		\node [style=none] (18) at (2, 1.5) {};
		\node [style=none] (19) at (2.5, -2.25) {};
		\node [style=none] (20) at (4, -0.75) {};
		\node [style=none] (21) at (5.25, 1.5) {};
		\node [style=none] (22) at (2.025, -1.425) {};
		\node [style=none] (23) at (5.25, 1.75) {\tiny $F$};
		\node [style=none] (24) at (4, -0.25) {};
		\node [style=none] (25) at (2.65, -1.7) {};
		\node [style=none] (26) at (6.25, 0) {\scriptsize $=$};
		\node [style=none] (27) at (1.5, -1.5) {};
		\node [style=none] (28) at (3.05, -0.75) {};
		\node [style=none] (29) at (-1.75, -1.5) {};
		\node [style=none] (30) at (-1.75, 1.5) {};
		\node [style=none] (31) at (-2.25, 1.5) {};
		\node [style=none] (32) at (-2.25, 1.75) {\tiny $F$};
		\node [style=none] (33) at (-3.25, -1.5) {};
		\node [style=none] (34) at (-5.75, 1.5) {};
		\node [style=none] (35) at (-6, -2.25) {};
		\node [style=none] (36) at (-4.75, -1) {};
		\node [style=none] (37) at (-3.25, -0.25) {};
		\node [style=none] (38) at (-5.5, -2.25) {};
		\node [style=none] (39) at (-5.5, -2.55) {\tiny $F$};
		\node [style=none] (41) at (10.75, -1.5) {};
		\node [style=none] (42) at (9, -1.5) {};
		\node [style=none] (43) at (10.75, 1.5) {};
		\node [style=none] (44) at (9, 1.5) {};
		\node [style=none] (45) at (7, 1.5) {};
		\node [style=none] (46) at (7.5, -2.25) {};
		\node [style=none] (47) at (10.25, 1.5) {};
		\node [style=none] (48) at (7.025, -1.425) {};
		\node [style=none] (49) at (10.25, 1.75) {\tiny $F$};
		\node [style=none] (50) at (9, -0.25) {};
		\node [style=none] (51) at (6.5, -1.5) {};
		\node [style=none] (52) at (7, -1.7) {\tiny $F$};
		\node [style=none] (53) at (2, -1.675) {\tiny $F$};
		\node [style=none] (64) at (2.8, -1.25) {};
		\node [style=none] (65) at (7.8, -1.25) {};
	\end{pgfonlayer}
	\begin{pgfonlayer}{edgelayer}
		\draw [in=0, out=-225] (4.center) to (7.center);
		\draw (8.center) to (4.center);
		\draw [style=unit, in=90, out=-165] (8.center) to (7.center);
		\draw [style=morphism-edge, in=0, out=-195] (9.center) to (10.center);
		\draw [style=Front] (5.center)
			 to [in=90, out=-90] (6.center)
			 to [in=-120, out=0] (4.center)
			 to (0.center)
			 to (1.center)
			 to cycle;
		\draw [style=morphism-edge, in=0, out=-90] (2.center) to (9.center);
		\draw [style=morphism-edge, in=90, out=195] (9.center) to (11.center);
		\draw (15.center) to (14.center);
		\draw (14.center) to (16.center);
		\draw (16.center) to (17.center);
		\draw [in=255, out=0] (19.center) to (15.center);
		\draw [style=morphism-edge, in=0, out=-90] (21.center) to (20.center);
		\draw [style=unit, in=75, out=-180] (24.center) to (19.center);
		\draw (17.center) to (18.center);
		\draw (24.center) to (15.center);
		\draw [in=-90, out=90] (27.center) to (18.center);
		\draw [style=Hidden, very thick, in=15, out=-195, looseness=0.75] (20.center) to (28.center);
		\draw [style=morphism-edge, in=90, out=-165, looseness=0.75] (28.center) to (22.center);
		\draw [in=0, out=-225] (33.center) to (36.center);
		\draw (37.center) to (33.center);
		\draw [style=unit, in=90, out=-165] (37.center) to (36.center);
		\draw [style=Front] (34.center)
			 to [in=90, out=-90] (35.center)
			 to [in=-120, out=0] (33.center)
			 to (29.center)
			 to (30.center)
			 to cycle;
		\draw (42.center) to (41.center);
		\draw (41.center) to (43.center);
		\draw [in=360, out=180] (43.center) to (44.center);
		\draw [in=255, out=0] (46.center) to (42.center);
		\draw [style=unit, in=75, out=-180] (50.center) to (46.center);
		\draw (44.center) to (45.center);
		\draw (50.center) to (42.center);
		\draw [in=-90, out=90] (51.center) to (45.center);
		\draw [style=morphism-edge, in=270, out=90] (48.center) to (47.center);
		\draw [style=Hidden, in=360, out=135] (15.center) to (64.center);
		\draw [in=0, out=-180] (64.center) to (27.center);
		\draw [style=Hidden, in=360, out=135] (42.center) to (65.center);
		\draw [in=0, out=-180] (65.center) to (51.center);
		\draw [style=morphism-edge, in=270, out=90] (38.center) to (31.center);
		\draw [in=240, out=0] (35.center) to (33.center);
		\draw (33.center) to (29.center);
		\draw (29.center) to (30.center);
		\draw (30.center) to (34.center);
		\draw [in=90, out=-90] (34.center) to (35.center);
		\draw [in=240, out=0] (6.center) to (4.center);
		\draw (4.center) to (0.center);
		\draw (0.center) to (1.center);
		\draw (1.center) to (5.center);
		\draw [in=90, out=-90] (5.center) to (6.center);
		\draw [style=morphism-edge, in=15, out=195] (20.center) to (25.center);
	\end{pgfonlayer}
\end{tikzpicture}
  \end{equation}
\end{example}

\begin{definition}\label{def: bimonad}
  A \emph{bimonad} on a monoidal category \(\cat{C}\) comprises
  an opmonoidal endofunctor \((B, B_2, B_0) \from \cat{C} \to \cat{C}\)
  together with opmonoidal natural transformations \(\mu \from B^2 \nt B\) and \(\eta \from \Id_{\cat{C}} \nt B\)
  implementing a monad structure on \(B\).
\end{definition}

Graphically, what it means for \(\mu\) and \(\eta\) to be opmonoidal natural transformations is:
\begin{equation}\label[diagram]{eq:mu-opmonoidal-nt}
  \begin{tikzpicture}
	\begin{pgfonlayer}{nodelayer}
		\node [style=none] (140) at (-1.75, -1.75) {};
		\node [style=none] (141) at (-4, -1.75) {};
		\node [style=none] (142) at (-1.75, 1.75) {};
		\node [style=none] (143) at (-4, 1.75) {};
		\node [style=none] (144) at (-7.25, 1) {};
		\node [style=none] (145) at (-7.75, 2.5) {};
		\node [style=none] (146) at (-7.75, -1.48) {};
		\node [style=none] (147) at (-7.25, -2.73) {};
		\node [style=none] (148) at (-4, 0.25) {};
		\node [style=none] (149) at (-2.5, 1.75) {};
		\node [style=none] (150) at (-6.5, -2.75) {};
		\node [style=none] (151) at (-6.95, -1.4) {};
		\node [style=none] (155) at (-8.25, 0) {\scriptsize $=$};
		\node [style=none] (156) at (-8.75, -1.75) {};
		\node [style=none] (157) at (-11, -1.75) {};
		\node [style=none] (158) at (-8.75, 1.75) {};
		\node [style=none] (159) at (-11, 1.75) {};
		\node [style=none] (160) at (-13.5, 1) {};
		\node [style=none] (161) at (-14, 2.5) {};
		\node [style=none] (162) at (-14, -1.48) {};
		\node [style=none] (163) at (-13.5, -2.73) {};
		\node [style=none] (164) at (-11, 0) {};
		\node [style=none] (165) at (-9.25, 1.75) {};
		\node [style=none] (166) at (-12.75, -2.75) {};
		\node [style=none] (167) at (-13, -1.37) {};
		\node [style=none] (171) at (-9.75, 0.75) {};
		\node [style=none] (172) at (-10.25, 1.75) {};
		\node [style=none] (174) at (-3.25, 1.75) {};
		\node [style=none] (176) at (-4, 0.75) {};
		\node [style=none] (177) at (-6, -0.5) {};
		\node [style=none] (178) at (-6.225, -0.35) {};
		\node [style=none] (179) at (4.75, 1.25) {};
		\node [style=none] (180) at (4.75, -1.25) {};
		\node [style=none] (181) at (2.25, -1.25) {};
		\node [style=none] (182) at (2.25, 1.25) {};
		\node [style=none] (183) at (3.5, 0.5) {};
		\node [style=none] (184) at (4, 1.25) {};
		\node [style=none] (185) at (3, 1.25) {};
		\node [style=none] (186) at (2.25, -0.25) {};
		\node [style=none] (187) at (5.25, 0) {\scriptsize $=$};
		\node [style=none] (188) at (8.25, 1.25) {};
		\node [style=none] (189) at (8.25, -1.25) {};
		\node [style=none] (190) at (5.75, -1.25) {};
		\node [style=none] (191) at (5.75, 1.25) {};
		\node [style=none] (192) at (7.5, 1.25) {};
		\node [style=none] (193) at (6.5, 1.25) {};
		\node [style=none] (194) at (5.75, -0.5) {};
		\node [style=none] (195) at (5.75, 0.25) {};
	\end{pgfonlayer}
	\begin{pgfonlayer}{edgelayer}
		\node [style=none] (153) at (-6.95, -1.68) {\tiny $B$};
		\node [style=none] (169) at (-13, -1.655) {\tiny $B$};
		\node [style=none] (168) at (-9.25, 2) {\tiny $B$};
		\node [style=none] (170) at (-12.75, -3.03) {\tiny $B$};
		\node [style=none] (152) at (-2.5, 2) {\tiny $B$};
		\node [style=none] (154) at (-6.5, -3.03) {\tiny $B$};
		\node [style=none] (175) at (-3.25, 2) {\tiny $B$};
		\node [style=none] (173) at (-10.25, 2) {\tiny $B$};
		\node [style=none] (196) at (4, 1.5) {\tiny $B$};
		\node [style=none] (197) at (3, 1.5) {\tiny $B$};
		\node [style=none] (198) at (7.5, 1.5) {\tiny $B$};
		\node [style=none] (199) at (6.5, 1.5) {\tiny $B$};
		\draw [in=-225, out=0] (146.center) to (141.center);
		\draw [in=-225, out=0] (162.center) to (157.center);
		\draw [in=0, out=-225] (143.center) to (145.center);
		\draw [style=morphism-edge, in=-30, out=165] (148.center) to (178.center);
		\draw [style=morphism-edge, in=135, out=165] (176.center) to (178.center);
		\draw [style=morphism-edge, in=90, out=-135] (178.center) to (151.center);
		\draw [style=Front] (147.center)
			 to (144.center)
			 to [in=-135, out=0] (143.center)
			 to (141.center)
			 to [in=0, out=-135] cycle;
		\draw (141.center) to (140.center);
		\draw (140.center) to (142.center);
		\draw (142.center) to (143.center);
		\draw (146.center) to (145.center);
		\draw (144.center) to (147.center);
		\draw (143.center) to (141.center);
		\draw [in=0, out=225] (143.center) to (144.center);
		\draw [in=225, out=0] (147.center) to (141.center);
		\draw [style=morphism-edge, in=0, out=-90] (149.center) to (148.center);
		\draw [style=morphism-edge, in=90, out=-180] (164.center) to (167.center);
		\draw [in=0, out=-225] (159.center) to (161.center);
		\draw [style=Front] (163.center)
			 to (160.center)
			 to [in=-135, out=0] (159.center)
			 to (157.center)
			 to [in=0, out=-120] cycle;
		\draw (157.center) to (156.center);
		\draw (156.center) to (158.center);
		\draw (158.center) to (159.center);
		\draw (162.center) to (161.center);
		\draw (160.center) to (163.center);
		\draw (159.center) to (157.center);
		\draw [in=0, out=225] (159.center) to (160.center);
		\draw [in=240, out=0] (163.center) to (157.center);
		\draw [style=morphism-edge, in=90, out=195] (164.center) to (166.center);
		\draw [style=morphism-edge, in=180, out=-90] (172.center) to (171.center);
		\draw [style=morphism-edge, in=0, out=-90] (165.center) to (171.center);
		\draw [style=morphism-edge, in=15, out=-90] (171.center) to (164.center);
		\draw [style=morphism-edge, in=0, out=-90] (174.center) to (176.center);
		\draw [style=morphism-edge, in=105, out=180] (176.center) to (177.center);
		\draw [style=morphism-edge, in=-60, out=180] (148.center) to (177.center);
		\draw [style=morphism-edge, in=90, out=-150] (177.center) to (150.center);
		\draw (179.center) to (182.center);
		\draw (179.center) to (180.center);
		\draw (180.center) to (181.center);
		\draw [style=unit] (181.center) to (182.center);
		\draw [style=morphism-edge, in=0, out=-90] (184.center) to (183.center);
		\draw [style=morphism-edge, in=180, out=-90] (185.center) to (183.center);
		\draw [style=morphism-edge, in=0, out=-90] (183.center) to (186.center);
		\draw (188.center) to (191.center);
		\draw (188.center) to (189.center);
		\draw (189.center) to (190.center);
		\draw [style=unit] (190.center) to (191.center);
		\draw [style=morphism-edge, in=0, out=-90] (193.center) to (195.center);
		\draw [style=morphism-edge, in=0, out=-90] (192.center) to (194.center);
	\end{pgfonlayer}
\end{tikzpicture}
\end{equation}
\begin{equation}\label[diagram]{eq:eta-opmonoidal-nt}
  \begin{tikzpicture}
	\begin{pgfonlayer}{nodelayer}
		\node [style=none] (21) at (4.25, 1.25) {};
		\node [style=none] (22) at (4.25, -1.25) {};
		\node [style=none] (23) at (1.75, -1.25) {};
		\node [style=none] (24) at (1.75, 1.25) {};
		\node [style=none] (25) at (1.75, -0.25) {};
		\node [style=none] (26) at (4.75, 0) {\scriptsize $=$};
		\node [style=none] (27) at (7.75, 1.25) {};
		\node [style=none] (28) at (7.75, -1.25) {};
		\node [style=none] (29) at (5.25, -1.25) {};
		\node [style=none] (30) at (5.25, 1.25) {};
		\node [style=blackdot] (31) at (3, 0.5) {};
		\node [style=none] (32) at (-7, -1.75) {};
		\node [style=none] (33) at (-9, -1.75) {};
		\node [style=none] (34) at (-7, 1.75) {};
		\node [style=none] (35) at (-9, 1.75) {};
		\node [style=none] (36) at (-11, 1) {};
		\node [style=none] (37) at (-11.75, 2.5) {};
		\node [style=none] (38) at (-11.75, -1.5) {};
		\node [style=none] (39) at (-11, -2.75) {};
		\node [style=none] (40) at (-10.25, -2.675) {};
		\node [style=none] (41) at (-11.425, -1.5) {};
		\node [style=none] (42) at (-6.5, 0) {\scriptsize $=$};
		\node [style=none] (43) at (-11.425, -1.75) {\tiny $B$};
		\node [style=none] (44) at (-10.25, -2.925) {\tiny $B$};
		\node [style=none] (45) at (-2.25, -1.75) {};
		\node [style=none] (46) at (-3.5, -1.75) {};
		\node [style=none] (47) at (-2.25, 1.75) {};
		\node [style=none] (48) at (-3.5, 1.75) {};
		\node [style=none] (49) at (-5.25, 1) {};
		\node [style=none] (50) at (-6, 2.5) {};
		\node [style=none] (51) at (-6, -0.75) {};
		\node [style=none] (52) at (-5.25, -2.25) {};
		\node [style=none] (53) at (-4.5, -2.225) {};
		\node [style=none] (54) at (-5.675, -0.75) {};
		\node [style=none] (55) at (-5.675, -1) {\tiny $B$};
		\node [style=none] (56) at (-4.5, -2.475) {\tiny $B$};
		\node [style=blackdot] (57) at (-8, 0.75) {};
		\node [style=none] (58) at (-9, 0) {};
		\node [style=blackdot] (59) at (-5.675, 0) {};
		\node [style=blackdot] (60) at (-4.5, -1.25) {};
	\end{pgfonlayer}
	\begin{pgfonlayer}{edgelayer}
		\draw (21.center) to (24.center);
		\draw (21.center) to (22.center);
		\draw (22.center) to (23.center);
		\draw [style=unit] (23.center) to (24.center);
		\draw (27.center) to (30.center);
		\draw (27.center) to (28.center);
		\draw (28.center) to (29.center);
		\draw [style=unit] (29.center) to (30.center);
		\draw [style=morphism-edge, in=0, out=-90] (31) to (25.center);
		\draw [in=-240, out=0] (51.center) to (46.center);
		\draw [in=-210, out=0] (38.center) to (33.center);
		\draw [in=0, out=-225] (35.center) to (37.center);
		\draw [style=morphism-edge, in=90, out=-180] (58.center) to (41.center);
		\draw [style=Front] (39.center)
			 to (36.center)
			 to [in=-135, out=0] (35.center)
			 to (33.center)
			 to [in=0, out=-120] cycle;
		\draw (33.center) to (32.center);
		\draw (32.center) to (34.center);
		\draw (34.center) to (35.center);
		\draw (38.center) to (37.center);
		\draw (36.center) to (39.center);
		\draw (35.center) to (33.center);
		\draw [in=0, out=225] (35.center) to (36.center);
		\draw [in=240, out=0] (39.center) to (33.center);
		\draw [in=0, out=-225] (48.center) to (50.center);
		\draw [style=Front] (52.center)
			 to (49.center)
			 to [in=-135, out=0] (48.center)
			 to (46.center)
			 to [in=0, out=-135] cycle;
		\draw (46.center) to (45.center);
		\draw (45.center) to (47.center);
		\draw (47.center) to (48.center);
		\draw (51.center) to (50.center);
		\draw (49.center) to (52.center);
		\draw (48.center) to (46.center);
		\draw [in=0, out=225] (48.center) to (49.center);
		\draw [in=225, out=0] (52.center) to (46.center);
		\draw [style=morphism-edge, in=0, out=-90] (57) to (58.center);
		\draw [style=morphism-edge, in=90, out=195] (58.center) to (40.center);
		\draw [style=morphism-edge, in=90, out=270] (59) to (54.center);
		\draw [style=morphism-edge, in=90, out=270] (60) to (53.center);
	\end{pgfonlayer}
\end{tikzpicture}
\end{equation}

Let \stdadj{} be an opmonoidal adjunction between \(\cat{C}\) and \(\cat{D}\).
The monad of the adjunction \(UF \from \cat{C} \to \cat{C}\) is a bimonad;
its comultiplication is defined  as the composition
\[
  UF(\blank \otimes \bblank) \xrightarrow{UF_{2;\blank,\bblank}} U(F\blank\otimes F\bblank) \xrightarrow{U_{2; F\blank,F\bblank}} UF\blank \otimes UF\bblank,
\]
and its counit is given by
\[
  UF1 \xrightarrow{U F_0} U1 \xrightarrow{\;U_0\;} 1.
\]

Adjunctions between monoidal categories are a broad topic with many facets,
see for example~\cite[Chapter~3]{Aguiar2010}.
For our purposes, we can restrict ourselves to the following situation.

\begin{definition}\label{def: monoidal_adjunction}
  We call an adjunction \(\stdadj\) between monoidal categories \(\cat{C}\) and \(\cat{D}\) \emph{opmonoidal} if
  \(F\) and \(U\) are opmonoidal functors, and
  the unit and counit of the adjunction are opmonoidal natural transformations.
  If \(F\) and \(U\) are moreover strong monoidal,
  we call \(\stdadj\) a \emph{(strong) monoidal adjunction}.
\end{definition}

Suppose \(\stdadj\) to be an opmonoidal adjunction.
In our graphical language, the conditions that the unit \(\eta\) and counit \(\varepsilon\) need to be opmonoidal natural transformations take the following form.
\begin{equation}\label[diagram]{eq:unit-adjunction-opmonoidal-nt}
  \begin{tikzpicture}
	\begin{pgfonlayer}{nodelayer}
		\node [style=none] (0) at (-2, -1.75) {};
		\node [style=none] (1) at (-3, -1.75) {};
		\node [style=none] (2) at (-2, 1.75) {};
		\node [style=none] (3) at (-3, 1.75) {};
		\node [style=none] (4) at (-5.75, 1) {};
		\node [style=none] (5) at (-6.5, 2.5) {};
		\node [style=none] (6) at (-6.5, -1.25) {};
		\node [style=none] (7) at (-5.75, -2.5) {};
		\node [style=none] (8) at (-3.75, -2.28) {};
		\node [style=none] (9) at (-5.25, -1.175) {};
		\node [style=none] (12) at (-7, 0) {\scriptsize $=$};
		\node [style=none] (13) at (-7.5, -1.75) {};
		\node [style=none] (14) at (-9.75, -1.75) {};
		\node [style=none] (15) at (-7.5, 1.75) {};
		\node [style=none] (16) at (-9.75, 1.75) {};
		\node [style=none] (17) at (-12.5, 1) {};
		\node [style=none] (18) at (-13.25, 2.5) {};
		\node [style=none] (19) at (-13.25, -1.48) {};
		\node [style=none] (20) at (-12.5, -2.73) {};
		\node [style=none] (21) at (-9.75, -0.5) {};
		\node [style=none] (22) at (-11.75, -2.75) {};
		\node [style=none] (23) at (-12.9, -1.455) {};
		\node [style=none] (26) at (-4.25, -0.7) {};
		\node [style=none] (27) at (-8.5, 1.25) {};
		\node [style=none] (28) at (-8, 0.75) {};
		\node [style=none] (29) at (-9, 0.75) {};
		\node [style=none] (30) at (-9.75, 0.25) {};
		\node [style=none] (31) at (-10.99, -2.74) {};
		\node [style=none] (33) at (-12, -1.41) {};
		\node [style=none] (35) at (-4.75, -2.5) {};
		\node [style=none] (37) at (-4.25, -1.135) {};
		\node [style=none] (39) at (-4.75, 0.225) {};
		\node [style=none] (49) at (2, -1.5) {};
		\node [style=none] (50) at (2, 1.5) {};
		\node [style=none] (51) at (4.5, 1.5) {};
		\node [style=none] (52) at (4.5, -1.5) {};
		\node [style=none] (53) at (4.5, 0.5) {};
		\node [style=none] (54) at (5, 0) {\scriptsize $=$};
		\node [style=none] (55) at (5.5, -1.5) {};
		\node [style=none] (56) at (5.5, 1.5) {};
		\node [style=none] (57) at (7.75, 1.5) {};
		\node [style=none] (58) at (7.75, -1.5) {};
		\node [style=none] (59) at (3.25, 1.25) {};
		\node [style=none] (60) at (3.75, 0.75) {};
		\node [style=none] (61) at (2.75, 0.75) {};
		\node [style=none] (62) at (2, 0.25) {};
		\node [style=none] (63) at (2, -0.75) {};
	\end{pgfonlayer}
	\begin{pgfonlayer}{edgelayer}
		\node [style=none] (10) at (-5.25, -1.475) {\tiny $F$};
		\node [style=none] (38) at (-4.25, -1.425) {\tiny $U$};
		\node [style=none] (34) at (-12, -1.68) {\tiny $U$};
		\node [style=none] (36) at (-4.75, -2.81) {\tiny $F$};
		\node [style=none] (32) at (-11, -3.03) {\tiny $U$};
		\node [style=none] (24) at (-12.9, -1.755) {\tiny $F$};
		\node [style=none] (25) at (-11.75, -3.03) {\tiny $F$};
		\node [style=none] (11) at (-3.75, -2.6) {\tiny $U$};
		\draw [in=-210, out=0] (19.center) to (14.center);
		\draw [in=-225, out=0] (6.center) to (1.center);
		\draw [in=0, out=-225] (3.center) to (5.center);
		\draw [style=morphism-edge] (9.center)
			 to [in=180, out=90] (39.center)
			 to [in=90, out=0] (37.center);
		\draw [style=Front] (7.center)
			 to (4.center)
			 to [in=-135, out=0] (3.center)
			 to (1.center)
			 to [in=0, out=-135] cycle;
		\draw (1.center) to (0.center);
		\draw (0.center) to (2.center);
		\draw (2.center) to (3.center);
		\draw (6.center) to (5.center);
		\draw (4.center) to (7.center);
		\draw (3.center) to (1.center);
		\draw [in=0, out=225] (3.center) to (4.center);
		\draw [in=225, out=0] (7.center) to (1.center);
		\draw [in=0, out=-225] (16.center) to (18.center);
		\draw [style=morphism-edge, in=90, out=180] (30.center) to (23.center);
		\draw [style=morphism-edge, in=90, out=180] (21.center) to (33.center);
		\draw [style=Front] (20.center)
			 to (17.center)
			 to [in=-135, out=0] (16.center)
			 to (14.center)
			 to [in=0, out=-120] cycle;
		\draw (14.center) to (13.center);
		\draw (13.center) to (15.center);
		\draw (15.center) to (16.center);
		\draw (19.center) to (18.center);
		\draw (17.center) to (20.center);
		\draw (16.center) to (14.center);
		\draw [in=0, out=225] (16.center) to (17.center);
		\draw [in=240, out=0] (20.center) to (14.center);
		\draw [style=morphism-edge] (35.center)
			 to [in=180, out=90] (26.center)
			 to [in=90, out=0] (8.center);
		\draw [style=morphism-edge] (30.center)
			 to [in=-105, out=0] (29.center)
			 to [in=180, out=75] (27.center)
			 to [in=90, out=0] (28.center)
			 to [in=0, out=-90] (21.center);
		\draw [style=morphism-edge, in=90, out=195] (30.center) to (22.center);
		\draw [style=morphism-edge, in=90, out=195] (21.center) to (31.center);
		\draw (49.center) to (52.center);
		\draw [style=unit] (49.center) to (50.center);
		\draw (50.center) to (51.center);
		\draw (51.center) to (52.center);
		\draw (55.center) to (58.center);
		\draw [style=unit] (55.center) to (56.center);
		\draw (56.center) to (57.center);
		\draw (57.center) to (58.center);
		\draw [style=morphism-edge, in=90, out=180] (59.center) to (61.center);
		\draw [style=morphism-edge, in=90, out=0] (59.center) to (60.center);
		\draw [style=morphism-edge, in=0, out=-90] (61.center) to (62.center);
		\draw [style=morphism-edge, in=0, out=-90] (60.center) to (63.center);
	\end{pgfonlayer}
\end{tikzpicture}
\end{equation}
\begin{equation}\label[diagram]{eq:counit-adjunction-opmonoidal-nt}
  \begin{tikzpicture}
	\begin{pgfonlayer}{nodelayer}
		\node [style=none] (79) at (-2, -1.75) {};
		\node [style=none] (80) at (-4.5, -1.75) {};
		\node [style=none] (81) at (-2, 1.75) {};
		\node [style=none] (82) at (-4.5, 1.75) {};
		\node [style=none] (83) at (-6.75, 1) {};
		\node [style=none] (84) at (-7.5, 2.5) {};
		\node [style=none] (85) at (-7.5, -1.25) {};
		\node [style=none] (86) at (-6.75, -2.5) {};
		\node [style=none] (87) at (-2.75, 1.75) {};
		\node [style=none] (89) at (-8, 0) {\scriptsize $=$};
		\node [style=none] (90) at (-8.5, -1.75) {};
		\node [style=none] (91) at (-11, -1.75) {};
		\node [style=none] (92) at (-8.5, 1.75) {};
		\node [style=none] (93) at (-11, 1.75) {};
		\node [style=none] (94) at (-12.5, 1) {};
		\node [style=none] (95) at (-13.25, 2.5) {};
		\node [style=none] (96) at (-13.25, -1.25) {};
		\node [style=none] (97) at (-12.5, -2.5) {};
		\node [style=none] (98) at (-9.25, 1.75) {};
		\node [style=none] (100) at (-9.75, -0.75) {};
		\node [style=none] (101) at (-10.25, 1.75) {};
		\node [style=none] (103) at (-3.75, 1.75) {};
		\node [style=none] (105) at (-4.5, 0) {};
		\node [style=none] (106) at (-4.5, 0.75) {};
		\node [style=none] (107) at (-5.75, -0.75) {};
		\node [style=none] (108) at (-6.25, -0.25) {};
		\node [style=none] (115) at (2, -1.25) {};
		\node [style=none] (116) at (2, 1.25) {};
		\node [style=none] (117) at (4.5, 1.25) {};
		\node [style=none] (118) at (4.5, -1.25) {};
		\node [style=none] (119) at (4.5, 0.25) {};
		\node [style=none] (120) at (5, 0) {\scriptsize $=$};
		\node [style=none] (121) at (5.5, -1.25) {};
		\node [style=none] (122) at (5.5, 1.25) {};
		\node [style=none] (123) at (7.75, 1.25) {};
		\node [style=none] (124) at (7.75, -1.25) {};
		\node [style=none] (125) at (3.75, 1.25) {};
		\node [style=none] (126) at (2.75, 1.25) {};
		\node [style=none] (129) at (3.25, 0.25) {};
		\node [style=none] (130) at (7, 1.25) {};
		\node [style=none] (131) at (6.25, 1.25) {};
		\node [style=none] (132) at (5.5, -0.25) {};
		\node [style=none] (133) at (5.5, 0.5) {};
	\end{pgfonlayer}
	\begin{pgfonlayer}{edgelayer}
		\node [style=none] (169) at (7, 1.5) {\tiny $F$};
		\node [style=none] (170) at (6.25, 1.5) {\tiny $U$};
		\node [style=none] (127) at (3.75, 1.5) {\tiny $F$};
		\node [style=none] (128) at (2.75, 1.5) {\tiny $U$};
		\node [style=none] (104) at (-3.75, 2) {\tiny $U$};
		\node [style=none] (102) at (-10.25, 2) {\tiny $U$};
		\node [style=none] (99) at (-9.25, 2) {\tiny $F$};
		\node [style=none] (88) at (-2.75, 2) {\tiny $F$};
		\draw [in=-210, out=0] (96.center) to (91.center);
		\draw [in=-225, out=0] (85.center) to (80.center);
		\draw [in=0, out=-225] (82.center) to (84.center);
		\draw [style=morphism-edge] (105.center)
			 to [in=-30, out=165, looseness=1.25] (108.center)
			 to [in=180, out=135] (106.center);
		\draw [style=Front] (86.center)
			 to (83.center)
			 to [in=-135, out=0] (82.center)
			 to (80.center)
			 to [in=0, out=-135] cycle;
		\draw (80.center) to (79.center);
		\draw (79.center) to (81.center);
		\draw (81.center) to (82.center);
		\draw (85.center) to (84.center);
		\draw (83.center) to (86.center);
		\draw (82.center) to (80.center);
		\draw [in=0, out=225] (82.center) to (83.center);
		\draw [in=225, out=0] (86.center) to (80.center);
		\draw [in=0, out=-225] (93.center) to (95.center);
		\draw [style=Front] (97.center)
			 to (94.center)
			 to [in=-120, out=0] (93.center)
			 to (91.center)
			 to [in=0, out=-135] cycle;
		\draw (91.center) to (90.center);
		\draw (90.center) to (92.center);
		\draw (92.center) to (93.center);
		\draw (96.center) to (95.center);
		\draw (94.center) to (97.center);
		\draw (93.center) to (91.center);
		\draw [in=0, out=240] (93.center) to (94.center);
		\draw [in=225, out=0] (97.center) to (91.center);
		\draw [style=morphism-edge] (98.center)
			 to [in=0, out=-90] (100.center)
			 to [in=-90, out=180] (101.center);
		\draw [style=morphism-edge, in=0, out=-90] (103.center) to (106.center);
		\draw [style=morphism-edge, in=0, out=-90] (87.center) to (105.center);
		\draw [style=morphism-edge] (105.center)
			 to [in=-40, out=-165] (107.center)
			 to [in=-150, out=138] (106.center);
		\draw (115.center) to (118.center);
		\draw [style=unit] (115.center) to (116.center);
		\draw (116.center) to (117.center);
		\draw (117.center) to (118.center);
		\draw (121.center) to (124.center);
		\draw [style=unit] (121.center) to (122.center);
		\draw (122.center) to (123.center);
		\draw (123.center) to (124.center);
		\draw [style=morphism-edge, in=0, out=-90] (125.center) to (129.center);
		\draw [style=morphism-edge, in=180, out=-90] (126.center) to (129.center);
		\draw [style=morphism-edge, in=0, out=-90] (131.center) to (133.center);
		\draw [style=morphism-edge, in=0, out=-90] (130.center) to (132.center);
	\end{pgfonlayer}
\end{tikzpicture}
\end{equation}

The next result is a slightly simplified version of~\cite[Lemma~7.10]{Turaev2017}.

\begin{lemma}\label{lem:monoidal-adjunctions-vs-bimonads}
  Let \stdadj{} be a pair of adjoint functors between two monoidal categories.
  The adjunction \(F \adjoint U\) is monoidal if and only if\, \(U\) is a strong monoidal functor.
  That is, the coherence morphisms of\, \(U\) are invertible.
\end{lemma}

Let \(B \from \cat{C} \to \cat{C}\) be the bimonad arising from the monoidal adjunction \stdadj.
One can show that $\cat{C}^B$ carries a monoidal structure such that the forgetful functor \(U^B \from \cat{C}^B \to \cat{C}\) is strict monoidal, see~\cite{Moerdijk2002, zawadowski2012:FormalTheoryMonoidalMonads, böhm2019:FormalTheoryMultimonoidalMonads}.
By the above lemma, the adjunction \(F^B \adjoint U^B\) is monoidal.
This raises the question whether the comparison functor---%
mediating between the two adjunctions---%
is compatible with this additional structure.
Due to~\cite{kelly74}, see also~\cite[Theorem~2.6]{Bruguieres2007}, we have the following result.

\begin{lemma}\label{lem:comparison-functor-bimonad-monoidal}
  Let \stdadj{} be a monoidal adjunction and write \(B \from \cat{C} \to \cat{C}\) for its induced bimonad.
  The comparison functor \(\Sigma \from \cat{D} \to \cat{C}^B\) is strong monoidal and
  \(U^B \Sigma = U\) as well as \(\Sigma F = F^B\) as strong and opmonoidal functors, respectively.
\end{lemma}
For a more general version see~\cite[Proposition~3.2]{böhm2019:FormalTheoryMultimonoidalMonads}.

\section{Comodule monads}\label{sec:comodule-monads}

\noindent Monads with a ``coaction'' over a bimonad were defined and studied by Aguiar and Chase in~\cite{Aguiar2012}.
This concept is needed to obtain an adequate monadic interpretation of twisted centres.
We briefly summarise the aspects of the aforementioned article that are needed for our investigation.\footnote{%
  \,We slightly deviate from~\cite{Aguiar2012} in that we study right comodule monads as opposed to their left versions.%
}

\begin{definition}\label{def:right-module-category}
  Let \(\cat{C}\) be a monoidal category.
  A \emph{right \(\cat{C}\)-module category} comprises
  a category \(\cat{M}\)
  together with an \emph{action} functor \(\ract\from \cat{M} \times \cat{C} \to \cat{M}\),
  satisfying appropriate \emph{associativity} and \emph{unitality} conditions,
  see for example~\cite[Definition~7.1.1]{Etingof2015}.
\end{definition}

Every monoidal category \(\cat{C}\) is a right \(\cat{C}\)-module category, where \(\otimes \eqdef \ract\).

\begin{example}\label{ex:module-twisting}
  Let \(\cat{C}\) be a monoidal category,
  and suppose that \(F \from \cat{C} \to \cat{C}\) is a strong monoidal functor.
  The action
  \(\ract \from \cat{C} \times \cat{C} \xrightarrow{\raisebox{-0.4em}{\tiny\(\cat{C} \times F\)}} \cat{C} \times \cat{C} \xrightarrow{\raisebox{-0.4em}{\tiny\(\;\;\otimes\;\;\)}} \cat{C}\)
  endows \(\cat{C}\) with the structure of a right module category over itself.
\end{example}

To keep our notation concise,
for the rest of this article
we fix two monoidal categories \((\cat{C}, \otimes, 1)\) and \((\cat{D}, \otimes, 1)\)
and over each a right module category \((\cat{M}, \ract)\) and \((\cat{N}, \bract)\).

\begin{definition}\label{def: comodule-functor}
  Let \(F \from \cat{C} \to \cat{D}\) be an opmonoidal functor.
  A \emph{(right) comodule functor over \(F\)} is a pair \((G, \delta)\) comprising
  a functor \(G \from \cat{M} \to \cat{N}\) together with
  a natural transformation \(\delta \from G(\blank \ract \bblank) \nt G\blank \bract F\bblank\),
  called the \emph{coaction} of \(G\), which is \emph{coassociative} and \emph{counital}.

  A comodule functor is called \emph{strong} if its coaction is an isomorphism.
\end{definition}

\begin{example}\label{ex:egno-module-functors}
  A \emph{right \(\cat{C}\)-module functor}
  in the sense of~\cite[Definition~7.2.1]{Etingof2015}
  comprises a functor \(F\from \cat{M} \to \cat{N}\) between two \(\cat{C}\)-module categories,
  and a natural transformation \(\delta \from F(\blank \otimes \bblank) \nt F\blank \otimes \bblank\)
  whose compatibility with the action equate to  \(F\) being a comodule functor over \(\Id_{\cat{C}}\from \cat{C} \to \cat{C}\).
\end{example}

\begin{example}[{\cite[Section~6.1]{Aguiar2012}}]\label{ex:comodule-algebra}
  If \(B\) is a bialgebra over a commutative ring \(\k\),
  then \(B \otimes_{\k} \blank \from \k\text{-Mod} \to \k\text{-Mod}\) becomes an opmonoidal functor.
  Let \(A\) be a right \(B\)-comodule algebra,
  and suppose  \(C\) is a \(\k\)-subalgebra of the \(B\)-coinvariants of \(A\).
  The coaction \(\nu \from A \to A \otimes_{\k} B\)
  is a map of \(C\)-\(C\)-bimodules,
  which turns
  \(A \otimes_{C} \blank \from C\text{-Mod} \to C\text{-Mod}\) into a right comodule functor over \(B \otimes_{\k} \blank\).
  The action of \(C\text{-Mod}\) on the category of \(\k\)-modules is given by tensoring over \(\k\).
\end{example}

The string diagrammatic calculus of \cref{sec:bimonads} can be adapted to this setting as follows:
in addition to combining sheets using tensor products,
we now additionally consider actions to do so as well.
In order to keep track of which splitting occurred,
functors between the module categories will be coloured blue.
As an example, consider the coaction \(\delta \from G(\blank \ract \bblank) \nt G\blank \bract F\bblank\) from \cref{def: comodule-functor}.
It is represented by
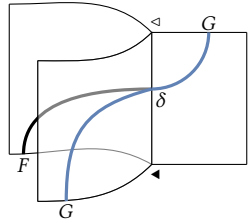
\[
  \begin{tikzpicture}
	\begin{pgfonlayer}{nodelayer}
		\node [style=none] (0) at (2.5, -2) {};
		\node [style=none] (1) at (0, -2) {};
		\node [style=none] (2) at (2.5, 1.5) {};
		\node [style=none] (3) at (0, 1.5) {};
		\node [style=none] (4) at (-3, 0.75) {};
		\node [style=none] (5) at (-3.75, 2.25) {};
		\node [style=none] (6) at (-3.75, -1.73) {};
		\node [style=none] (7) at (-3, -2.98) {};
		\node [style=none] (8) at (0, 0) {};
		\node [style=none] (9) at (1.5, 1.5) {};
		\node [style=none] (11) at (-2.25, -2.975) {};
		\node [style=none] (12) at (-3.375, -1.73) {};
		\node [style=none] (21) at (1.5, 1.75) {\tiny $G$};
		\node [style=none] (22) at (-3.375, -2.03) {\tiny $F$};
		\node [style=none] (23) at (-2.25, -3.255) {\tiny $G$};
		\node [style=none] (24) at (0.25, -0.3) {\tiny $\delta$};
		\node [style=none] (25) at (0.1, -2.275) {\tiny $\blacktriangleleft$};
    \node [style=none] (26) at (0.1, 1.775) {\tiny $\ract$};
	\end{pgfonlayer}
	\begin{pgfonlayer}{edgelayer}
		\draw [in=-210, out=0] (6.center) to (1.center);
		\draw [in=0, out=-225] (3.center) to (5.center);
		\draw [style=morphism-edge, in=90, out=-180] (8.center) to (12.center);
		\draw [style=Front] (7.center)
			 to (4.center)
			 to [in=-135, out=0] (3.center)
			 to (1.center)
			 to [in=0, out=-135] cycle;
		\draw (1.center) to (0.center);
		\draw (0.center) to (2.center);
		\draw (2.center) to (3.center);
		\draw (6.center) to (5.center);
		\draw (4.center) to (7.center);
		\draw (3.center) to (1.center);
		\draw [in=0, out=225] (3.center) to (4.center);
		\draw [in=225, out=0] (7.center) to (1.center);
		\draw [style=comodule-edge, in=0, out=-90] (9.center) to (8.center);
		\draw [style=comodule-edge, in=90, out=195, looseness=1.25] (8.center) to (11.center);
	\end{pgfonlayer}
\end{tikzpicture}
\]
The mentioned compatibility of the coaction with the comultiplication and counit of \(F\)
result in diagrams analogous
to~\eqref{eq:coass-of-opmonoidal-functor}
and~\eqref{eq:counit-of-opmonoidal-functor}%
---except that one of the strings will be in a different colour.
\begin{equation}\label[diagram]{eq:coassoc-of-comodule-monad}
  \begin{tikzpicture}
	\begin{pgfonlayer}{nodelayer}
		\node [style=none] (0) at (4.55, 0.2) {};
		\node [style=none] (1) at (-0.5, 2) {};
		\node [style=none] (2) at (-2.75, 2) {};
		\node [style=none] (3) at (-0.5, -2) {};
		\node [style=none] (4) at (-4.75, 2.75) {};
		\node [style=none] (5) at (-7, 1.25) {};
		\node [style=none] (6) at (-7.75, 2.25) {};
		\node [style=none] (7) at (-2.75, -2) {};
		\node [style=none] (8) at (-8.5, -1) {};
		\node [style=none] (9) at (-5, -2.75) {};
		\node [style=none] (10) at (-7.75, -2.25) {};
		\node [style=none] (11) at (-7, -3.75) {};
		\node [style=none] (12) at (-3.75, -0.75) {};
		\node [style=none] (13) at (-8.5, 3.5) {};
		\node [style=none] (14) at (-1.25, 2) {};
		\node [style=none] (15) at (0, 0) {\scriptsize $=$};
		\node [style=none] (16) at (-1.25, 2.25) {\tiny $G$};
		\node [style=none] (17) at (-2.725, 1.25) {};
		\node [style=none] (18) at (-6.5, -3.75) {};
		\node [style=none] (19) at (-4.375, 0.65) {};
		\node [style=none] (20) at (-7.25, -2.25) {};
		\node [style=none] (21) at (-8, -1) {};
		\node [style=none] (22) at (-6.5, -4) {\tiny $G$};
		\node [style=none] (23) at (-7.25, -2.5) {\tiny $F$};
		\node [style=none] (24) at (-8, -1.25) {\tiny $F$};
		\node [style=none] (25) at (8.5, 2) {};
		\node [style=none] (26) at (6.25, 2) {};
		\node [style=none] (27) at (8.5, -2) {};
		\node [style=none] (28) at (4.25, 2.75) {};
		\node [style=none] (29) at (2, 1.25) {};
		\node [style=none] (30) at (1.25, 2.25) {};
		\node [style=none] (31) at (6.25, -2) {};
		\node [style=none] (32) at (0.5, -0.75) {};
		\node [style=none] (33) at (4, -2.75) {};
		\node [style=none] (34) at (1.25, -2.25) {};
		\node [style=none] (35) at (2, -3.75) {};
		\node [style=none] (36) at (5.25, 1) {};
		\node [style=none] (37) at (0.5, 3.5) {};
		\node [style=none] (38) at (7.75, 2) {};
		\node [style=none] (39) at (7.75, 2.25) {\tiny $G$};
		\node [style=none] (40) at (2.5, -3.75) {};
		\node [style=none] (41) at (1.75, -2.25) {};
		\node [style=none] (42) at (1, -0.75) {};
		\node [style=none] (43) at (2.5, -4) {\tiny $G$};
		\node [style=none] (44) at (1.75, -2.5) {\tiny $F$};
		\node [style=none] (45) at (1, -1) {\tiny $F$};
		\node [style=none] (46) at (4.4, -0.5) {};
		\node [style=none] (47) at (5.675, 0.5) {};
	\end{pgfonlayer}
	\begin{pgfonlayer}{edgelayer}
		\draw [in=135, out=-60] (28.center) to (36.center);
		\draw [in=135, out=-60] (4.center) to (12.center);
		\draw [style=morphism-edge, in=180, out=90] (42.center) to (0.center);
		\draw [style=morphism-edge, in=30, out=-165] (17.center) to (19.center);
		\draw [style=morphism-edge, in=90, out=-180] (19.center) to (21.center);
		\draw [in=-240, out=0, looseness=0.75] (8.center) to (7.center);
		\draw [in=0, out=-225] (9.center) to (10.center);
		\draw (2.center) to (1.center);
		\draw (1.center) to (3.center);
		\draw [style=Front] (9.center)
			 to [in=-120, out=0] (7.center)
			 to [in=315, out=90] (12.center)
			 to [in=-60, out=135] (4.center)
			 to [in=0, out=-135, looseness=0.75] (6.center)
			 to (10.center)
			 to [in=135, out=0] cycle;
		\draw [style=Front, in=-150, out=60] (9.center) to (12.center);
		\draw [in=-225, out=0] (4.center) to (2.center);
		\draw [style=morphism-edge, in=90, out=-165] (19.center) to (20.center);
		\draw [style=Front] (5.center)
			 to [in=-135, out=0, looseness=0.75] (2.center)
			 to [in=15, out=-90] (12.center)
			 to [in=90, out=-45] (7.center)
			 to [in=0, out=-120] (9.center)
			 to [in=0, out=-105] (11.center)
			 to cycle;
		\draw (3.center) to (7.center);
		\draw (13.center) to (8.center);
		\draw [in=60, out=-150] (12.center) to (9.center);
		\draw [in=360, out=135] (4.center) to (13.center);
		\draw [in=240, out=0] (9.center) to (7.center);
		\draw [in=0, out=-105] (9.center) to (11.center);
		\draw (11.center) to (5.center);
		\draw [in=225, out=0, looseness=0.75] (5.center) to (2.center);
		\draw [in=15, out=-90] (2.center) to (12.center);
		\draw [in=90, out=-45] (12.center) to (7.center);
		\draw [style=comodule-edge, in=270, out=15] (17.center) to (14.center);
		\draw [style=comodule-edge, in=90, out=-150] (17.center) to (18.center);
		\draw [in=-240, out=0, looseness=0.75] (32.center) to (31.center);
		\draw [in=0, out=-225] (33.center) to (34.center);
		\draw (26.center) to (25.center);
		\draw (25.center) to (27.center);
		\draw [style=Front] (33.center)
			 to [in=-120, out=0] (31.center)
			 to [in=315, out=90] (36.center)
			 to [in=-60, out=135] (28.center)
			 to [in=0, out=-135, looseness=0.75] (30.center)
			 to (34.center)
			 to [in=135, out=0] cycle;
		\draw [style=Front, in=-150, out=60] (33.center) to (36.center);
		\draw [in=-225, out=0] (28.center) to (26.center);
		\draw [style=morphism-edge, in=195, out=90] (41.center) to (46.center);
		\draw [style=morphism-edge, in=195, out=0] (0.center) to (47.center);
		\draw [style=Front] (29.center)
			 to [in=-135, out=0, looseness=0.75] (26.center)
			 to [in=15, out=-90] (36.center)
			 to [in=90, out=-45] (31.center)
			 to [in=0, out=-120] (33.center)
			 to [in=0, out=-105] (35.center)
			 to cycle;
		\draw (27.center) to (31.center);
		\draw (37.center) to (32.center);
		\draw [in=60, out=-150] (36.center) to (33.center);
		\draw [in=360, out=135] (28.center) to (37.center);
		\draw [in=240, out=0] (33.center) to (31.center);
		\draw [in=0, out=-105] (33.center) to (35.center);
		\draw (35.center) to (29.center);
		\draw [in=225, out=0, looseness=0.75] (29.center) to (26.center);
		\draw [in=15, out=-90] (26.center) to (36.center);
		\draw [in=90, out=-45] (36.center) to (31.center);
		\draw [style=comodule-edge, in=210, out=90] (40.center) to (46.center);
		\draw [style=comodule-edge, in=30, out=-90] (38.center) to (47.center);
		\draw [style=comodule-edge, in=210, out=30] (46.center) to (47.center);
		\draw [in=0, out=-135, looseness=0.75] (4.center) to (6.center);
		\draw [in=0, out=-135, looseness=0.75] (28.center) to (30.center);
		\draw (6.center) to (10.center);
		\draw (30.center) to (34.center);
	\end{pgfonlayer}
\end{tikzpicture}
\end{equation}
\begin{equation}\label[diagram]{eq:counitality-of-comodule-monad}
  \begin{tikzpicture}
	\begin{pgfonlayer}{nodelayer}
		\node [style=none] (0) at (-0.575, -1.5) {};
		\node [style=none] (1) at (-0.575, 1.5) {};
		\node [style=none] (2) at (-1.075, 1.5) {};
		\node [style=none] (3) at (-1.075, 1.75) {\tiny $G$};
		\node [style=none] (4) at (-2.075, -1.5) {};
		\node [style=none] (5) at (-4.575, 1.5) {};
		\node [style=none] (6) at (-4.825, -2.25) {};
		\node [style=none] (7) at (-4.075, -0.75) {};
		\node [style=none] (8) at (-2.075, 0.5) {};
		\node [style=none] (9) at (-2.075, -0.25) {};
		\node [style=none] (10) at (-3.575, 0.05) {};
		\node [style=none] (11) at (-4.325, -2.25) {};
		\node [style=none] (12) at (-4.325, -2.55) {\tiny $G$};
		\node [style=none] (13) at (0, 0) {\scriptsize $=$};
		\node [style=none] (29) at (4.675, -1.5) {};
		\node [style=none] (30) at (4.675, 1.5) {};
		\node [style=none] (31) at (4.175, 1.5) {};
		\node [style=none] (32) at (4.175, 1.75) {\tiny $G$};
		\node [style=none] (33) at (3.175, -1.5) {};
		\node [style=none] (34) at (0.675, 1.5) {};
		\node [style=none] (35) at (0.425, -2.25) {};
		\node [style=none] (36) at (1.675, -1) {};
		\node [style=none] (37) at (3.175, -0.25) {};
		\node [style=none] (38) at (0.925, -2.25) {};
		\node [style=none] (39) at (0.925, -2.55) {\tiny $G$};
	\end{pgfonlayer}
	\begin{pgfonlayer}{edgelayer}
		\draw [in=0, out=-225] (4.center) to (7.center);
		\draw (8.center) to (4.center);
		\draw [style=unit, in=90, out=-165] (8.center) to (7.center);
		\draw [style=morphism-edge, in=0, out=-195] (9.center) to (10.center);
		\draw [style=Front] (5.center)
			 to [in=90, out=-90] (6.center)
			 to [in=-120, out=0] (4.center)
			 to (0.center)
			 to (1.center)
			 to cycle;
		\draw [style=comodule-edge, in=0, out=-90] (2.center) to (9.center);
		\draw [style=comodule-edge, in=90, out=195] (9.center) to (11.center);
		\draw [in=0, out=-225] (33.center) to (36.center);
		\draw (37.center) to (33.center);
		\draw [style=unit, in=90, out=-165] (37.center) to (36.center);
		\draw [style=Front] (34.center)
			 to [in=90, out=-90] (35.center)
			 to [in=-120, out=0] (33.center)
			 to (29.center)
			 to (30.center)
			 to cycle;
		\draw [style=comodule-edge, in=270, out=90] (38.center) to (31.center);
		\draw [in=240, out=0] (35.center) to (33.center);
		\draw (33.center) to (29.center);
		\draw (29.center) to (30.center);
		\draw (30.center) to (34.center);
		\draw [in=90, out=-90] (34.center) to (35.center);
		\draw [in=240, out=0] (6.center) to (4.center);
		\draw (4.center) to (0.center);
		\draw (0.center) to (1.center);
		\draw (1.center) to (5.center);
		\draw [in=90, out=-90] (5.center) to (6.center);
	\end{pgfonlayer}
\end{tikzpicture}
\end{equation}

\begin{definition}\label{def: morphism-of-comodules}
  Suppose that \((G, \delta^{(G)}),\, (K, \delta^{(K)}) \from \cat{M} \to \cat{N}\) are comodule functors over \(B, F \from \cat{C} \to \cat{D}\).
  A \emph{comodule natural transformation} from \(G\) to \(K\) comprises
  a pair of natural transformations \(\phi \from G \nt K\) and \(\psi \from B \nt F\),
  such that
  \begin{equation}\label[diagram]{eq:comodule-natural-transformation}
    \begin{tikzcd}
      {G(m \ract x)} & {K(m \ract x)} \\
      {Gm \bract Bx} & {Km \bract Fx}
      \arrow["{\phi_{m \ract x}}", from=1-1, to=1-2]
      \arrow["{\delta^{(K)}_{m,x}}", from=1-2, to=2-2]
      \arrow["{\delta_{m,x}^{(G)}}"', from=1-1, to=2-1]
      \arrow["{\phi_m \bract \psi_x}"', from=2-1, to=2-2]
    \end{tikzcd}
  \end{equation}
  commutes, for all \(x \in \cat{C}\) and \(m \in \cat{M}\).

  We call \((\phi, \psi)\) a \emph{morphism of comodule functors} if \(B=F\) and \(\psi=\id_B\).
\end{definition}

Given a comodule natural transformation $(\phi \from G \nt K,\, \psi \from B \nt F)$,
the graphical version of \cref{eq:comodule-natural-transformation} is displayed in our next picture,
where the blue dot indicates \(\phi\) and the black dot indicates $\psi$.
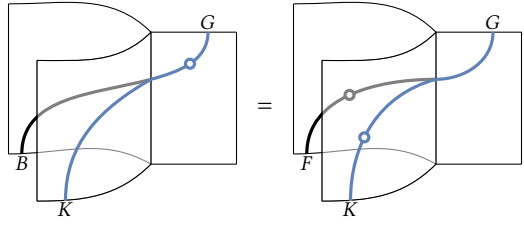
\[
  \begin{tikzpicture}
	\begin{pgfonlayer}{nodelayer}
		\node [style=none] (6) at (6.75, -1.5) {};
		\node [style=none] (7) at (4.5, -1.5) {};
		\node [style=none] (8) at (6.75, 2) {};
		\node [style=none] (9) at (4.5, 2) {};
		\node [style=none] (10) at (1.5, 1.25) {};
		\node [style=none] (11) at (0.75, 2.75) {};
		\node [style=none] (12) at (0.75, -1.23) {};
		\node [style=none] (13) at (1.5, -2.48) {};
		\node [style=none] (14) at (4.5, 0.75) {};
		\node [style=none] (15) at (6, 2) {};
		\node [style=none] (17) at (2.25, -2.475) {};
		\node [style=none] (18) at (1.1, -1.23) {};
		\node [style=none] (19) at (6, 2.25) {\tiny $G$};
		\node [style=none] (20) at (1.1, -1.48) {\tiny $F$};
		\node [style=none] (21) at (2.25, -2.705) {\tiny $K$};
		\node [style=none] (24) at (0, 0) {\scriptsize $=$};
		\node [style=none] (25) at (-0.75, -1.5) {};
		\node [style=none] (26) at (-3, -1.5) {};
		\node [style=none] (27) at (-0.75, 2) {};
		\node [style=none] (28) at (-3, 2) {};
		\node [style=none] (29) at (-6, 1.25) {};
		\node [style=none] (30) at (-6.75, 2.75) {};
		\node [style=none] (31) at (-6.75, -1.23) {};
		\node [style=none] (32) at (-6, -2.48) {};
		\node [style=none] (33) at (-3, 0.75) {};
		\node [style=none] (34) at (-1.5, 2) {};
		\node [style=none] (35) at (-5.25, -2.475) {};
		\node [style=none] (36) at (-6.4, -1.23) {};
		\node [style=none] (37) at (-1.5, 2.25) {\tiny $G$};
		\node [style=none] (38) at (-6.4, -1.48) {\tiny $B$};
		\node [style=none] (39) at (-5.25, -2.705) {\tiny $K$};
	\end{pgfonlayer}
	\begin{pgfonlayer}{edgelayer}
		\draw [in=-210, out=0] (12.center) to (7.center);
		\draw [in=-210, out=0] (31.center) to (26.center);
		\draw [style=morphism-edge, in=90, out=-180] (14.center) to node [style=blackdot, midway] {} (18.center);
		\draw [in=0, out=-225] (9.center) to (11.center);
		\draw [style=Front] (13.center)
			 to (10.center)
			 to [in=-135, out=0] (9.center)
			 to (7.center)
			 to [in=0, out=-135] cycle;
		\draw (7.center) to (6.center);
		\draw (6.center) to (8.center);
		\draw (8.center) to (9.center);
		\draw (12.center) to (11.center);
		\draw (10.center) to (13.center);
		\draw (9.center) to (7.center);
		\draw [in=0, out=225] (9.center) to (10.center);
		\draw [in=225, out=0] (13.center) to (7.center);
		\draw [style=comodule-edge, in=0, out=-90] (15.center) to (14.center);
		\draw [style=morphism-edge, in=90, out=-165] (33.center) to (36.center);
		\draw [in=0, out=-225] (28.center) to (30.center);
		\draw [style=Front] (32.center)
			 to (29.center)
			 to [in=-135, out=0] (28.center)
			 to (26.center)
			 to [in=0, out=-135] cycle;
		\draw (26.center) to (25.center);
		\draw (25.center) to (27.center);
		\draw (27.center) to (28.center);
		\draw (31.center) to (30.center);
		\draw (29.center) to (32.center);
		\draw (28.center) to (26.center);
		\draw [in=0, out=225] (28.center) to (29.center);
		\draw [in=225, out=0] (32.center) to (26.center);
		\draw [style=comodule-edge, in=90, out=210] (33.center) to (35.center);
		\draw [style=comodule-edge, in=90, out=195] (14.center) to node [style=blackdot, pos=0.6] {} (17.center);
		\draw [style=comodule-edge, in=15, out=-90] (34.center) to node [style=blackdot, midway] {} (33.center);
	\end{pgfonlayer}
\end{tikzpicture}
\]

\begin{remark}\label{rmk:comod-nat-trafo-is-comod-morphism}
  Suppose the pair \(\phi \from G \nt K\) and \(\psi \from B \nt F\) constitute a comodule natural transformation.
  We can view \(\phi \from G \nt K\) as a morphism of comodule functors over \(F\) if we equip \(G\) with a new coaction.
  It is given for all \(x \in \cat{C}\) and \(m \in \cat{M}\) by
  \[
    G(m \ract x) \xrightarrow{\;\delta^{(G)}_{m,x}\;} Gm \bract Bx \xrightarrow{Gm\bract\psi_x} Gm \bract Fx.
  \]
  It follows that by altering the involved coactions suitably,
  comodule natural transformations and morphisms of comodule functors can be identified with each other.
\end{remark}

Let \((B, \mu, \eta, B_2, B_0)\) be a bimonad on \(\cat{C}\).
The unit \(\eta \from \Id_{\cat{C}} \to B\) implements a coaction on \(\Id_{\cat{M}} \from \cat{M} \to \cat{M}\) via
\[
  \id_m \ract \eta_x \from \Id_{\cat{M}}(m \ract x) \to \Id_{\cat{M}}m \ract Bx, \qquad \text{ for all } x \in \cat{C}, m \in \cat{M}.
\]
Using the multiplication \(\mu \from B^2 \nt B\),
we can equip the composition \(GK\from \cat{M} \to \cat{M}\) of
two \(B\)-comodule functors \(G, K \from \cat{M} \to \cat{M}\) with a comodule structure:
\[
  GK(\blank \ract \bblank) \xrightarrow{G\delta^{(K)}} G(K\blank \ract B\bblank) \xrightarrow{\delta^{(K)}} GK\blank \ract B^2\bblank \xrightarrow{\id \ract \mu} GK\blank \ract B\bblank.
\]
Due to the associativity and unitality of the multiplication of \(B\), the category \(\Com(B, \cat{M})\) of comodule endofunctors on \(\cat{M}\) over \(B\) is monoidal.
Studying its monoids will be a main focus of the rest of this article.

\begin{definition}\label{def: comodule-monad}
  Consider a bimonad \(B\) on \(\cat{C}\).
  A \emph{comodule monad} over \(B\) on \(\cat{M}\) is a comodule endofunctor \((K, \delta) \from \cat{M} \to \cat{M}\)
  together with morphisms of comodule functors \(\mu \from K^2 \nt K\)  and \(\eta \from \Id_{\cat{M}} \nt K\)
  such that \((K,\mu,\eta)\) is a monad.
\end{definition}

The conditions for the multiplication and unit of
a comodule monad \(K \from \cat{M} \to \cat{M}\)
over a bimonad \(B \from \cat{C} \to \cat{C}\)
to be morphisms of comodule functors amount to
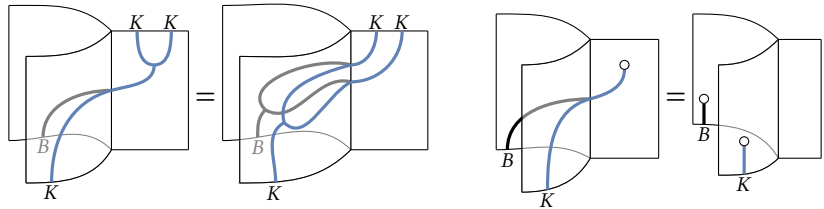
\[
  \begin{tikzpicture}[xscale=0.9, yscale=0.9]
	\begin{pgfonlayer}{nodelayer}
		\node [style=none] (0) at (-1, -1.5) {};
		\node [style=none] (1) at (-3.25, -1.5) {};
		\node [style=none] (2) at (-1, 2) {};
		\node [style=none] (3) at (-3.25, 2) {};
		\node [style=none] (4) at (-6.5, 1.25) {};
		\node [style=none] (5) at (-7, 2.75) {};
		\node [style=none] (6) at (-7, -1.23) {};
		\node [style=none] (7) at (-6.5, -2.48) {};
		\node [style=none] (8) at (-3.25, 0.5) {};
		\node [style=none] (9) at (-1.75, 2) {};
		\node [style=none] (10) at (-5.45, -2.45) {};
		\node [style=none] (11) at (-5.975, -1.125) {};
		\node [style=none] (15) at (-7.5, 0) {$=$};
		\node [style=none] (16) at (-8, -1.5) {};
		\node [style=none] (17) at (-10.25, -1.5) {};
		\node [style=none] (18) at (-8, 2) {};
		\node [style=none] (19) at (-10.25, 2) {};
		\node [style=none] (20) at (-12.75, 1.25) {};
		\node [style=none] (21) at (-13.25, 2.75) {};
		\node [style=none] (22) at (-13.25, -1.23) {};
		\node [style=none] (23) at (-12.75, -2.48) {};
		\node [style=none] (24) at (-10.25, 0.25) {};
		\node [style=none] (25) at (-8.5, 2) {};
		\node [style=none] (26) at (-12, -2.45) {};
		\node [style=none] (27) at (-12.25, -1.12) {};
		\node [style=none] (31) at (-9, 1) {};
		\node [style=none] (39) at (-9.5, 2) {};
		\node [style=none] (41) at (-2.5, 2) {};
		\node [style=none] (43) at (-3.25, 1) {};
		\node [style=none] (44) at (-5.2, -0.7) {};
		\node [style=none] (45) at (-5.75, -0.325) {};
		\node [style=none] (75) at (5.75, -1.75) {};
		\node [style=none] (76) at (3.75, -1.75) {};
		\node [style=none] (77) at (5.75, 1.75) {};
		\node [style=none] (78) at (3.75, 1.75) {};
		\node [style=none] (79) at (1.75, 1) {};
		\node [style=none] (80) at (1, 2.5) {};
		\node [style=none] (81) at (1, -1.5) {};
		\node [style=none] (82) at (1.75, -2.75) {};
		\node [style=none] (83) at (2.5, -2.675) {};
		\node [style=none] (84) at (1.325, -1.5) {};
		\node [style=none] (85) at (6.25, 0) {$=$};
		\node [style=none] (88) at (10.5, -1.75) {};
		\node [style=none] (89) at (9.25, -1.75) {};
		\node [style=none] (90) at (10.5, 1.75) {};
		\node [style=none] (91) at (9.25, 1.75) {};
		\node [style=none] (92) at (7.5, 1) {};
		\node [style=none] (93) at (6.75, 2.5) {};
		\node [style=none] (94) at (6.75, -0.75) {};
		\node [style=none] (95) at (7.5, -2.25) {};
		\node [style=none] (96) at (8.25, -2.225) {};
		\node [style=none] (97) at (7.075, -0.75) {};
		\node [style=blackdot] (100) at (4.75, 1) {};
		\node [style=none] (101) at (3.75, 0) {};
		\node [style=blackdot] (102) at (7.075, 0) {};
		\node [style=blackdot] (103) at (8.25, -1.25) {};
	\end{pgfonlayer}
	\begin{pgfonlayer}{edgelayer}
		\node [style=none] (13) at (-5.975, -1.43) {\tiny $B$};
		\node [style=none] (29) at (-12.25, -1.405) {\tiny $B$};
		\node [style=none] (98) at (7.075, -1.05) {\tiny $B$};
		\node [style=none] (99) at (8.25, -2.525) {\tiny $K$};
		\node [style=none] (86) at (1.325, -1.8) {\tiny $B$};
		\node [style=none] (87) at (2.5, -2.975) {\tiny $K$};
		\node [style=none] (42) at (-2.5, 2.25) {\tiny $K$};
		\node [style=none] (40) at (-9.5, 2.25) {\tiny $K$};
		\node [style=none] (28) at (-8.5, 2.25) {\tiny $K$};
		\node [style=none] (30) at (-12, -2.755) {\tiny $K$};
		\node [style=none] (12) at (-1.75, 2.25) {\tiny $K$};
		\node [style=none] (14) at (-5.475, -2.78) {\tiny $K$};
		\draw [in=-240, out=0] (94.center) to (89.center);
		\draw [in=-210, out=0] (81.center) to (76.center);
		\draw [in=-225, out=0] (6.center) to (1.center);
		\draw [in=-225, out=0] (22.center) to (17.center);
		\draw [in=0, out=-225] (3.center) to (5.center);
		\draw [style=morphism-edge, in=-30, out=165] (8.center) to (45.center);
		\draw [style=morphism-edge, in=135, out=165] (43.center) to (45.center);
		\draw [style=morphism-edge, in=90, out=-135] (45.center) to (11.center);
		\draw [style=Front] (7.center)
			 to (4.center)
			 to [in=-135, out=0] (3.center)
			 to (1.center)
			 to [in=0, out=-135] cycle;
		\draw (1.center) to (0.center);
		\draw (0.center) to (2.center);
		\draw (2.center) to (3.center);
		\draw (6.center) to (5.center);
		\draw (4.center) to (7.center);
		\draw (3.center) to (1.center);
		\draw [in=0, out=225] (3.center) to (4.center);
		\draw [in=225, out=0] (7.center) to (1.center);
		\draw [style=comodule-edge, in=0, out=-90] (9.center) to (8.center);
		\draw [style=morphism-edge, in=90, out=-180] (24.center) to (27.center);
		\draw [in=0, out=-225] (19.center) to (21.center);
		\draw [style=Front] (23.center)
			 to (20.center)
			 to [in=-135, out=0] (19.center)
			 to (17.center)
			 to [in=0, out=-120] cycle;
		\draw (17.center) to (16.center);
		\draw (16.center) to (18.center);
		\draw (18.center) to (19.center);
		\draw (22.center) to (21.center);
		\draw (20.center) to (23.center);
		\draw (19.center) to (17.center);
		\draw [in=0, out=225] (19.center) to (20.center);
		\draw [in=240, out=0] (23.center) to (17.center);
		\draw [style=comodule-edge, in=90, out=195] (24.center) to (26.center);
		\draw [style=comodule-edge, in=180, out=-90] (39.center) to (31.center);
		\draw [style=comodule-edge, in=0, out=-90] (25.center) to (31.center);
		\draw [style=comodule-edge, in=15, out=-90] (31.center) to (24.center);
		\draw [style=comodule-edge, in=0, out=-90] (41.center) to (43.center);
		\draw [style=comodule-edge, in=105, out=195] (43.center) to (44.center);
		\draw [style=comodule-edge, in=-45, out=210] (8.center) to (44.center);
		\draw [style=comodule-edge, in=90, out=-150] (44.center) to (10.center);
		\draw [in=0, out=-225] (78.center) to (80.center);
		\draw [style=morphism-edge, in=90, out=-180] (101.center) to (84.center);
		\draw [style=Front] (82.center)
			 to (79.center)
			 to [in=-135, out=0] (78.center)
			 to (76.center)
			 to [in=0, out=-120] cycle;
		\draw (76.center) to (75.center);
		\draw (75.center) to (77.center);
		\draw (77.center) to (78.center);
		\draw (81.center) to (80.center);
		\draw (79.center) to (82.center);
		\draw (78.center) to (76.center);
		\draw [in=0, out=225] (78.center) to (79.center);
		\draw [in=240, out=0] (82.center) to (76.center);
		\draw [in=0, out=-225] (91.center) to (93.center);
		\draw [style=Front] (95.center)
			 to (92.center)
			 to [in=-135, out=0] (91.center)
			 to (89.center)
			 to [in=0, out=-135] cycle;
		\draw (89.center) to (88.center);
		\draw (88.center) to (90.center);
		\draw (90.center) to (91.center);
		\draw (94.center) to (93.center);
		\draw (92.center) to (95.center);
		\draw (91.center) to (89.center);
		\draw [in=0, out=225] (91.center) to (92.center);
		\draw [in=225, out=0] (95.center) to (89.center);
		\draw [style=comodule-edge, in=15, out=-90] (100) to (101.center);
		\draw [style=comodule-edge, in=90, out=195] (101.center) to (83.center);
		\draw [style=morphism-edge, in=90, out=270] (102) to (97.center);
		\draw [style=comodule-edge, in=90, out=270] (103) to (96.center);
	\end{pgfonlayer}
\end{tikzpicture}
\]
Notice how these conditions are analogous to those in
\cref{eq:mu-opmonoidal-nt,eq:eta-opmonoidal-nt}.

\begin{example}\label{ex:comodule-monads-example}
  Consider the poset \((\mathbb{R}, \leq)\).
  It is a monoidal category with addition as tensor product and \(0 \in \mathbb{R}\) as unit.
  Hasegawa and Lemay, see~\cite{hasegawa-lemay2018:LinearHopf}, defined a bimonad using the ceiling function.
  Let
  \begin{equation*}
    H \from \mathbb{R}\to \mathbb{R}, \qquad x \mapsto \lceil x \rceil \eqdef\min \{\, n\in \mathbb{Z} \mid n \geq x \,\}.
  \end{equation*}
  As \(\lceil x + y \rceil \leq \lceil x \rceil + \lceil y \rceil\) for all \(x,y \in \mathbb{R}\) and \(\lceil 0 \rceil = 0\), the functor \(H\) is opmonoidal with comultiplication and counit given by the unique arrows
  \begin{equation*}
    H_1 \from H0=0 \to 0 \qquad\text{and}\qquad H_{2; x,y} \from H(x+y) \to Hx+ Hy, \text{ for all \(x,y \in \mathbb{R}\)}.
  \end{equation*}
  The idempotence of the ceiling function implies that the identity natural transformation defines a multiplication \(\mu \from H^2 \to H\).
  Its unit corresponds to \({(x \to Hx)}_{x \in \mathbb{R}}\).

  Given a number \(a \in [0, 1)\) in the half-open interval bounded by \(0\) and \(1\), define
  \begin{equation*}
    C_a \from \mathbb{R} \to \mathbb{R}, \qquad x \mapsto \min\{\, a+ n \mid n \in \mathbb{Z}, a+n \geq x \,\}.
  \end{equation*}
  In case \(a=0\), we have \(C_a=H\).
  Otherwise \(C_a(0)= a> 0\) and \(C_a\) cannot be opmonoidal.
  Nonetheless, \(C_a(x+y) \leq C_a x + \lceil y \rceil = C_a x + Hy\) holds,
  and the unique natural arrow
  \begin{equation*}
    \delta^{(a)}_{x,y} \from C_a(x+y) \to C_a x + Hy, \qquad\qquad \text{for all } x,y\in \mathbb{R}
  \end{equation*}
  defines a coaction of\, \(H\) on \(C\).
  Again, \(C_a^2=C_a\) is idempotent and \(x \leq C_a(x)\) for all \(x \in \mathbb{R}\).
  Thus, it is a comodule monad over \(H\).
\end{example}

\begin{example}\label{ex:copower-as-comodule-monad}
  Let \((\cat{V}, \otimes, 1)\) be a closed symmetric monoidal category.
  A \(\cat{V}\)-category \(\cat{C}\)
  is said to be \emph{copowered}
  over \(\cat{V}\),
  see~\cite[Section~3.7]{kelly05:basic},
  if there exists a functor \(\blank \cdot \bblank \from \cat{C} \times \cat{V} \to \cat{C}\),
  such that for all \(c \in \cat{C}\) we have
  \[
    \adj{ c \cdot\blank}{\cat{C}(c, \blank)}{\cat{V}}{\cat{C}}.
  \]
  One obtains \(c\cdot (v \otimes w) \cong (c\cdot v) \cdot w\) by the Yoneda lemma:
  \begin{align*}
    \cat{C}\big(c \cdot (v \otimes w), x\big)
    & \cong \cat{V}\big(v \otimes w, \cat{C}(c, x)\big)
      \cong \cat{V}\big(w, \cat{V}(v, \cat{C}(c, x))\big) \\
    & \cong \cat{V}\big(w ,\cat{C}((c \cdot v), x)\big)
      \cong \cat{C}\big((c \cdot v)\cdot w, x\big).
  \end{align*}

  In fact, this turns \(\cat{C}\) into a \(\cat{V}\)-module category,
  and therefore the identity monad on \(\cat{C}\) into a comodule monad over \(\Id_{\cat{V}}\) with trivial coaction.

  Suppose that \(B \eqdef UF\) is a bimonad on \(\cat{V}\).
  The unit \(\eta^{(B)} \from \Id_{\cat{V}} \to B\) is a morphism of bimonads and we may extend the coaction of\, \(\Id_{\cat{C}}\) to
  \[
    \delta\from \blank \cdot \bblank \xrightarrow{\; \blank \,\cdot\, \eta^{(B)}\;} \blank \otimes B\bblank.
  \]
\end{example}

\begin{remark}\label{rem:comodule-monads-induce-module-categories}
  Let \(B \from \cat{C} \to \cat{C}\) be a bimonad and  \((K, \delta) \from \cat{M} \to \cat{M}\) a comodule monad over it.
  The coaction of \(K\) allows us to define an action \(\widehat{\ract} \from \cat{M}^{K}\times \cat{C}^B \to \cat{M}^K\).
  For any two modules \((m, \vartheta_{m}) \in \cat{M}^K\) and \((x, \vartheta_{x}) \in \cat{C}^B\), it is given by
  \[
    (m, \vartheta_{m}) \mathbin{\widehat{\ract}}  (x, \vartheta_{x}) \eqdef \big(m \ract x, (\vartheta_{m}\ract \vartheta_{x}) \circ \delta_{m,x} \big).
  \]
  The axioms of the coaction of \(B\) on \(K\) translate precisely to the compatibility of the action of\, \(\cat{C}^B\) on \(\cat{M}^{K}\) with the tensor product and unit of\, \(\cat{C}^B\).
\end{remark}

We have already seen that monads and adjunctions are in close correspondence, and that additional structures on the monads have their counterparts expressed in terms of the units and counits of adjunctions.
In the case of comodule monads this is slightly more complicated as we have two adjunctions to consider: one corresponding to the bimonad and one to the comodule monad.

\begin{definition}\label{def:comodule-adjunction}
  Consider two adjunctions \(\stdadj\) and \(\adj{G}{V}{\cat{M}}{\cat{N}}\) such that \(F \dashv U\) is monoidal and \(G,V\) are comodule functors over \(F,U\).
  We call the pair \((G \dashv V, F \dashv U)\) a \emph{comodule adjunction} if the following two identities hold:
  \[
    \begin{tikzcd}
      {m \ract x} & {VG(m \ract x)} & {GV(n \bract y)} & {G(Vn \ract Uy)} \\
      {VGm \ract UFx} & {V(Gm \bract Fx)} & {n \bract y} & {GVn \bract FUy}
      \arrow["{\eta^{(G \lad V)}_{m \,\ract\, x}}", from=1-1, to=1-2]
      \arrow["{V\delta^{(G)}_{m,x}}", from=1-2, to=2-2]
      \arrow["{\delta^{(V)}_{Gm \,\bract\, Fx}}", from=2-2, to=2-1]
      \arrow["{\eta^{(G \lad V)}_m \,\ract\, \eta^{(F \lad U)}_x}"', from=1-1, to=2-1]
      \arrow["{\varepsilon^{(G \lad V)}_{n \,\bract\, y}}"', from=1-3, to=2-3]
      \arrow["{G\delta^{(V)}_{n,y}}", from=1-3, to=1-4]
      \arrow["{\delta^{(G)}_{Vn \,\ract\, Uy}}", from=1-4, to=2-4]
      \arrow["{\varepsilon^{(G \lad V)}_n \,\bract\, \varepsilon^{(F \lad U)}_y}", from=2-4, to=2-3]
    \end{tikzcd}
  \]
\end{definition}

From their string diagrammatic representations, one observes that the conditions required by \cref{def:comodule-adjunction} are analogous to
\cref{eq:unit-adjunction-opmonoidal-nt,eq:counit-adjunction-opmonoidal-nt}:
\begin{equation}\label[diagram]{eq:comodule-adjunction-unit-and-counit}
  \adjustbox{scale=0.95}{\begin{tikzpicture}
	\begin{pgfonlayer}{nodelayer}
		\node [style=none] (53) at (-1.5, -1.5) {};
		\node [style=none] (54) at (-2.5, -1.5) {};
		\node [style=none] (55) at (-1.5, 1.5) {};
		\node [style=none] (56) at (-2.5, 1.5) {};
		\node [style=none] (57) at (-5.25, 0.75) {};
		\node [style=none] (58) at (-6, 2.25) {};
		\node [style=none] (59) at (-6, -1) {};
		\node [style=none] (60) at (-5.25, -2.25) {};
		\node [style=none] (61) at (-3.225, -2.005) {};
		\node [style=none] (62) at (-4.75, -0.9) {};
		\node [style=none] (65) at (-6.75, 0) {\scriptsize $=$};
		\node [style=none] (66) at (-7.5, -1.5) {};
		\node [style=none] (67) at (-9.75, -1.5) {};
		\node [style=none] (68) at (-7.5, 1.5) {};
		\node [style=none] (69) at (-9.75, 1.5) {};
		\node [style=none] (70) at (-12.5, 0.75) {};
		\node [style=none] (71) at (-13.25, 2.25) {};
		\node [style=none] (72) at (-13.25, -1.23) {};
		\node [style=none] (73) at (-12.5, -2.48) {};
		\node [style=none] (74) at (-9.75, -0.475) {};
		\node [style=none] (75) at (-11.75, -2.5) {};
		\node [style=none] (76) at (-12.9, -1.205) {};
		\node [style=none] (79) at (-3.725, -0.45) {};
		\node [style=none] (80) at (-8.5, 1.225) {};
		\node [style=none] (81) at (-8, 0.8) {};
		\node [style=none] (82) at (-9, 0.8) {};
		\node [style=none] (83) at (-9.75, 0.35) {};
		\node [style=none] (84) at (-10.99, -2.49) {};
		\node [style=none] (86) at (-12, -1.16) {};
		\node [style=none] (88) at (-4.225, -2.25) {};
		\node [style=none] (90) at (-3.75, -0.885) {};
		\node [style=none] (92) at (-4.25, 0.3) {};
		\node [style=none] (93) at (13.25, -1.5) {};
		\node [style=none] (94) at (10.75, -1.5) {};
		\node [style=none] (95) at (13.25, 1.5) {};
		\node [style=none] (96) at (10.75, 1.5) {};
		\node [style=none] (97) at (8, 0.75) {};
		\node [style=none] (98) at (7.25, 2.25) {};
		\node [style=none] (99) at (7.25, -1) {};
		\node [style=none] (100) at (8, -2.25) {};
		\node [style=none] (101) at (12.5, 1.5) {};
		\node [style=none] (103) at (6.75, 0) {\scriptsize $=$};
		\node [style=none] (104) at (6.25, -1.5) {};
		\node [style=none] (105) at (3.75, -1.5) {};
		\node [style=none] (106) at (6.25, 1.5) {};
		\node [style=none] (107) at (3.75, 1.5) {};
		\node [style=none] (108) at (2.25, 0.75) {};
		\node [style=none] (109) at (1.5, 2.25) {};
		\node [style=none] (110) at (1.5, -1) {};
		\node [style=none] (111) at (2.25, -2.25) {};
		\node [style=none] (112) at (5.5, 1.5) {};
		\node [style=none] (114) at (5, -0.75) {};
		\node [style=none] (115) at (4.5, 1.5) {};
		\node [style=none] (117) at (11.5, 1.5) {};
		\node [style=none] (119) at (10.75, 0) {};
		\node [style=none] (120) at (10.75, 0.75) {};
		\node [style=none] (121) at (9.25, -0.75) {};
		\node [style=none] (122) at (8.5, -0.25) {};
	\end{pgfonlayer}
	\begin{pgfonlayer}{edgelayer}
		\node [style=none] (118) at (11.5, 1.75) {\tiny $V$};
		\node [style=none] (116) at (4.5, 1.75) {\tiny $V$};
		\node [style=none] (113) at (5.5, 1.75) {\tiny $G$};
		\node [style=none] (102) at (12.5, 1.75) {\tiny $G$};
		\node [style=none] (91) at (-3.75, -1.175) {\tiny $U$};
		\node [style=none] (89) at (-4.225, -2.56) {\tiny $G$};
		\node [style=none] (87) at (-12, -1.43) {\tiny $U$};
		\node [style=none] (85) at (-11, -2.78) {\tiny $V$};
		\node [style=none] (77) at (-12.9, -1.505) {\tiny $F$};
		\node [style=none] (78) at (-11.75, -2.78) {\tiny $G$};
		\node [style=none] (63) at (-4.75, -1.2) {\tiny $F$};
		\node [style=none] (64) at (-3.225, -2.325) {\tiny $V$};
		\draw [in=-210, out=0] (72.center) to (67.center);
		\draw [in=-225, out=0] (59.center) to (54.center);
		\draw [in=0, out=-225] (56.center) to (58.center);
		\draw [style=morphism-edge] (62.center)
			 to [in=180, out=90] (92.center)
			 to [in=90, out=0] (90.center);
		\draw [style=Front] (60.center)
			 to (57.center)
			 to [in=-135, out=0] (56.center)
			 to (54.center)
			 to [in=0, out=-135] cycle;
		\draw (54.center) to (53.center);
		\draw (53.center) to (55.center);
		\draw (55.center) to (56.center);
		\draw (59.center) to (58.center);
		\draw (57.center) to (60.center);
		\draw (56.center) to (54.center);
		\draw [in=0, out=225] (56.center) to (57.center);
		\draw [in=225, out=0] (60.center) to (54.center);
		\draw [in=0, out=-225] (69.center) to (71.center);
		\draw [style=morphism-edge, in=90, out=180, looseness=0.75] (83.center) to (76.center);
		\draw [style=morphism-edge, in=90, out=180, looseness=0.75] (74.center) to (86.center);
		\draw [style=Front] (73.center)
			 to (70.center)
			 to [in=-135, out=0] (69.center)
			 to (67.center)
			 to [in=0, out=-120] cycle;
		\draw (67.center) to (66.center);
		\draw (66.center) to (68.center);
		\draw (68.center) to (69.center);
		\draw (72.center) to (71.center);
		\draw (70.center) to (73.center);
		\draw (69.center) to (67.center);
		\draw [in=0, out=225] (69.center) to (70.center);
		\draw [in=240, out=0] (73.center) to (67.center);
		\draw [style=comodule-edge] (88.center)
			 to [in=180, out=90] (79.center)
			 to [in=90, out=0] (61.center);
		\draw [style=comodule-edge] (83.center)
			 to [in=-105, out=15] (82.center)
			 to [in=180, out=75] (80.center)
			 to [in=90, out=0] (81.center)
			 to [in=15, out=-90] (74.center);
		\draw [style=comodule-edge, in=90, out=195] (83.center) to (75.center);
		\draw [style=comodule-edge, in=90, out=195] (74.center) to (84.center);
		\draw [in=-210, out=0] (110.center) to (105.center);
		\draw [in=-225, out=0] (99.center) to (94.center);
		\draw [in=0, out=-225] (96.center) to (98.center);
		\draw [style=morphism-edge] (119.center)
			 to [in=-30, out=165] (122.center)
			 to [in=165, out=135] (120.center);
		\draw [style=Front] (100.center)
			 to (97.center)
			 to [in=-135, out=0] (96.center)
			 to (94.center)
			 to [in=0, out=-135] cycle;
		\draw (94.center) to (93.center);
		\draw (93.center) to (95.center);
		\draw (95.center) to (96.center);
		\draw (99.center) to (98.center);
		\draw (97.center) to (100.center);
		\draw (96.center) to (94.center);
		\draw [in=0, out=225] (96.center) to (97.center);
		\draw [in=225, out=0] (100.center) to (94.center);
		\draw [in=0, out=-225] (107.center) to (109.center);
		\draw [style=Front] (111.center)
			 to (108.center)
			 to [in=-120, out=0] (107.center)
			 to (105.center)
			 to [in=0, out=-135] cycle;
		\draw (105.center) to (104.center);
		\draw (104.center) to (106.center);
		\draw (106.center) to (107.center);
		\draw (110.center) to (109.center);
		\draw (108.center) to (111.center);
		\draw (107.center) to (105.center);
		\draw [in=0, out=240] (107.center) to (108.center);
		\draw [in=225, out=0] (111.center) to (105.center);
		\draw [style=comodule-edge] (112.center)
			 to [in=0, out=-90] (114.center)
			 to [in=-90, out=180] (115.center);
		\draw [style=comodule-edge, in=0, out=-90] (117.center) to (120.center);
		\draw [style=comodule-edge, in=0, out=-90] (101.center) to (119.center);
		\draw [style=comodule-edge] (119.center)
			 to [in=-40, out=-150] (121.center)
			 to [in=-165, out=138] (120.center);
	\end{pgfonlayer}
\end{tikzpicture}}
\end{equation}

Our main theorem extends~\cite[Proposition~4.1.2]{Aguiar2012} and~\cite[Lemma~7.10]{Turaev2017};
we prove it analogously to the latter.

\begin{theorem}\label{thm:comodule-adjunctions-correspond-to-strong-comodule-functors}
  Let \(\cat{C}\) and \(\cat{D}\) be monoidal categories,
  and suppose that \(\cat{M}\) and \(\cat{N}\) are right \(\cat{C}\)- and \(\cat{D}\)-module categories, respectively.
  Suppose that \,\stdadj{} is a monoidal adjunction,
  and let \(\adj{G}{V}{\cat{M}}{\cat{N}}\) be any adjunction.
  Lifts of\, \(G \dashv V\) to a comodule adjunction are in bijection with lifts of \(\,V \from \cat{N} \to \cat{M}\) to a strong comodule functor.
\end{theorem}
\begin{proof}
  Let \(G \dashv V\) be a comodule adjunction over \(F \dashv U\),
  and write \(\delta^{(V)}\) and \(\delta^{(G)}\) for the coactions of \(V\) and \(G\), respectively.
  For all \(m \in \cat{M}\) and \(x \in \cat{C}\),
  we define the inverse \(\delta^{-(V)}\from V\blank \ract U\bblank \nt V(\blank \bract \bblank)\) of\, \(\delta^{(V)}\) via
  \[
    Vn \ract  Uy
    \xrightarrow{\eta^{(G \adjoint V)}_{Vn \,\ract\, Uy}} VG(Vn \ract Uy)
    \xrightarrow{V\delta^{(G)}_{Vn, Uy}} V(GVm \bract FUx)
    \xrightarrow{V\left(\varepsilon^{(G \adjoint V)}_m \,\bract\, \varepsilon^{(F \adjoint U)}_x\right)} V(m \bract x).
  \]
  That is,
  \begin{equation}\label[diagram]{eq:inverse-of-coaction}
    \begin{tikzpicture}
	\begin{pgfonlayer}{nodelayer}
		\node [style=none] (0) at (3, -2.25) {};
		\node [style=none] (1) at (0, -2.25) {};
		\node [style=none] (2) at (3, 1.75) {};
		\node [style=none] (3) at (0, 1.75) {};
		\node [style=none] (4) at (-3, 1.25) {};
		\node [style=none] (5) at (-3.75, 2.75) {};
		\node [style=none] (6) at (-3.75, -1.75) {};
		\node [style=none] (7) at (-3, -3) {};
		\node [style=none] (11) at (-1.25, -1) {};
		\node [style=none] (14) at (-2, -1.5) {};
		\node [style=none] (15) at (0, 0) {};
		\node [style=none] (16) at (1.5, 1) {};
		\node [style=none] (17) at (0, 0) {};
		\node [style=none] (18) at (2, 0.25) {};
		\node [style=none] (19) at (2, -2.25) {};
		\node [style=none] (20) at (2, -2.55) {\tiny $V$};
		\node [style=none] (21) at (-1.75, 1.255) {};
		\node [style=none] (22) at (-1.75, 1.475) {\tiny $V$};
		\node [style=none] (23) at (-2.5, 2.74) {};
		\node [style=none] (24) at (-2.5, 3) {\tiny $U$};
	\end{pgfonlayer}
	\begin{pgfonlayer}{edgelayer}
		\draw [in=-210, out=0] (6.center) to (1.center);
		\draw [in=0, out=-225] (3.center) to (5.center);
		\draw [style=morphism-edge] (23.center)
			 to [in=180, out=-90, looseness=0.75] (14.center)
			 to [in=-180, out=0, looseness=1.25] (15.center);
		\draw [style=Front] (7.center)
			 to (4.center)
			 to [in=-150, out=0] (3.center)
			 to (1.center)
			 to [in=0, out=-135] cycle;
		\draw (1.center) to (0.center);
		\draw (0.center) to (2.center);
		\draw (2.center) to (3.center);
		\draw (6.center) to (5.center);
		\draw (4.center) to (7.center);
		\draw (3.center) to (1.center);
		\draw [in=0, out=210] (3.center) to (4.center);
		\draw [in=225, out=0] (7.center) to (1.center);
		\draw [style=comodule-edge] (17.center)
			 to [in=180, out=0] (16.center)
			 to [in=90, out=0] (18.center)
			 to (19.center);
		\draw [style=comodule-edge] (21.center)
			 to [in=180, out=-90, looseness=0.75] (11.center)
			 to [in=-150, out=0, looseness=1.25] (15.center);
	\end{pgfonlayer}
\end{tikzpicture}
  \end{equation}
  Using that \(G\) and \(V\) are part of a comodule adjunction,
  a straightforward computation proves \(\delta^{-(V)} \circ \delta^{(V)} = \Id_{V(\blank \,\bract\, \bblank)}\):
  \begin{equation*}
    \begin{tikzpicture}
	\begin{pgfonlayer}{nodelayer}
		\node [style=none] (0) at (-4, -2.75) {};
		\node [style=none] (1) at (-7, -2.75) {};
		\node [style=none] (2) at (-4, 2.75) {};
		\node [style=none] (3) at (-7, 2.75) {};
		\node [style=none] (4) at (-10, 2) {};
		\node [style=none] (5) at (-10.75, 3.5) {};
		\node [style=none] (6) at (-10.75, -2.25) {};
		\node [style=none] (7) at (-10, -3.5) {};
		\node [style=none] (9) at (-5.5, 2.75) {};
		\node [style=none] (15) at (-5.5, 3) {\tiny $V$};
		\node [style=none] (18) at (-8.75, 0.75) {};
		\node [style=none] (19) at (-8.25, -1) {};
		\node [style=none] (21) at (-7, 1.5) {};
		\node [style=none] (22) at (-9.5, 0) {};
		\node [style=none] (23) at (-9, -1.5) {};
		\node [style=none] (25) at (-5.5, 1) {};
		\node [style=none] (26) at (-7, 0) {};
		\node [style=none] (27) at (-5, 0.25) {};
		\node [style=none] (28) at (-5, -2.75) {};
		\node [style=none] (29) at (-5, -3.05) {\tiny $V$};
		\node [style=none] (30) at (3.75, -2.75) {};
		\node [style=none] (31) at (0.75, -2.75) {};
		\node [style=none] (32) at (3.75, 2.75) {};
		\node [style=none] (33) at (0.75, 2.75) {};
		\node [style=none] (34) at (-2.25, 2) {};
		\node [style=none] (35) at (-3, 3.5) {};
		\node [style=none] (36) at (-3, -2.25) {};
		\node [style=none] (37) at (-2.25, -3.5) {};
		\node [style=none] (38) at (2.25, 2.75) {};
		\node [style=none] (39) at (2.25, 3) {\tiny $V$};
		\node [style=none] (46) at (2.75, 1.75) {};
		\node [style=none] (48) at (3.25, 1) {};
		\node [style=none] (49) at (2.75, -2.75) {};
		\node [style=none] (50) at (2.75, -3.05) {\tiny $V$};
		\node [style=none] (51) at (1.75, 0.25) {};
		\node [style=none] (52) at (1.25, 1.5) {};
		\node [style=none] (53) at (-3.5, 0) {$=$};
		\node [style=none] (71) at (4.25, 0) {$=$};
		\node [style=none] (72) at (11.5, -2.75) {};
		\node [style=none] (73) at (8.5, -2.75) {};
		\node [style=none] (74) at (11.5, 2.75) {};
		\node [style=none] (75) at (8.5, 2.75) {};
		\node [style=none] (76) at (5.5, 2) {};
		\node [style=none] (77) at (4.75, 3.5) {};
		\node [style=none] (78) at (4.75, -2.25) {};
		\node [style=none] (79) at (5.5, -3.5) {};
		\node [style=none] (81) at (10, 3) {\tiny $V$};
		\node [style=none] (83) at (10, 2.75) {};
		\node [style=none] (84) at (10, -2.75) {};
		\node [style=none] (85) at (10, -3.05) {\tiny $V$};
	\end{pgfonlayer}
	\begin{pgfonlayer}{edgelayer}
		\draw [style=morphism-edge, in=195, out=0, looseness=1.25] (23.center) to (26.center);
		\draw [in=-225, out=0] (36.center) to (31.center);
		\draw [in=-225, out=0] (6.center) to (1.center);
		\draw [in=0, out=-225] (3.center) to (5.center);
		\draw [style=morphism-edge, in=90, out=-165, looseness=0.75] (21.center) to (22.center);
		\draw [style=morphism-edge, in=180, out=-90] (22.center) to (23.center);
		\draw [style=Front] (7.center)
			 to (4.center)
			 to [in=-135, out=0] (3.center)
			 to (1.center)
			 to [in=0, out=-135] cycle;
		\draw (1.center) to (0.center);
		\draw (0.center) to (2.center);
		\draw (2.center) to (3.center);
		\draw (6.center) to (5.center);
		\draw (4.center) to (7.center);
		\draw (3.center) to (1.center);
		\draw [in=0, out=225] (3.center) to (4.center);
		\draw [in=225, out=0] (7.center) to (1.center);
		\draw [style=comodule-edge, in=90, out=180, looseness=0.75] (21.center) to (18.center);
		\draw [style=comodule-edge, in=180, out=-90] (18.center) to (19.center);
		\draw [style=comodule-edge] (26.center)
			 to [in=180, out=0] (25.center)
			 to [in=90, out=0] (27.center)
			 to (28.center);
		\draw [style=comodule-edge, in=0, out=-90] (9.center) to (21.center);
		\draw [in=0, out=-225] (33.center) to (35.center);
		\draw [style=Front] (37.center)
			 to (34.center)
			 to [in=-150, out=0] (33.center)
			 to (31.center)
			 to [in=0, out=-135] cycle;
		\draw (31.center) to (30.center);
		\draw (30.center) to (32.center);
		\draw (32.center) to (33.center);
		\draw (36.center) to (35.center);
		\draw (34.center) to (37.center);
		\draw (33.center) to (31.center);
		\draw [in=0, out=210] (33.center) to (34.center);
		\draw [in=225, out=0] (37.center) to (31.center);
		\draw [style=comodule-edge] (49.center)
			 to [in=-90, out=90] (48.center)
			 to [in=0, out=90] (46.center)
			 to [in=0, out=180] (51.center)
			 to [in=-90, out=180] (52.center)
			 to [in=-90, out=90] (38.center);
		\draw [in=-225, out=0] (78.center) to (73.center);
		\draw [in=0, out=-225] (75.center) to (77.center);
		\draw [style=Front] (79.center)
			 to (76.center)
			 to [in=-150, out=0] (75.center)
			 to (73.center)
			 to [in=0, out=-135] cycle;
		\draw (73.center) to (72.center);
		\draw (72.center) to (74.center);
		\draw (74.center) to (75.center);
		\draw (78.center) to (77.center);
		\draw (76.center) to (79.center);
		\draw (75.center) to (73.center);
		\draw [in=0, out=210] (75.center) to (76.center);
		\draw [in=225, out=0] (79.center) to (73.center);
		\draw [style=comodule-edge, in=-90, out=90] (84.center) to (83.center);
		\draw [style=comodule-edge, in=-150, out=0, looseness=1.25] (19.center) to (26.center);
	\end{pgfonlayer}
\end{tikzpicture}
  \end{equation*}
  A similar strategy can be used to show that \(\delta^{(V)} \circ \delta^{-(V)} = \id_{V\blank \,\ract\, U\bblank}\).
  Thus, \(\delta^{(V)}\) is a natural isomorphism and therefore \(V\) is a strong comodule functor.

  Conversely, suppose that \((V, \delta^{(V)}) \from \cat{N} \to \cat{M}\) is a strong comodule functor over \(U\).
  Define an arrow \(\delta^{(G)}\from G(\blank \ract \bblank) \nt G\blank \bract F\bblank\) by
  \begin{equation}\label[diagram]{eq:coaction-via-counit-unit-and-inverse}
    \begin{tikzpicture}
	\begin{pgfonlayer}{nodelayer}
		\node [style=none] (0) at (3.5, -1.5) {};
		\node [style=none] (1) at (0.5, -1.5) {};
		\node [style=none] (2) at (3.5, 2) {};
		\node [style=none] (3) at (0.5, 2) {};
		\node [style=none] (4) at (-2.5, 1.25) {};
		\node [style=none] (5) at (-3.25, 2.75) {};
		\node [style=none] (6) at (-3.25, -1.25) {};
		\node [style=none] (7) at (-2.5, -2.5) {};
		\node [style=none] (14) at (2.5, 2) {};
		\node [style=none] (15) at (2.5, 2.25) {\tiny $G$};
		\node [style=none] (16) at (-0.5, 1) {};
		\node [style=none] (19) at (-1.75, 1.75) {};
		\node [style=none] (22) at (0.5, 0.25) {};
		\node [style=none] (23) at (-1.75, -2.475) {};
		\node [style=none] (24) at (-2.85, -1.25) {};
		\node [style=none] (25) at (1.75, -0.75) {};
		\node [style=none] (26) at (-1.75, -2.775) {\tiny $G$};
		\node [style=none] (27) at (-2.85, -1.55) {\tiny $F$};
		\node [style=none] (28) at (-1, 1.25) {};
	\end{pgfonlayer}
	\begin{pgfonlayer}{edgelayer}
		\draw [style=morphism-edge, in=180, out=90, looseness=0.75] (24.center) to (19.center);
		\draw [style=morphism-edge, in=150, out=-60] (28.center) to (22.center);
		\draw [style=morphism-edge, in=360, out=120] (28.center) to (19.center);
		\draw [in=0, out=-225] (3.center) to (5.center);
		\draw [style=Front] (7.center)
			 to (4.center)
			 to [in=-135, out=0] (3.center)
			 to (1.center)
			 to [in=0, out=-135] cycle;
		\draw (1.center) to (0.center);
		\draw (0.center) to (2.center);
		\draw (2.center) to (3.center);
		\draw [in=-210, out=0] (6.center) to (1.center);
		\draw (6.center) to (5.center);
		\draw (4.center) to (7.center);
		\draw (3.center) to (1.center);
		\draw [in=0, out=225] (3.center) to (4.center);
		\draw [in=225, out=0] (7.center) to (1.center);
		\draw [style=comodule-edge] (14.center)
			 to [in=0, out=-90] (25.center)
			 to [in=-15, out=180] (22.center);
		\draw [style=comodule-edge, in=360, out=135] (22.center) to (16.center);
		\draw [style=comodule-edge, in=90, out=-180, looseness=0.75] (16.center) to (23.center);
	\end{pgfonlayer}
\end{tikzpicture}
  \end{equation}
  Suppose that \(\eta\) and \(\varepsilon\) are the unit and counit of the adjunction \(F \lad U\).
  Due to~\cite[Lemma~7.10]{Turaev2017}, the comultiplication and counit of\, \(F \from \cat{C} \to \cat{D}\) are for all \(x,y \in \cat{C}\) given by
  \[
    F_{2,x,y} \eqdef F(x \otimes y)
    \xrightarrow{F(\eta_x \otimes \eta_y)} F(UFx \otimes UFy)
    \xrightarrow{FU_{-2, Fx, Fy}} FU(Fx \otimes Fy)
    \xrightarrow{\varepsilon_{Fx \otimes Fx}} Fx \otimes Fy,
  \]
  \[
    F_0 \eqdef F1 \xrightarrow{FU_{-0}} FU1 \xrightarrow{\epsilon_1} 1.
  \]
  Note that, graphically, \(F_2\) looks just like \cref{eq:coaction-via-counit-unit-and-inverse},
  with black strings taking the place of blue ones.
  The following shows that \(\delta^{(G)}\from G(\blank \ract \bblank) \nt G\blank \bract F\bblank\) is a coaction;
  \ie, satisfies the coassociativity and counitality conditions given in
  \cref{eq:coassoc-of-comodule-monad,eq:counitality-of-comodule-monad}:
  \[
    \adjustbox{scale=0.9}{\input{lifting-adjunctions-c2-prop-1.tikz}}
  \]
  and
  \[
    \begin{tikzpicture}
	\begin{pgfonlayer}{nodelayer}
		\node [style=none] (157) at (-5.75, -0.5) {};
		\node [style=none] (162) at (-6.25, 0.25) {};
		\node [style=none] (163) at (-6.875, -0.75) {};
		\node [style=none] (148) at (-3, -2) {};
		\node [style=none] (149) at (-3, 2) {};
		\node [style=none] (150) at (-4, 2) {};
		\node [style=none] (151) at (-4, 2.25) {\tiny $G$};
		\node [style=none] (152) at (-5.75, -2) {};
		\node [style=none] (153) at (-8.25, 2) {};
		\node [style=none] (154) at (-8.5, -2.75) {};
		\node [style=none] (155) at (-7.75, -1.25) {};
		\node [style=none] (156) at (-5.75, 1) {};
		\node [style=none] (159) at (-8, -2.75) {};
		\node [style=none] (160) at (-8, -3.05) {\tiny $G$};
		\node [style=none] (161) at (-4.5, -1.25) {};
		\node [style=none] (164) at (-7, 0.5) {};
		\node [style=none] (181) at (-2.5, 0) {\scriptsize $=$};
		\node [style=none] (185) at (2.95, -2) {};
		\node [style=none] (186) at (2.95, 2) {};
		\node [style=none] (187) at (1.95, 2) {};
		\node [style=none] (188) at (1.95, 2.25) {\tiny $G$};
		\node [style=none] (189) at (0.7, -2) {};
		\node [style=none] (190) at (-1.8, 2) {};
		\node [style=none] (191) at (-2.05, -2.75) {};
		\node [style=none] (192) at (-1.3, -1.25) {};
		\node [style=none] (193) at (0.7, -0.5) {};
		\node [style=none] (194) at (-1.55, -2.75) {};
		\node [style=none] (195) at (-1.55, -3.05) {\tiny $G$};
		\node [style=none] (196) at (1.2, -0.25) {};
		\node [style=none] (197) at (-0.55, 1) {};
		\node [style=none] (201) at (9, -2) {};
		\node [style=none] (202) at (9, 2) {};
		\node [style=none] (203) at (8, 2) {};
		\node [style=none] (204) at (8, 2.25) {\tiny $G$};
		\node [style=none] (205) at (6.75, -2) {};
		\node [style=none] (206) at (4.25, 2) {};
		\node [style=none] (207) at (4, -2.75) {};
		\node [style=none] (208) at (4.75, -1.25) {};
		\node [style=none] (209) at (6.75, -0.5) {};
		\node [style=none] (210) at (4.5, -2.75) {};
		\node [style=none] (211) at (4.5, -3.05) {\tiny $G$};
		\node [style=none] (214) at (3.5, 0) {\scriptsize $=$};
		\node [style=none] (215) at (-7.525, -0.25) {};
	\end{pgfonlayer}
	\begin{pgfonlayer}{edgelayer}
		\draw [style=morphism-edge] (157.center)
			 to [in=360, out=135] (162.center)
			 to [in=0, out=180] (163.center)
			 to [in=360, out=180] (215.center);
		\draw [in=0, out=-225] (152.center) to (155.center);
		\draw (156.center) to (152.center);
		\draw [style=unit, in=90, out=-165] (156.center) to (155.center);
		\draw [style=Front] (153.center)
			 to [in=90, out=-90] (154.center)
			 to [in=-120, out=0] (152.center)
			 to (148.center)
			 to (149.center)
			 to cycle;
		\draw [in=240, out=0] (154.center) to (152.center);
		\draw (152.center) to (148.center);
		\draw (148.center) to (149.center);
		\draw (149.center) to (153.center);
		\draw [in=90, out=-90] (153.center) to (154.center);
		\draw [style=comodule-edge, in=345, out=180] (161.center) to (157.center);
		\draw [style=comodule-edge, in=270, out=0, looseness=0.75] (161.center) to (150.center);
		\draw [style=comodule-edge, in=150, out=0] (164.center) to (157.center);
		\draw [style=comodule-edge, in=90, out=-180] (164.center) to (159.center);
		\draw [in=0, out=-225] (189.center) to (192.center);
		\draw (193.center) to (189.center);
		\draw [style=unit, in=90, out=-165] (193.center) to (192.center);
		\draw [style=Front] (190.center)
			 to [in=90, out=-90] (191.center)
			 to [in=-120, out=0] (189.center)
			 to (185.center)
			 to (186.center)
			 to cycle;
		\draw [in=240, out=0] (191.center) to (189.center);
		\draw (189.center) to (185.center);
		\draw (185.center) to (186.center);
		\draw (186.center) to (190.center);
		\draw [in=90, out=-90] (190.center) to (191.center);
		\draw [style=comodule-edge, in=270, out=0, looseness=0.75] (196.center) to (187.center);
		\draw [style=comodule-edge, in=90, out=-180] (197.center) to (194.center);
		\draw [style=comodule-edge, in=360, out=180] (196.center) to (197.center);
		\draw [in=0, out=-225] (205.center) to (208.center);
		\draw (209.center) to (205.center);
		\draw [style=unit, in=90, out=-165] (209.center) to (208.center);
		\draw [style=Front] (206.center)
			 to [in=90, out=-90] (207.center)
			 to [in=-120, out=0] (205.center)
			 to (201.center)
			 to (202.center)
			 to cycle;
		\draw [in=240, out=0] (207.center) to (205.center);
		\draw (205.center) to (201.center);
		\draw (201.center) to (202.center);
		\draw (202.center) to (206.center);
		\draw [in=90, out=-90] (206.center) to (207.center);
		\draw [style=comodule-edge, in=90, out=-90] (203.center) to (210.center);
	\end{pgfonlayer}
\end{tikzpicture}
  \]
  A straightforward computation proves that the unit of the adjunction \(G \dashv V\) satisfies
  the axioms displayed in  \cref{eq:comodule-adjunction-unit-and-counit}:
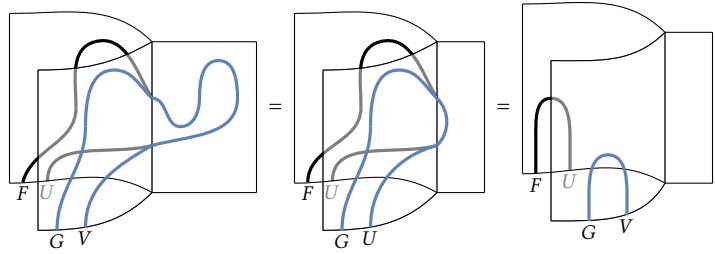
  \begin{equation*}
    \begin{tikzpicture}
	\begin{pgfonlayer}{nodelayer}
		\node [style=none] (31) at (-2.75, -2.25) {};
		\node [style=none] (32) at (-5.5, -2.25) {};
		\node [style=none] (33) at (-2.75, 1.75) {};
		\node [style=none] (34) at (-5.5, 1.75) {};
		\node [style=none] (35) at (-8.5, 1) {};
		\node [style=none] (36) at (-9.25, 2.5) {};
		\node [style=none] (37) at (-9.25, -2) {};
		\node [style=none] (38) at (-8.5, -3.25) {};
		\node [style=none] (42) at (-8.9, -1.975) {};
		\node [style=none] (58) at (-8.25, -1.95) {};
		\node [style=none] (61) at (-3.75, 1.25) {};
		\node [style=none] (63) at (-3.25, 0.25) {};
		\node [style=none] (64) at (-4.25, 0.25) {};
		\node [style=none] (65) at (-4.75, -0.5) {};
		\node [style=none] (66) at (-5.5, 0.25) {};
		\node [style=none] (67) at (-6.5, 1) {};
		\node [style=none] (69) at (-7.25, -0.5) {};
		\node [style=none] (70) at (-8, -3.25) {};
		\node [style=none] (72) at (-7, 1.75) {};
		\node [style=none] (74) at (-7.5, 0) {};
		\node [style=none] (76) at (-7.25, -3.175) {};
		\node [style=none] (78) at (-5.5, -1) {};
		\node [style=none] (107) at (9.25, -2) {};
		\node [style=none] (108) at (8, -2) {};
		\node [style=none] (109) at (9.25, 2) {};
		\node [style=none] (110) at (8, 2) {};
		\node [style=none] (111) at (5, 1.25) {};
		\node [style=none] (112) at (4.25, 2.75) {};
		\node [style=none] (113) at (4.25, -1.75) {};
		\node [style=none] (114) at (5, -3) {};
		\node [style=none] (115) at (4.6, -1.75) {};
		\node [style=none] (116) at (5.5, -1.65) {};
		\node [style=none] (124) at (6.5, -1.25) {};
		\node [style=none] (125) at (7, -2.25) {};
		\node [style=none] (126) at (6, -2.25) {};
		\node [style=none] (127) at (6, -3) {};
		\node [style=none] (128) at (5, 0.25) {};
		\node [style=none] (129) at (5.5, -0.75) {};
		\node [style=none] (130) at (4.6, -0.75) {};
		\node [style=none] (132) at (7, -2.85) {};
		\node [style=none] (166) at (-2.25, 0) {\tiny $=$};
		\node [style=none] (168) at (3.75, 0) {\tiny $=$};
		\node [style=none] (169) at (3.25, -2.25) {};
		\node [style=none] (170) at (2, -2.25) {};
		\node [style=none] (171) at (3.25, 1.75) {};
		\node [style=none] (172) at (2, 1.75) {};
		\node [style=none] (173) at (-1, 1) {};
		\node [style=none] (174) at (-1.75, 2.5) {};
		\node [style=none] (175) at (-1.75, -2) {};
		\node [style=none] (176) at (-1, -3.25) {};
		\node [style=none] (177) at (-1.4, -1.975) {};
		\node [style=none] (178) at (-0.75, -1.95) {};
		\node [style=none] (185) at (2, 0.25) {};
		\node [style=none] (186) at (1, 1) {};
		\node [style=none] (187) at (0.25, -0.5) {};
		\node [style=none] (188) at (-0.5, -3.25) {};
		\node [style=none] (189) at (0.5, 1.75) {};
		\node [style=none] (190) at (0, 0) {};
		\node [style=none] (192) at (0.25, -3.175) {};
		\node [style=none] (194) at (2, -1) {};
	\end{pgfonlayer}
	\begin{pgfonlayer}{edgelayer}
		\node [style=none] (193) at (0.25, -3.475) {\tiny $U$};
		\node [style=none] (191) at (-0.5, -3.55) {\tiny $G$};
		\node [style=none] (179) at (-1.4, -2.275) {\tiny $F$};
		\node [style=none] (180) at (-0.75, -2.25) {\tiny $U$};
		\node [style=none] (133) at (7, -3.15) {\tiny $V$};
		\node [style=none] (131) at (6, -3.3) {\tiny $G$};
		\node [style=none] (117) at (4.6, -2.05) {\tiny $F$};
		\node [style=none] (118) at (5.5, -1.95) {\tiny $U$};
		\node [style=none] (77) at (-7.25, -3.475) {\tiny $V$};
		\node [style=none] (75) at (-8, -3.55) {\tiny $G$};
		\node [style=none] (59) at (-8.9, -2.275) {\tiny $F$};
		\node [style=none] (60) at (-8.25, -2.25) {\tiny $U$};
		\draw [style=morphism-edge, in=375, out=120] (66.center) to (72.center);
		\draw [in=0, out=-225] (34.center) to (36.center);
		\draw [style=morphism-edge, in=90, out=195] (72.center) to (74.center);
		\draw [style=morphism-edge, in=75, out=-90] (74.center) to (42.center);
		\draw [style=morphism-edge, in=90, out=-165] (78.center) to (58.center);
		\draw [style=Front] (32.center)
			 to [in=0, out=-135] (38.center)
			 to (35.center)
			 to [in=-150, out=0] (34.center)
			 to cycle;
		\draw (32.center) to (31.center);
		\draw (31.center) to (33.center);
		\draw (33.center) to (34.center);
		\draw [in=-210, out=0] (37.center) to (32.center);
		\draw (37.center) to (36.center);
		\draw [in=210, out=0] (35.center) to (34.center);
		\draw (34.center) to (32.center);
		\draw [in=0, out=225] (32.center) to (38.center);
		\draw (38.center) to (35.center);
		\draw [style=comodule-edge] (78.center)
			 to [in=-90, out=15] (63.center)
			 to [in=0, out=90] (61.center)
			 to [in=90, out=180] (64.center)
			 to [in=0, out=-90] (65.center)
			 to [in=0, out=180] (66.center);
		\draw [style=comodule-edge, in=90, out=180] (67.center) to (69.center);
		\draw [style=comodule-edge, in=90, out=-90] (69.center) to (70.center);
		\draw [style=comodule-edge, in=90, out=-150] (78.center) to (76.center);
		\draw [in=0, out=-225] (110.center) to (112.center);
		\draw [style=morphism-edge] (115.center)
			 to [in=-90, out=90] (130.center)
			 to [in=180, out=90] (128.center)
			 to [in=90, out=0] (129.center)
			 to [in=90, out=-90] (116.center);
		\draw [style=Front] (108.center)
			 to [in=0, out=-120] (114.center)
			 to (111.center)
			 to [in=-135, out=0] (110.center)
			 to cycle;
		\draw (108.center) to (107.center);
		\draw (107.center) to (109.center);
		\draw (109.center) to (110.center);
		\draw [in=-210, out=0] (113.center) to (108.center);
		\draw (113.center) to (112.center);
		\draw [in=225, out=0] (111.center) to (110.center);
		\draw (110.center) to (108.center);
		\draw [in=0, out=240] (108.center) to (114.center);
		\draw (114.center) to (111.center);
		\draw [style=comodule-edge] (132.center)
			 to [in=-90, out=90] (125.center)
			 to [in=0, out=90] (124.center)
			 to [in=90, out=180] (126.center)
			 to [in=90, out=-90] (127.center);
		\draw [style=comodule-edge, in=135, out=0] (67.center) to (66.center);
		\draw [style=morphism-edge, in=15, out=120] (185.center) to (189.center);
		\draw [in=0, out=135] (172.center) to (174.center);
		\draw [style=morphism-edge, in=90, out=-165] (189.center) to (190.center);
		\draw [style=morphism-edge, in=90, out=-90] (190.center) to (177.center);
		\draw [style=morphism-edge, in=90, out=-165] (194.center) to (178.center);
		\draw [style=Front] (170.center)
			 to [in=0, out=-135] (176.center)
			 to (173.center)
			 to [in=-150, out=0] (172.center)
			 to cycle;
		\draw (170.center) to (169.center);
		\draw (169.center) to (171.center);
		\draw (171.center) to (172.center);
		\draw [in=150, out=0] (175.center) to (170.center);
		\draw (175.center) to (174.center);
		\draw [in=-150, out=0] (173.center) to (172.center);
		\draw (172.center) to (170.center);
		\draw [in=0, out=-135] (170.center) to (176.center);
		\draw (176.center) to (173.center);
		\draw [style=comodule-edge, in=90, out=180] (186.center) to (187.center);
		\draw [style=comodule-edge, in=90, out=-90] (187.center) to (188.center);
		\draw [style=comodule-edge, in=90, out=-165] (194.center) to (192.center);
		\draw [style=comodule-edge, in=135, out=0] (186.center) to (185.center);
		\draw [style=comodule-edge, in=-45, out=45] (194.center) to (185.center);
	\end{pgfonlayer}
\end{tikzpicture}
  \end{equation*}
  An analogous computation for the counit shows that \(G \dashv V\) is a comodule adjunction.

  To see that these constructions are inverse to each other,
  first suppose that we have a comodule adjunction \((G, \delta^{(G)}) \adjoint (V, \delta^{(V)})\).
  By utilising \(\delta^{-(V)}\) as given in \cref{eq:inverse-of-coaction}, we obtain another coaction \(\lambda^{(G)}\) on \(G\),
  see \cref{eq:coaction-via-counit-unit-and-inverse}.
  A direct computation shows that \(\delta^{(G)} = \lambda^{(G)}\):
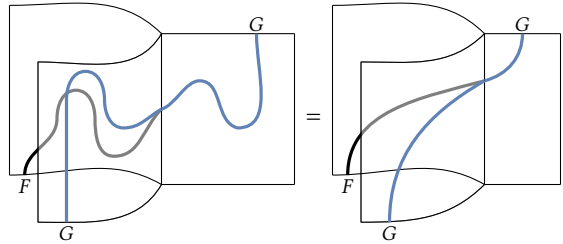
  \begin{equation*}
    \begin{tikzpicture}
	\begin{pgfonlayer}{nodelayer}
		\node [style=none] (0) at (0.5, -1.75) {};
		\node [style=none] (1) at (-3, -1.75) {};
		\node [style=none] (2) at (0.5, 2.25) {};
		\node [style=none] (3) at (-3, 2.25) {};
		\node [style=none] (4) at (-6.25, 1.5) {};
		\node [style=none] (5) at (-7, 3) {};
		\node [style=none] (6) at (-7, -1.5) {};
		\node [style=none] (7) at (-6.25, -2.75) {};
		\node [style=none] (8) at (-0.5, 2.25) {};
		\node [style=none] (9) at (-5, 1.25) {};
		\node [style=none] (10) at (-5.5, 0.5) {};
		\node [style=none] (11) at (-5.25, 0.75) {};
		\node [style=none] (12) at (-4.75, 0) {};
		\node [style=none] (13) at (-5.75, 0) {};
		\node [style=none] (14) at (-5.5, -2.75) {};
		\node [style=none] (15) at (-6.6, -1.5) {};
		\node [style=none] (16) at (-1, -0.25) {};
		\node [style=none] (17) at (-5.5, -3.05) {\tiny $G$};
		\node [style=none] (18) at (-6.6, -1.8) {\tiny $F$};
		\node [style=none] (19) at (-0.5, 2.5) {\tiny $G$};
		\node [style=none] (20) at (-2, 1) {};
		\node [style=none] (21) at (-3, 0.25) {};
		\node [style=none] (22) at (-4, -0.25) {};
		\node [style=none] (23) at (-4.5, 0.5) {};
		\node [style=none] (24) at (-4.25, -1) {};
		\node [style=none] (25) at (7.5, -1.75) {};
		\node [style=none] (26) at (5.5, -1.75) {};
		\node [style=none] (27) at (7.5, 2.25) {};
		\node [style=none] (28) at (5.5, 2.25) {};
		\node [style=none] (29) at (2.25, 1.5) {};
		\node [style=none] (30) at (1.5, 3) {};
		\node [style=none] (31) at (1.5, -1.5) {};
		\node [style=none] (32) at (2.25, -2.75) {};
		\node [style=none] (33) at (6.5, 2.25) {};
		\node [style=none] (34) at (3, -2.75) {};
		\node [style=none] (35) at (1.9, -1.5) {};
		\node [style=none] (36) at (3, -3.05) {\tiny $G$};
		\node [style=none] (37) at (1.9, -1.8) {\tiny $F$};
		\node [style=none] (38) at (6.5, 2.5) {\tiny $G$};
		\node [style=none] (39) at (1, 0) {\scriptsize $=$};
		\node [style=none] (40) at (5.5, 1) {};
	\end{pgfonlayer}
	\begin{pgfonlayer}{edgelayer}
		\draw [in=0, out=-225] (3.center) to (5.center);
		\draw [style=morphism-edge] (21.center)
			 to [in=0, out=-135] (24.center)
			 to [in=-90, out=180] (12.center)
			 to [in=0, out=90] (11.center)
			 to [in=90, out=180] (13.center)
			 to [in=90, out=-90] (15.center);
		\draw [style=Front] (7.center)
			 to (4.center)
			 to [in=-120, out=0] (3.center)
			 to (1.center)
			 to [in=0, out=-120] cycle;
		\draw (1.center) to (0.center);
		\draw (0.center) to (2.center);
		\draw (2.center) to (3.center);
		\draw [in=-210, out=0] (6.center) to (1.center);
		\draw (6.center) to (5.center);
		\draw (4.center) to (7.center);
		\draw (3.center) to (1.center);
		\draw [in=0, out=240] (3.center) to (4.center);
		\draw [in=240, out=0] (7.center) to (1.center);
		\draw [style=comodule-edge] (14.center)
			 to [in=-90, out=90] (10.center)
			 to [in=180, out=90] (9.center)
			 to [in=90, out=0] (23.center)
			 to [in=180, out=-90] (22.center)
			 to [in=-150, out=0] (21.center);
		\draw [style=comodule-edge] (21.center)
			 to [in=180, out=15, looseness=0.75] (20.center)
			 to [in=165, out=0] (16.center)
			 to [in=-90, out=0] (8.center);
		\draw [in=0, out=-225] (28.center) to (30.center);
		\draw [style=morphism-edge, in=90, out=-165] (40.center) to (35.center);
		\draw [style=Front] (32.center)
			 to (29.center)
			 to [in=-135, out=0] (28.center)
			 to (26.center)
			 to [in=0, out=-135] cycle;
		\draw (26.center) to (25.center);
		\draw (25.center) to (27.center);
		\draw (27.center) to (28.center);
		\draw [in=-210, out=0] (31.center) to (26.center);
		\draw (31.center) to (30.center);
		\draw (29.center) to (32.center);
		\draw (28.center) to (26.center);
		\draw [in=0, out=225] (28.center) to (29.center);
		\draw [in=225, out=0] (32.center) to (26.center);
		\draw [style=comodule-edge, in=15, out=-90] (33.center) to (40.center);
		\draw [style=comodule-edge, in=90, out=210] (40.center) to (34.center);
	\end{pgfonlayer}
\end{tikzpicture}
  \end{equation*}

  The converse direction is clear since the map which associates to any strong comodule structure on $V$ a comodule adjunction $G \lad V$ preserves the coaction of $V$.
\end{proof}

The philosophy that monads and adjunctions are two sides of the same coin extends to the comodule setting.
Suppose that we have a monoidal adjunction \stdadj{} and over it a comodule adjunction \(\adj{G}{V}{\cat{M}}{\cat{N}}\).
As stated in~\cite[Proposition~4.3.1]{Aguiar2012}, the bimonad \(B\eqdef UF\) admits a coaction on the monad \(K \eqdef VG\).
For any \(m \in \cat{M}\) and \(x \in \cat{C}\) it is given by
\[
  VG(m \ract x) \xrightarrow{V\delta^{(G)}_{m,x}} V(Gm \bract Fx) \xrightarrow{\delta^{(V)}_{Gm,Fx}} Km\ract Bx.
\]

Using the previous result, we clarify the structure of comparison functors associated to comodule adjunctions.
Its proof is analogous to~\cite[Theorem~2.6]{Bruguieres2007}.

\begin{proposition}\label{prop:comparison-functor-comodule-monad-comodule}
  Consider a comodule adjunction \(\adj{G}{V}{\cat{M}}{\cat{N}}\) over a monoidal adjunction \(\stdadj\),
  and denote the associated comodule monad and bimonad by \(K \eqdef VG \from \cat{M} \to \cat{M}\) and \(B \eqdef UF \from \cat{C} \to \cat{C}\), respectively.
  The comparison functor \(\Sigma^{K} \from \cat{N} \to \cat{M}^K\) is a strong comodule functor over \(\Sigma^{B} \from \cat{D} \to \cat{C}^B\).
  We furthermore have that
  \[
    U^K \Sigma^{K} = V \qquad\text{and}\qquad \Sigma^{K} G = F^K
  \]
  as comodule functors.
\end{proposition}
\begin{proof}
  For any \(n \in \cat{N}\) we have \(\Sigma^{K}n= (Vn, V\varepsilon_{n})\) and a direct computation shows that the coaction of \(V\) lifts to a coaction of\, \(\Sigma^{B}\) on \(\Sigma^{K}\).
  That is,
  \[
    U^{K} \delta^{(\Sigma^{K})}_{n, y} = \delta ^{(V)}_{n, y},
    \qquad \qquad \text{ for all } n \in \cat{N} \text{ and } y \in \cat{D}.
  \]
  Since \(U^K\) is conservative and \(\delta^{(V)}\) is an isomorphism by Theorem~\ref{thm:comodule-adjunctions-correspond-to-strong-comodule-functors}, \(\Sigma^{K}\) is a strong comodule functor.

  The \(U^B\) coaction on \(U^{K}\) is given by the identity natural transformation and we get \(U^{K}\Sigma^{K} = V\)  as comodule functors.
  Lastly, we compute for any \(x \in \cat{C}\) and
  \(m \in \cat{M}\),
  \[
    \delta^{(\Sigma^{K} G)}_{m, x}
    = \delta^{(U^K \Sigma^{K} G)}_{m, x}
    = \delta^{(VG)}_{m, x}
    = \delta^{(K)}_{m, x}
    = \delta^{(U^K F^K)}_{m, x}
    = \delta^{(F^K)}_{m, x}.
  \]
\end{proof}

Theorem~\ref{thm:comodule-adjunctions-correspond-to-strong-comodule-functors} yields furthermore a description of a comodule monad's coaction in terms of its \EM{} adjunction;
\ie, \cref{thm:intr:comodule-monad-reconstruction}, our desired analogue of \cref{thm:intr:bimonad-reconstruction}.

\begin{proof}[Proof of \cref{thm:intr:comodule-monad-reconstruction}]
  Suppose \(\cat{C}^{B}\) acts from the right on \(\cat{M}^{K}\) such that \(U^{K}\) is a strict comodule functor.
  Due to \cref{thm:comodule-adjunctions-correspond-to-strong-comodule-functors}, \(K = U^K F^K\) is a comodule monad via the coaction
  \begin{equation}
    \delta^{(K)} \eqdef \delta^{(U^K)} \circ U^K\delta^{(F^K)} = U^K\delta^{(F^K)}.
  \end{equation}

  Conversely, if \(K\) is a comodule monad, \(\cat{M}^{K}\) becomes a suitable right module over \(\cat{C}^{B}\) with the action as given in \cref{rem:comodule-monads-induce-module-categories}.

  Since the coaction on \(K\) and the action of\, \(\cat{C}^B\) on \(\cat{M}^K\) determine the coactions of \(F^{K}\) uniquely, the above constructions are inverse to each other by \cref{thm:comodule-adjunctions-correspond-to-strong-comodule-functors}.
\end{proof}

\printbibliography{}

\end{document}